\documentclass[noinfoline]{imsart}

\usepackage[a4paper,top=3cm, bottom=3cm, left=2.5cm, right=2.5cm]{geometry}
\usepackage{fancyhdr}

\RequirePackage{amsthm,amsmath,amsfonts,amssymb}

\allowdisplaybreaks
\RequirePackage[colorlinks,citecolor=blue,urlcolor=blue]{hyperref}
\usepackage[capitalize]{cleveref}
\usepackage{mathrsfs} 

\usepackage{graphicx}
\graphicspath{{Images/}}
\usepackage{float}
\usepackage[figuresright]{rotating} 
\usepackage{pdfpages}
\usepackage{wrapfig}

\usepackage{array}
\usepackage{booktabs} 
\usepackage{enumerate}
\usepackage{multicol}
\usepackage{tikz}
\usepackage{algorithm}
\usepackage[noend]{algpseudocode}

\makeatletter
\def\algbackskip{\hskip-\ALG@thistlm}
\makeatother

\renewcommand{\b}[1]{\mathbb{#1}} 
\newcommand{\m}[1]{\boldsymbol{#1}} 
\newcommand{\s}[1]{\mathscr{#1}} 
\renewcommand{\c}[1]{\mathcal{#1}} 

\newcommand{\tr}{\text{tr}} 
\newcommand{\Cov}{\text{Cov}} 
\newcommand{\Var}{\text{Var}} 

\newcommand*\Laplace{\mathop{}\!\mathbin\bigtriangleup}

\newcommand{\rd}{\mathrm{d}} 

\DeclareFontEncoding{FMS}{}{}
\DeclareFontSubstitution{FMS}{futm}{m}{n}
\DeclareFontEncoding{FMX}{}{}
\DeclareFontSubstitution{FMX}{futm}{m}{n}
\DeclareSymbolFont{fouriersymbols}{FMS}{futm}{m}{n}
\DeclareSymbolFont{fourierlargesymbols}{FMX}{futm}{m}{n}
\DeclareMathDelimiter{\VERT}{\mathord}{fouriersymbols}{152}{fourierlargesymbols}{147}

\newcommand{\vertiii}[1]{{\left\vert\kern-0.25ex\left\vert\kern-0.25ex\left\vert #1 
    \right\vert\kern-0.25ex\right\vert\kern-0.25ex\right\vert}} 

\makeatletter
\newcommand{\bigplus}{%
  \DOTSB\mathop{\mathpalette\mattos@bigplus\relax}\slimits@
}
\newcommand\mattos@bigplus[2]{%
  \vcenter{\hbox{%
    \sbox\z@{$#1\sum$}%
    \resizebox{!}{0.9\dimexpr\ht\z@+\dp\z@}{\raisebox{\depth}{$\m@th#1+$}}%
  }}%
  \vphantom{\sum}%
}
\makeatother

\DeclareMathOperator*{\odiag}{diag}
\DeclareMathOperator*{\spann}{span}
\DeclareMathOperator*{\argmin}{\arg\min}

\theoremstyle{plain}
\newtheorem{theorem}{Theorem}[section]
\newtheorem{lemma}[theorem]{Lemma}
\newtheorem{corollary}[theorem]{Corollary}
\newtheorem{proposition}[theorem]{Proposition}
\newtheorem{definition}[theorem]{Definition}
\newtheorem{remark}[theorem]{Remark}
\newtheorem{example}[theorem]{Example}

\begin{document}

\pagestyle{fancy}
\fancyhead{}
\renewcommand{\headrulewidth}{0pt} 
\fancyhead[CE]{H. Yun \& A. Caponera \& V.~M. Panaretos}
\fancyhead[CO]{Continuous Heterogeneity from Low-Dose Tomography}

\begin{frontmatter}
\title{{\large Low-Dose Tomography of Random Fields\\ and the Problem of Continuous Heterogeneity}}


\begin{aug}
\author[A]{\fnms{Ho} \snm{Yun}\ead[label=e1]{ho.yun@epfl.ch}}
\and
\author[B]{\fnms{Alessia} \snm{Caponera}\ead[label=e3]{acaponera@luiss.it}}
\and
\author[A]{\fnms{Victor M.} \snm{Panaretos}\ead[label=e2,mark]{victor.panaretos@epfl.ch}}

\thankstext{t1}{Research supported by a Swiss National Science Foundation grant to V.M. Panaretos.}
\address[A]{Ecole Polytechnique F\'ed\'erale de Lausanne, \printead{e1,e2}}
\address[B]{LUISS Guido Carli, \printead{e3}}
\end{aug}

\begin{abstract}
We consider the problem of nonparametric estimation of the conformational variability in a population of related structures, based on low-dose tomography of a random sample of representative individuals. In this context, each individual represents a random perturbation of a common template and is imaged noisily and discretely at but a few projection angles. Such problems arise in the cryo Electron Microscopy of structurally heterogeneous biological macromolecules. We model the population as a random field, whose mean captures the typical structure, and whose covariance reflects the heterogeneity. We show that consistent estimation is achievable with as few as two projections per individual, and derive uniform convergence rates reflecting how the various parameters of the problem affect statistical efficiency, and their trade-offs. Our analysis formulates the domain of the forward operator to be a reproducing kernel Hilbert space, where we establish representer and Mercer theorems tailored to question at hand. This allows us to exploit pooling estimation strategies central to functional data analysis,  illustrating their versatility in a novel context. We provide an efficient computational implementation using tensorized Krylov methods and demonstrate the performance of our methodology by way of simulation. 
\end{abstract}


\end{frontmatter}

\maketitle


\section{Introduction}

In the classical problem of tomography \cite{natterer2001mathematics, helgason1965radon, ludwig1966radon, christ1984estimates}, one is interested in reconstructing a fixed three-dimensional structure, given two-dimensional projections thereof at several different orientations. The fixed structure is represented by a probability density function in $\b{R}^3$, and its projections are taken to be two-dimensional marginal distributions, obtained by integrating along the viewing direction. Statistical versions of this problem arise when randomness is introduced ``into the equation'' \cite{yun2025computerized, vardi1985statistical, kak2001principles}. Classically, this happens through measurement error or random sampling: one observes pixelated and noise-corrupted projections at a finite number of randomly sampled orientations.

Yet, randomness can also be inherent at the level of the three-dimensional structure itself: when the structure is not fixed, but can manifest in various different but related configurations \cite{anden2018structural, scheres2012bayesian, sigworth2007cryo}. In this case, the target is not a single fixed structure, but rather a population of related structures that can be conceptualized as perturbations of some common template. The estimand, therefore, is not a three-dimensional density but rather the law of a three-dimensional random field, making the problem intrinsically statistical at the signal level \cite{schwander2014conformations, frank2016continuous}. Correspondingly, the tomographic data are characterized by heterogeneity, as amongst the two dimensional projections appear the marginals of various \emph{different} instantiations of the three dimensional structure -- observed still at random orientations with noise contamination.

This heterogeneous tomography problem arises perhaps most prominently in single-particle cryogenic Electron Microscopy (cryo-EM), a Nobel-prize winning imaging modality for determining the three-dimensional electron density structure of biological macromolecules \cite{frank2006three, kuhlbrandt2014resolution, nogales2016development, cheng2015primer, bartesaghi20152}. In cryo-EM, thousands of macromolecules of identical composition are imaged under a microscope after being vitrified at randomly assumed orientations \cite{dubochet1981vitrification, singer2018mathematics}. If an identical composition always yielded an identical potential density, this would reduce to a standard ``homogeneous'' tomographic problem. However, macromolecules typically exist in varying configurations due to their flexibility and/or their biological functioning. And, increasingly, these configurations are viewed not as a few discrete states but as a continuum characterized by an energy landscape \cite{frank2018new, frank2016continuous}. The challenge is thus to retrieve not a single structure, but rather the probability law of a continuum random field -- the law encoding the said conformational landscape. For biophysicists, this challenge is also a great opportunity, as the ability to experimentally determine the conformational distribution of a biological macromolecule can lead to more systematic functional interpretations.

A concurrent challenge in heterogeneous cryo-EM is the \emph{low-dose} problem \cite{glaeser2021single, singer2018mathematics}: if the conformational landscape is continuous, each manifestation essentially occurs only once in the sample, and so we can have but a few vantage points thereof, limited to the number of tilts taken for imaging. Ideally, one would collect many tilts per particle to gather more data, but exposing biological samples to a high electron dose destroys their structure \cite{hayward1979radiation, grant2015measuring}. One is thus often limited to a low-dose setting, where each realization of the random field is imaged with a minimally sufficient number of tilt angles \cite{baxter2009determination, panaretos2010second}.

Motivated by cryo-EM, this paper aims to develop the statistical foundations related to the problem of (continuously) heterogeneous tomography in a possibly low-dose setting. This includes theory and rates, as well as methodology and implementation. To do so, we formulate the heterogeneity problem as one of inferring the modes of variation of the three-dimensional random object (random field) around a template structure, as one would do in a functional principal component analysis (PCA).
This effectively reduces the problem to that of non-parametrically estimating the random field's mean structure (the template) and the covariance kernel (which encodes the modes of variation around the template).
Though functional PCA is arguably a coarse descriptor of heterogeneity,
it provides a framework that is amenable to theoretical analysis, while retaining enough of the key elements of reality to draw useful, if approximate, conclusions. Indeed, this framework allows us to fruitfully interface with the field of Functional Data Analysis (FDA) \cite{hsing2015theoretical, ramsay2002applied}, where the non-parametric estimation of means and covariances of random processes from sparse and noisy data is a central and well-understood problem \cite{panaretos2012nonparametric, gilliam1993rate, caponera2022functional,caponera2022rate,caponera2025two, cai2011optimal} -- albeit not yet in the complex setting of tomography. It also allows us to develop a rigorous statistical analysis of the problem where the unknowns are in the continuum, but the data are recorded discretely and with measurement error. Our analysis shows the interplay between smoothness conditions and experimental parameters, such as the number of instantiations imaged, and the number of vantage points (tilts) per instantiation, including related phase transitions. Interestingly, it is shown that consistent recovery of both mean (template) and covariance (heterogeneity) is feasible with as few as two vantage points per instantiation, provided the number of instantiations diverges. Parallel to our theoretical results, we also provide complete methodology, including a computational implementation. Our methodology leverages reproducing kernel Hilbert spaces, by establishing bespoke Mercer and representer theorems. Our implementation uses tensorized Krylov methods \cite{yun2025fast} to yield computationally 
efficient approach, despite the dimensionality of the estimands.

\section{Background and Contributions}
Given an integrable $f: \b{B}^{d} \to \b{R}$ on the unit ball, the X-ray transform $f \mapsto \{\s{P}_{R} f : R \in SO(d)\}$ generates its collection of \emph{marginals} along the axis $\m{r}:=R^{\top} e_{d}$ of rotation $R$ from the special orthogonal group $SO(d)$ \cite{christ1984estimates, smith1977practical}.
More generally, the X-ray transform is a special case of the $k$-plane transform, which integrates $f$ over $d_{r}$-dimensional planes \cite{solmon1976x, solmon1979note, sharafutdinov2016reshetnyak}:
\begin{align}\label{eq:intro:xray}
     \s{P}_{R} f(\m{x})\equiv \s{P} f(R, \m{x}) :=
    \int_{(1-\|\m{x}\|^{2})^{1/2} \cdot \b{B}^{d_{r}}} f (R^{-1} [\m{x}:z]) \rd z, \quad \m{x} \in \b{B}^{d-d_{r}}.
\end{align}
where $[\m{x}:z]$  represents the vector obtained by concatenating $\m{x}\in\mathbb{R}^{d-d_{r}}$ and $z \in \b{R}^{d_{r}}$. Different values of $d_{r}$ correspond to different imaging modalities, such as the X-ray transform ($d_{r} = 1$) for CT and cryo-EM, the Radon transform ($d_{r} = d-1$) \cite{natterer2001mathematics}, or nuclear magnetic resonance imaging ($d_{r} = d-2$) \cite{parhi2024distributional}.

The idealized \emph{homogeneous} tomographic reconstruction  consists of recovering $f$ from the complete collection of its projections, $R \mapsto \s{P}_{R} f$, known as the \emph{sinogram} \cite{natterer2001mathematics, helgason1965radon, ludwig1966radon, christ1984estimates}. In practice, however, one only has access to a finite number of ``sinogram values'' $\{\s{P}_{\m{R}_{j}} f\}_{j=1}^{r}$, each evaluated at discrete locations $\m{X}_{jk} \in \b{B}^{d-1}$ and possibly contaminated by noise:
\begin{equation*}
    Z_{jk}= \s{P}_{\m{R}_{j}}f(\m{X}_{jk})+\varepsilon_{jk}.
\end{equation*}
In statistical settings, the rotations $\m{R}_{j}$, locations $\m{X}_{jk}$, and measurement errors $\varepsilon_{jk}$ will be random and completely independent. We refer to this setting as \emph{homogeneous} tomography, because it is always the same object $f$ that is imaged repeatedly, just from different angles. Incidentally, the term \emph{sinogram} arises from the sinusoidal patterns that off-center structures create in the projection data, as seen in \cref{fig:sino}. Accordingly, sinograms must satisfy certain moment constraints, known as the Helgason-Ludwig consistency conditions \cite{ludwig1966radon, helgason1965radon}. 

\begin{figure}[h]
\centering
\includegraphics[width=\textwidth, height=4cm]{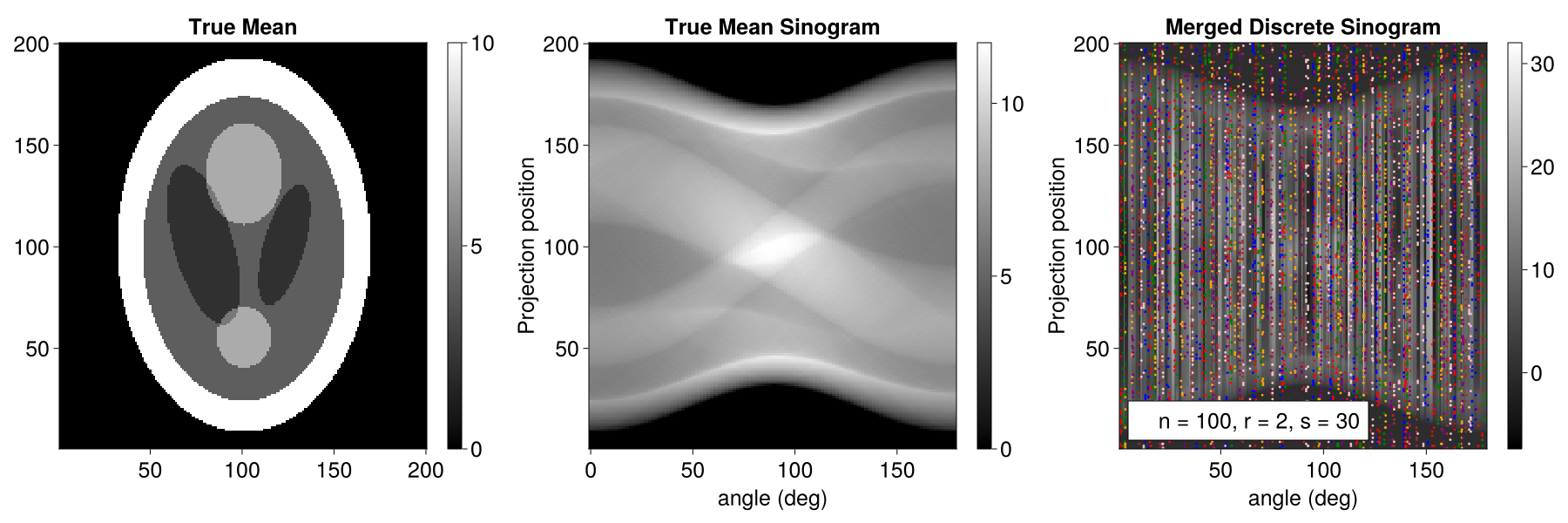}
\caption{\textbf{Left}: The true mean function. \textbf{Center}: Its ideal sinogram over 180$^{\circ}$. \textbf{Right}: Merged and noisy discrete sinogram from a simulation with $n=100$ random functions, each observed at only $r=2$ random angles and $s=30$ locations, where different color corresponds to the observations from a different function.}
\label{fig:sino}
\end{figure}

To formulate the heterogeneous problem \cite{anden2018structural, scheres2012bayesian, katsevich2015covariance, sigworth2007cryo}, we model the object of interest as a second-order random field $\c{Y}$ on the unit ball $\b{B}^{d}$. The law of $\c{Y}$ captures the continuous structural heterogeneity of the object, and we aim to estimate its mean $\m{\mu}$ and covariance $\m{\Sigma}$: 
\begin{align}\label{eq:intro:moments}
    \m{\mu} := \b{E} \c{Y}, \quad
    \m{\Gamma} := \b{E} [\c{Y} \otimes \c{Y}], \quad
    \m{\Sigma} := \b{E} [(\c{Y}-\m{\mu}) \otimes (\c{Y}-\m{\mu})] = \m{\Gamma} - \m{\mu} \otimes \m{\mu},
\end{align}
in order to perform PCA that reveals the main modes of conformational variation around its mean. We assume access only to low-dose tomographic measurements on i.i.d. replicates:
\begin{equation*}
    Z_{ijk} = \s{P}_{\m{R}_{ij}} \c{Y}_{i} (\m{X}_{ijk}) + \varepsilon_{ijk}, \quad 1 \le i \le n, \, 1 \le j \le r_{i}, \, 1 \le k \le s_{ij},
\end{equation*}
where:
\begin{itemize}
\item each replication $\c{Y}_i$ is imaged at replication-specific orientations or \emph{tilts} $\m{R}_{ij}$. 
\item each projection image $\s{P}_{\m{R}_{ij}} \c{Y}_{i}$ is measured at its own evaluation points $\m{X}_{ijk}$.
\item the  errors $\varepsilon_{ijk}$ are independent, zero-mean, with variance  $\sigma^2_{r_i}$  that can depend on $r_i$. 
\item the viewing angles $\m{R}_{ij}$ and evaluation points $\m{X}_{ijk}$ are drawn at random, independently of each other and of the measurement errors.
\end{itemize}
To incorporate the \emph{low-dose} aspect, we allow the noise variance $\sigma^2_{r_i}$ to depend on the number of tilts $r_i$. Concretely, we think of $r\mapsto\sigma^2_{r}$ as an increasing function of $r$ (e.g. $\sigma^2_r \propto r$), reflecting the fact that that there is a \emph{radiation dose restriction}: the total radiation per individual should be limited, so that there is a trade-off between the number of viewing angles and the noise contamination of the corresponding images. In other words, an individual's radiation dose is pre-set, and this fixed budget must be distributed among the $r_i$ viewing angles; more angles result in a lower dose and higher noise for each projection image. 

For estimation, we assume that the fields $\c{Y}_{i}$ are random elements in a Reproducing Kernel Hilbert Space (RKHS) $\b{H}(K)$ \cite{aronszajn1950theory, paulsen2016introduction}. Though there are other possible smoothness conditions \cite{stefanov2022radon, stefanov2020semiclassical, monard2023sampling, katsevich2024analysis, natterer1980sobolev, sharafutdinov2016reshetnyak, sharafutdinov2012integral} to mitigate the inherent ill-posedness, the RKHS setting is particularly well-suited to our context: by deriving bespoke Mercer and representer theorems, we are able to reduce the problem to a ridge regression problem via feature maps $ \{\varphi(\m{R}_{ij}, \m{X}_{ijk}): 1 \le i \le n, 1 \le j \le r_{i}, 1 \le k \le s_{ij} \}$ adapted to the observation scheme, leading to less pronounced artifacts \cite{yun2025computerized}. Crucially, this allows the recovery of both the mean and covariance, as the covariance problem also admits a representer theorem using a dictionary of tensor products of the feature maps used for the mean. This opens the way for using pooling techniques common in FDA \cite{yao2005functional, staniswalis1998nonparametric}. One might suspect that the Gram matrices required for implementing such a strategy are prohibitively large to store or manipulate directly. This is true, and we employ a tensorized Krylov method that bypasses their explicit construction \cite{yun2025fast}. After solving for the coefficient matrix, a generalized eigendecomposition of this matrix yields the sought functional PCA 
of $\hat{\m{\Sigma}}_{\nu, \eta} = \hat{\m{\Gamma}}_{\eta} - \hat{\m{\mu}}_{\nu} \otimes \hat{\m{\mu}}_{\nu}$, which represents the main modes of continuous heterogeneity. 

We then derive uniform convergence rates for both the mean and covariance estimators under a source condition on the distribution of $\c{Y}_{i}$'s. Our results generalize known convergence rates from the standard FDA setting (where $\s{P}_{R} = I$)  \cite{cai2010nonparametric, cai2011optimal, caponera2022functional, caponera2025two} to this considerably more complex setting. Furthermore, our findings reveal several crucial phase transitions that are governed by the number of viewing angles and the low-dose imaging constraint: 
\begin{itemize}
\item We establish the minimum tilts required for consistent estimation as the number of particles $n$ grows. For the mean, as few as one tilt per particle is sufficient. As for the covariance, consistent estimation is possible with just two tilts per particle. Then, the convergence rate depends on the number of tilts, not on the locational sampling frequency.

\item We identify a phase transition based on how the number of tilts compares to the number of particles $n$. In a dense-tilt regime, where the number of tilts is sufficiently large, the error is limited by $n$ and we achieve a parametric rate of convergence.

\item We formally characterize the trade-off imposed by a fixed radiation budget. The convergence rate slows as the noise $\sigma_{r}^{2}$ increases faster with the number of tilts $r$. For mean estimation, the rate actually degrades as more tilts are added, if the noise variance grows faster than linearly ($\sigma_{r}^{2} \propto r$). For covariance estimation, the boundary is much stricter: the rate degrades if noise variance grows faster than the square root of $r$ ($\sigma_{r}^{2} \propto \sqrt{r}$).
\end{itemize}

The remainder of this paper is organized as follows. \cref{sec:related} reviews related work on the heterogeneous tomography problem, and \cref{sec:notation} introduces key notation. We develop our main theoretical contributions in \cref{sec:setup,sec:Mercer:inv,sec:asym}, where we establish our RKHS framework, adapt Mercer's theory for the inverse problem setting, and derive the asymptotic properties of our estimators. We note that our proofs rely on the abstract property of orientational duality, making our results applicable to general $k$-plane transforms (Supplement \ref{sec:sino:op}).
\cref{sec:cryo} discusses how to apply our theory to the cryo-EM problem with unknown orientations, and \cref{sec:simul} presents simulation studies.

\section{Related Work}\label{sec:related}
The numerous methods developed to address heterogeneity analysis in single-particle cryo-EM can be broadly classified into two paradigms. A common prerequisite for these methods is the prior estimation of particle viewing angles, typically achieved through 3D alignment based on the Fourier slice theorem \cite{penczek2006estimation, natterer2001mathematics}. 

The first paradigm, \emph{discrete} heterogeneity, assumes that biomolecules exist in a finite number of stable states and frames the problem as a 3D classification task, where the primary goal is to sort particle images into distinct classes and reconstruct a high-resolution 3D map for each. Early methods like multi-body refinement treated molecules as collections of rigid parts \cite{heymann2004molecular}. This evolved into the unsupervised classification methods, including two highly successful implementations \texttt{RELION} \cite{scheres2012relion, scheres2012bayesian} and \texttt{cryoSPARC} \cite{punjani2017cryosparc}, both of which leverage regularized maximum-likelihood approaches. For numerical optimization, they use expectation-maximization or stochastic gradient descent, respectively. However, these methods tend to struggle when the number of conformations is large or the variations are subtle.

To overcome these limitations, methods that model \emph{continuous} heterogeneity have emerged as an alternative, see \cite{anden2018structural, katsevich2015covariance} and the references therein. These techniques aim to characterize the primary modes of molecular transition within a lower-dimensional space. The main approaches include:
\begin{itemize}
\item \textbf{Linear Subspace Methods}: This approach \cite{punjani20213d,penczek2006estimation,penczek2011identifying,tagare2015directly,katsevich2015covariance,anden2018structural, gilles2025regularized} assumes that
even complex molecular motions can be approximated by a few simple, linear components. The primary approach involves estimating the full mean and covariance of the molecular ensemble, then performing PCA, a methodology that shares a foundation with our paper.

\item \textbf{Manifold Learning Methods}: Rather than assuming motions are linear, these methods posit they follow curved paths or surfaces. They embed the high-dimensional cryo-EM data into a latent space, using kernel functions \cite{frank2016continuous, tang2007eman2} or graph Laplacian \cite{moscovich2020cryo}, and performing PCA thereof for dimension reduction.

\item \textbf{Physics-Inspired Methods}: 
This line of work models structural changes using more direct physical analogies, including methods that fit rigid bodies into a flexible assembly \cite{nakane2018characterisation}, calculate smooth warping fields to morph a consensus structure \cite{punjani20233dflex, schwab2024dynamight}, or incorporate constraints from molecular dynamics \cite{herreros2023estimating}.

\item \textbf{Deep Learning Methods}: These methods leverage neural networks to learn complex, non-linear latent representations of the conformational landscape from the data \cite{zhong2021drng, chen2021deep}.
\end{itemize}
Common to these methods, however, is the \emph{discretization} of the structure into a cubic voxel grid, which is used to estimate the covariance of the 3D volumes that often requires data downsampling to remain numerically tractable \cite{liu1995estimation1, liu1995estimation2}. Furthermore, most of these methods  depend on the Filtered Back-Projection (FBP), and with the exception of PCA-based approaches, often lack formal performance guarantees but rely on numerical evidence \cite{seitz2022recovery}.

Our work, therefore, takes a conceptual step beyond the aforementioned literature 
-- we perceive the underlying biomolecular structure as a continuous random function, approaching the problem from a \emph{functional} viewpoint. Instead of relying on a voxel grid, we encode the smoothness of the conformational space using a RKHS, just as the Nyquist-Shannon sampling frequency \cite{shannon2006communication} dictates the resolution limit of reconstruction. This perspective leads to a reconstruction scheme based on kernel regression that is fundamentally different from FBP and is, crucially, dimension-agnostic -- it does not suffer from FBP's inconsistent behavior in odd versus even dimensions due to the involvement of the Hilbert transform, a global operation \cite{natterer2001mathematics}. By abstracting away from the fine experimental details (e.g., contrast transfer function or translation alignment), our simplified model, much like foundational \emph{stylized models} in statistics like the spiked covariance or stochastic block models, is intended to provide an analytically tractable setting 
providing theoretical insight, despite being a coarse approximation.

\section{Notation}\label{sec:notation}
Throughout the paper, we adhere to the following notation: 
\begin{itemize}
    \item $\b{S}^{d-1}=\{\m{z} \in \b{R}^{d}: \|\m{z}\| = 1\}$ is the unit sphere in $\b{R}^{d}$.
    \begin{equation*}
        |\b{S}^{d-1}|=\frac{2 \pi^{d/2}}{\Gamma(d/2)}, \quad |\b{S}^{0}|=2, \, |\b{S}^{1}|=2 \pi, \, |\b{S}^{2}|=4 \pi.
    \end{equation*}
    \item $\b{B}^{d}=\{\m{z} \in \b{R}^{d}: \|\m{z}\| \le 1\}$ is the closed unit ball in $\b{R}^{d}$. 
    \item For a given linear operator $\s{P}$, the null and range spaces are denoted by $\c{N}(\s{P})$ and $\c{R}(\s{P})$, respectively. $^{\dag}$ symbol denotes the Moore-Penrose generalized inverse.
\end{itemize}

We frequently use tensor products and recall key facts. To make the abstract tensor algebra concrete, it is customary in the FDA literature \cite{hsing2015theoretical} to define the elementary tensor $f_{1} \otimes f_{2}$, for $f_{1} \in \c{H}_{1}$, $f_{2} \in \c{H}_{2}$ in separable $\b{R}$-Hilbert spaces, as
\begin{equation*}
    f_{1} \otimes f_{2} \, : \, f '_{2} \in \c{H}_{2} \quad \mapsto \quad \langle f_{2}, f '_{2} \rangle_{\c{H}_{2}} \, f_{1} \in \c{H}_{1}.
\end{equation*}
In our covariance analysis, we encounter fourth-order tensors $f_{1} \otimes f_{2} \otimes f_{3} \otimes f_{4}$, and the grouping of these terms matters:
\begin{itemize}
    \item Thinking of operators acting on operators, the natural grouping is $(f_{1} \otimes f_{2}) \otimes (f_{3} \otimes f_{4})$.
    \item For spectral analysis, it is more natural to group the vectors and covectors separately \cite{lee2003smooth, hall2013quantum}. This corresponds to a \emph{reshuffling} of the tensor into $(f_{1} \otimes f_{3})$ and $(f_{2} \otimes f_{4})$. To distinguish this second case, we use the special notation $(f_{1} \otimes f_{3}) \otimes_{2} (f_{2} \otimes f_{4})$. 
\end{itemize}

The tensor product space $\c{H}_{1} \otimes \c{H}_{2}$ is formally the metric completion of the vector space spanned by these elementary tensors, and the universal property \cite{kadison1986fundamentals} reveals that $\c{H}_{1} \otimes \c{H}_{2}$ is isometric and linearly isomorphic to the space of Hilbert-Schmidt (HS) operators $\c{B}_{2}(\c{H}_{2}, \c{H}_{1})$. This space is itself a separable Hilbert space with inner product
\begin{equation*}
    \langle \s{K}_{1}, \s{K}_{2} \rangle_{HS} := \sum_{l=1}^{\infty} \langle \s{K}_{1} e_{l}, \s{K}_{2} e_{l} \rangle_{\c{H}_{2}}, 
\end{equation*}
where $\{e_{l}\}_{l=1}^{\infty}$ is any complete orthonormal system (CONS) for $\c{H}_{1}$. Note that for elementary tensors, it holds that $\langle f_{1} \otimes f_{2}, f '_{1} \otimes f '_{2} \rangle_{HS} = \langle f_{1}, f '_{1} \rangle_{\c{H}_{1}} \langle f_{2}, f '_{2} \rangle_{\c{H}_{2}}$.

We denote the Banach space of bounded operators by $\c{B}_{\infty}(\c{H})$, equipped with the operator norm $\VERT \cdot \VERT_{\infty}$. The Banach space of compact operators on $\c{H}$ is denoted by $(\c{B}_{0}(\c{H}), \VERT \cdot \VERT_{\infty})$. For $1 \le p < \infty$, the $p$-Schatten class on $\c{H}$ is a collection of operators as follows:
\begin{equation*}
    \c{B}_{p}(\c{H}) := \left\{ \s{K} \in \c{B}_{0}(\c{H}) : \VERT \s{K} \VERT_{p} := \sum_{l=1}^{\infty} \langle e_{l}, (\s{K}^{*} \s{K})^{p/2} e_{l} \rangle_{\c{H}} < +\infty \right\}, 
\end{equation*}
for any CONS $\{e_{l}\}_{l=1}^{\infty}$ for $\c{H}$. It can be shown that $(\c{B}_{p}(\c{H}), \VERT \cdot \VERT_{p})$ is a Banach space, and the $p$-Schatten norm $\VERT \cdot \VERT_{p}$ is called the trace and the HS norm when $p=1$ and $p=2$, respectively. 

\section{Modeling}\label{sec:setup}
If the support of the $d$-dimensional function is compact, we can simply assume it to be the unit ball $\b{B}^{d}$. 
When the orientation $R$ is also considered as an input in \eqref{eq:intro:xray}, we use the notation $\s{P} f(R, \m{x}) = \s{P}_{R} f(\m{x})$ interchangeably,  and $\s{P}$ is referred to as the sinogram operator in this context.
Although we focus on the X-ray transform in the main body, the theoretical results presented in this paper are not limited to them. We treat the problem more abstractly, as detailed in Supplement \ref{sec:sino:op}. For instance, the orientation group $\b{G} = SO(d)$ can be viewed as a compact Lie group, equipped with the normalized Haar measure $\rd R$ \cite{hall2013lie, faraut2008analysis}. The Lie group acts from the left on a compact domain $\Omega^{d} = \b{B}^{d}$: $\rho(R): \Omega^{d} \rightarrow \Omega^{d}, \m{z} \mapsto R \m{z}$. Moving forward, the domain for the resulting projection images is $\Omega^{d-d_{r}} = \b{B}^{d-d_{r}}$, with $d_{r} = 1$, and we use the same notation to denote the left group action on functions. To be specific, for $f: \Omega^{d} \to \b{R}$ and $g : \b{G} \times \Omega^{d-d_{r}} \to \b{R}$, we have:
\begin{alignat}{2}
    &\rho(R) f (\m{z}) := f(R^{-1} \m{z})&&, \quad \m{z} \in \Omega^{d}, \label{eq:left:reg:def1} \\
    &\rho(R_{1}) g (R_{2}, \m{x}) := g(R_{1} R_{2}, \m{x})&&, \quad R_{1}, R_{2} \in \b{G}, \, \m{x} \in \Omega^{d-d_{r}}. \label{eq:left:reg:def2}
\end{alignat}
Rather than the specific integral formula in \eqref{eq:intro:xray}, the key property we distill from the X-ray transform, or the $k$-plane transform more generally, is \emph{orientational duality}: $\s{P}_{R} = \s{P}_{I} \circ \rho(R)$ for any $R \in \b{G}$, which has a clear physical interpretation. In CT imaging, for instance, scanning an object at $30^{\circ}$ is equivalent to rotating the object by $-30^{\circ}$ and then scanning it at $0^{\circ}$. Because our proofs rely only on this property and some mild smoothness conditions (see \cref{def:sino:op}), our findings naturally extend to any operator that shares this structure.

Consider weight functions $W_{0, d}:\Omega^{d} \rightarrow [0, +\infty)$ and $W_{d_{r}, d}:\Omega^{d-d_{r}} \rightarrow [0, +\infty)$ for input and output functions, which have finite total weight:
\begin{equation}\label{eq:normal:const}
\omega_{0, d} := \int_{\Omega^{d}} W_{0, d}(\m{z}) \rd \m{z} <\infty, \quad \omega_{d_{r}, d} := \int_{\Omega^{d-d_{r}}} W_{d_{r}, d}(\m{x}) \rd \m{x} <\infty,
\end{equation}
where $\rd \m{z}$ and $\rd \m{x}$ represent the standard Lebesgue measures. We assume that the input weight function $W_{0, d}$ is $\b{G}$-invariant, i.e., $W_{0, d}(\m{z}) = W_{0, d}(R \m{z})$ for any $R \in \b{G}$. Consequently, the input space $\c{L}_{2} (\Omega^{d}, W_{0, d})$  is a well-defined Hilbert space containing $\c{C}(\Omega^{d})$, with its inner product given by:
\begin{equation*}
    \langle f_{1}, f_{2} \rangle_{W_{0, d}} = \int_{\Omega^{d}} f_{1}(\m{z}) f_{2}(\m{z}) W_{0, d}(\m{z}) \rd \m{z}.
\end{equation*}
The same holds for the output space $\c{L}_{2} (\Omega^{d-d_{r}}, W_{d_{r}, d})$:
\begin{equation*}
    \langle g_{1}, g_{2} \rangle_{W_{d_{r}, d}} = \int_{\Omega^{d-d_{r}}} g_{1}(\m{x}) g_{2}(\m{x}) W_{d_{r}, d}(\m{x}) \rd \m{x}.
\end{equation*}
Similarly, $\c{L}_{2} (\b{G} \times \Omega^{d-d_{r}}, W_{d_{r}, d})$ forms a Hilbert space with the inner product defined as:
\begin{align*}
    \langle g_{1}, g_{2} \rangle_{W_{d_{r}, d}} 
    := \int_{\b{G}} \int_{\Omega^{d-d_{r}}} g_{1}(R, \m{x}) g_{2}(R, \m{x}) W_{d_{r}, d}(\m{x}) \rd \m{x} \rd R,
\end{align*}

\subsection{RKHS Framework}
To assess reconstruction reliability, it is essential to equip the function space with a smoothness structure. In the literature on the X-ray and Radon transforms, Sobolev spaces \cite{natterer1980sobolev} or their variants \cite{natterer1980sobolev, sharafutdinov2016reshetnyak, sharafutdinov2012integral} have been explored as alternatives to the weighted $\c{L}_{2}$ domain. Building upon this, \cite{yun2025computerized} generalized this concept by framing the function space as an RKHS. RKHSs encode desired levels of smoothness through its reproducing kernel, which is often constructed via Green's functions \cite{hsing2015theoretical, paulsen2016introduction}. The RKHS framework is particularly appealing due to its strong theoretical foundation and the representer theorem, which underpins its widespread use in statistics and machine learning \cite{kimeldorf1970correspondence, kimeldorf1971some, unser2021unifying}.

\begin{definition}
Let $\b{H}$ be a Hilbert space of real-valued functions defined on a set $\Omega$. A bivariate function $K:\Omega \times \Omega \rightarrow \b{R}$ is called a reproducing kernel for $\b{H}$ if
\begin{enumerate}
    \item For any $\m{z} \in \Omega$, a generator $k_{\m{z}}(\cdot):= K(\cdot, \m{z})$ at $\m{z} \in \Omega$ belongs to $\b{H}$.
    \item For any $f \in \b{H}$, the point evaluation at $\m{z} \in \Omega$ is given by $f(\m{z})=\langle f, k_{\m{z}} \rangle_{\b{H}}$.
\end{enumerate}
A Hilbert space equipped with a reproducing kernel is called an RKHS. 
\end{definition}

The Moore-Aronszajn theorem \cite{aronszajn1950theory} states that any reproducing kernel $K$ is symmetric and semi-positive definite (s.p.d.), and conversely, any such a kernel uniquely defines an RKHS, denoted $\b{H} = \b{H}(K)$ \cite{paulsen2016introduction, hsing2015theoretical}. When the kernel is continuous, it is referred to as a Mercer kernel. In our work, it is particularly relevant to
consider a $\b{G}$-invariant Mercer kernel, meaning  $K(\m{z}_{1}, \m{z}_{2}) = K(R \m{z}_{1}, R \m{z}_{2})$ for all $R \in \b{G}$. We call this type of kernel as the sinogram kernel, see \cref{def:sino:ker} for a general definition. For the X-ray transform, the evaluation functional of a projection image at any $\m{x} \in \Omega^{d-d_{r}}$ is bounded by some constant $B > 0$:
\begin{equation*}
    |\s{P}_{R} f ( \m{x})| \le B \|f\|_{\b{H}}, \quad R \in \b{G}, \, f \in \b{H}(K),
\end{equation*}
see \cref{rmk:xray:arbi:kern}. Consequently, the Riesz representation theorem guarantees that there exists a unique feature map $\varphi(R, \m{x}) \in \b{H}$ with $\|\varphi(R, \m{x})\|_{\b{H}} \le B$, yielding the reproducing property:
\begin{equation}\label{eq:repro:prop}
    \langle \varphi(R, \m{x}), f \rangle_{\b{H}} = \s{P} f (R, \m{x}), \quad f \in \b{H}(K).
\end{equation}
The significant advantage of this framework is that we can explicitly construct the feature map using  the reproducing property \eqref{eq:repro:prop}: $\varphi(R, \m{x})(\m{z}) = \langle \varphi(R, \m{x}), k_{\m{z}} \rangle_{\b{H}} = \s{P} k_{\m{z}} (R, \m{x})$, which only requires the forward operator 
$\s{P}$, not its inverse or adjoint.

\subsection{Estimation}\label{sec:scheme}

Let $\c{Y}$ be a second-order random function taking values in the RKHS $\b{H} = \b{H}(K)$, that is $\b{E} \|\c{Y}\|_{\b{H}}^{2} < \infty$. Let $\m{\mu}, \m{\Gamma}, \m{\Sigma}$ denote the mean, second moment, and covariance of the random element $\c{Y}$, respectively, as defined in \eqref{eq:intro:moments}.
Let $\b{P}_{\m{R}}$ and $\b{P}_{\m{X}}$ denote the probability measures on $\b{G}$ and $\Omega^{d-d_{r}}$, respectively, given by
\begin{equation*}
    \rd \b{P}_{\m{R}} (R) = \rd R, \quad \rd \b{P}_{\m{X}} (\m{x}) = \omega_{d_{r}, d}^{-1} W_{d_{r}, d}(\m{x}) \rd \m{x}.
\end{equation*}

We consider $n$ i.i.d. copies $\{\c{Y}_{i}: i=1, \dots, n \}$ of $\c{Y}$. For each sample $\c{Y}_{i}$, we generate projections $\s{P}_{\m{R}_{ij}} \c{Y}_{i}$ at random orientations  $\m{R}_{ij}$. These projections are then observed at random locations $\m{X}_{ijk}$ under the noise $\varepsilon_{ijk}$:
\begin{equation}\label{eq:obs:scheme}
    Z_{ijk} = \s{P} \c{Y}_{i} (\m{R}_{ij}, \m{X}_{ijk}) + \varepsilon_{ijk}
    =\langle \c{Y}_{i}, \varphi(\m{R}_{ij}, \m{X}_{ijk}) \rangle_{\b{H}} + \varepsilon_{ijk},
\end{equation}
where $\m{R}_{ij} \stackrel{iid}{\sim} \b{P}_{\m{R}}$,
$\m{X}_{ijk} \stackrel{iid}{\sim} \b{P}_{\m{X}}$, and $\varepsilon_{ijk} \stackrel{iid}{\sim} \c{N}(0, \sigma_{r_{i}}^{2})$ are all mutually independent, with $1 \le i \le n, 1 \le j \le r_{i}, 1 \le k \le s_{ij}$. For brevity, we denote $\varphi_{ijk}$ instead of $\varphi(\m{R}_{ij}, \m{X}_{ijk})$, an $\b{H}$-valued random element with its norm bounded by $B > 0$. 

To estimate the mean and covariance, we employ Tikhonov regularization, which yields finite-dimensional ridge regressions in our framework, owing to the reproducing property \eqref{eq:repro:prop}, with the normal equation provided by suitable Gram matrices. To formulate this, we first derive the conditional expectation 
given the random orientations and locations:
\begin{equation*}
    \b{E}[Z_{ijk} \vert (\m{R}_{ij}, \m{X}_{ijk}) ] = \s{P} \m{\mu} (\m{R}_{ij}, \m{X}_{ijk}) = \langle \m{\mu}, \varphi_{ijk} \rangle_{\b{H}}.
\end{equation*}
This leads to the empirical risk functional $\hat{L}_{\nu}:\b{H} \rightarrow \b{R}$ for mean estimation with a tuning parameter $\nu > 0$:
\begin{align}\label{eq:emp:risk:mean}
    \hat{L}_{\nu}(f)
    &:= \frac{\omega_{d_{r}, d}}{n} \sum_{i=1}^{n} \frac{1}{r_{i}} \sum_{j=1}^{r_{i}} \frac{1}{s_{ij}} \sum_{k=1}^{s_{ij}} \left(Z_{ijk}-\s{P}_{\m{R}_{ij}} f(\m{X}_{ijk}) \right)^{2} +\nu \|f\|^{2}_{\b{H}}, 
\end{align}
which can viewed as the negative log-posterior density under a white-noise prior on $f$ \cite{yun2025computerized}. We remark that other type of penalties are considered in the discrete and homogeneous tomography setup -- total variation \cite{jia2011gpu, unser2021unifying}, $\c{L}_{1}$-penalty \cite{zhang2016low, jiang2018reconstruction}, or hybrid projection \cite{chung2024computational}.

To estimate the covariance $\m{\Sigma}$, we employ a plug-in approach: first, we estimate the second moment operator $\m{\Gamma}$, and then we define the covariance estimator as $\hat{\m{\Sigma}} := \hat{\m{\Gamma}} - \hat{\m{\mu}} \otimes \hat{\m{\mu}}$. For second-order estimation, note that
\begin{align}\label{eq:WW:cond:exp}
    \b{E}[Z_{ijk} Z_{i'j'k'} \vert (\m{R}_{ij}, \m{X}_{ijk}) , (\m{R}_{i'j'}, \m{X}_{i'j'k'}) ] 
    &=
    \begin{cases}
        \langle \m{\mu}, \varphi_{ijk} \rangle_{\b{H}} \langle \m{\mu}, \varphi_{i'j'k'} \rangle_{\b{H}}, &i \neq i', \\
        \langle \m{\Gamma} \varphi_{ijk}, \varphi_{ij'k'} \rangle_{\b{H}} + \sigma_{r_{i}}^{2} \delta_{jj'} \delta_{kk'}, &i = i',
    \end{cases}
\end{align}
which reveals that the product of observations $Z_{ijk} Z_{ij'k'}$ from any off-diagonal pair $(j,k) \neq (j', k')$ provides a conditionally unbiased estimate of $\m{\Gamma}$ evaluated at the corresponding orientation and location. While this is true, we restrict our analysis to cross-terms where $j \neq j'$, which correspond to products across different projection images of the same object. This technical restriction is crucial for asymptotic analysis, because the feature maps $\varphi_{ijk}$ and $\varphi_{ij'k'}$ are \emph{fully} independent only when $j \ne j'$. In contrast, for $j = j'$ and $k \ne k'$, the feature maps are merely conditionally independent given $\m{R}_{ij}$, thus complicates the asymptotic derivation, because \eqref{eq:emp:tensor:UE} in \cref{sec:asym:cov} no longer hold unconditionally.

To simplify the notation, we merge the multi-index $(j,k)$ into $J$ when there is no risk of confusion. Accordingly, we denote by $\varphi_{iJJ'}^{\otimes} \in \b{H} \otimes \b{H}$ and $Z_{iJJ'} \in \b{R}$, instead of $\varphi_{ijk} \otimes \varphi_{ij'k'}$ and $Z_{ijk} Z_{ij'k'}$, respectively. While the mean can be estimated with just one projection ($r_{i} \ge 1$) per object  is enough for mean estimation, estimating the covariance requires at least two projections ($r_{i} \ge 2$) to capture the second-order interactions within the process $\c{Y}_{i}$. With these considerations, the empirical risk functional for the second moment operator $\hat{L}_{\eta}^{\odot}: \b{H} \otimes \b{H} \rightarrow \b{R}$ with a tuning parameter $\eta > 0$ is given as follow:
\begin{align}\label{eqn:emp:risk:cov}
    &\hat{L}_{\eta}^{\odot}(G)
    := \frac{\omega_{d_{r}, d}^{2}}{n} \sum_{i=1}^{n} \frac{1}{r_{i}(r_{i}-1)} \sum_{1 \le j \neq j' \le r_{i}} \frac{1}{s_{ij}s_{ij'}} \times \nonumber \\
    &\hspace*{6em} \quad \sum_{\substack{1 \le k \le s_{ij} \\ 1 \le k' \le s_{ij'}}} (Z_{ijk} Z_{ij'k'} - (\s{P}_{\m{R}_{ij}} \otimes \s{P}_{\m{R}_{ij'}}) G(\m{X}_{ijk}, \m{X}_{ij'k'}) )^{2} + \eta \|G\|^{2}_{\b{H} \otimes \b{H}}. 
\end{align}

\begin{proposition}\label{prop:repre::proj}
Given random orientations $\m{R}_{ij} \in \b{G}$ and locations $\m{X}_{ijk} \in \Omega^{d-d_{r}}$, define finite-dimensional subspaces of $\b{H}$ and $\b{H} \otimes \b{H}$:
\begin{align*}
   &\b{H}_{\m{F}}:= \text{span} \{\varphi_{ijk}: 1 \le i \le n, 1 \le j \le r_{i}, 1 \le k \le s_{ij} \}, \\
   &(\b{H} \otimes \b{H})_{\m{F}}:= \text{span} \{\varphi_{ijk} \otimes \varphi_{ij'k'}: 1 \le i \le n, 1 \le j \neq j' \le r_{i}, 1 \le k \le s_{ij}, 1 \le k' \le s_{ij'} \},
\end{align*}
respectively. Then, for $f \in \b{H}$ and $G \in \b{H} \otimes \b{H}$,
\begin{align*}
    &f \in \b{H}_{\m{F}}^{\perp} \quad \Leftrightarrow \quad \s{P}_{\m{R}_{ij}} f(\m{X}_{ij})=0, \quad 1 \le i \le n, 1 \le j \le r_{i}, \\
    &G \in (\b{H} \otimes \b{H})_{\m{F}}^{\perp} \quad \Leftrightarrow \quad (\s{P}_{\m{R}_{ij}} \otimes \s{P}_{\m{R}_{ij'}}) G(\m{X}_{ij}, \m{X}_{ij'})=0, \quad 1 \le i \le n, 1 \le j \neq j' \le r_{i}.
\end{align*}
\end{proposition}

The above proposition immediately leads us to the following representer theorem.

\begin{proposition}\label{prop:repre}
Given $\nu, \eta > 0$, let the empirical risk functionals $\hat{L}_{\nu}:\b{H} \rightarrow \b{R}$ and $\hat{L}_{\eta}^{\odot}: \b{H} \otimes \b{H} \rightarrow \b{R}$ be defined as in \eqref{eq:emp:risk:mean} and \eqref{eqn:emp:risk:cov}. 
Then, both the Tikhonov regularizers:
\begin{equation*}
    \hat{\m{\mu}}_{\nu} := \argmin_{f \in \b{H}} \hat{L}_{\nu}(f), \quad
    \hat{\m{\Gamma}}_{\eta} := \argmin_{G \in \b{H} \otimes \b{H}} \hat{L}_{\eta}^{\odot}(G).
\end{equation*}
are unique, and belong to $\b{H}_{\m{F}}$ and $(\b{H} \otimes \b{H})_{\m{F}}$, respectively. Additionally, $\hat{\m{\Gamma}}_{\eta}$ is self-adjoint.
\end{proposition}

Although the second moment operator $\m{\Gamma} \in \c{B}_{1}(\b{H})$ is trace-class, our estimation is carried out in the larger space of HS operators $\b{H} \otimes \b{H} \cong \c{B}_{2}(\b{H})$. Nevertheless, \cref{prop:repre} shows that the estimator $\hat{\m{\Gamma}}_{\eta}$ is finite-rank and therefore belongs to $\c{B}_{1}(\b{H})$ automatically as desired.
We proceed to find the coefficients  for the estimators $\hat{\m{\mu}}_{\nu} \in \b{H}_{\m{F}}$ and $\hat{\m{\Gamma}}_{\eta} \in (\b{H} \otimes \b{H})_{\m{F}}$ with respect to their corresponding frames. This requires defining the following observation vectors and sample-size-adjusted Gram matrices:
\begin{align}
    &\m{w}:= \mathrm{vec}\left[ \frac{Z_{ijk}}{\sqrt{r_{i}s_{ij}}} \right], \, \m{w}^{\odot}:= \mathrm{vec}\left[ \frac{Z_{ijk}Z_{ij'k'}}{\sqrt{r_{i}(r_{i}-1)s_{ij}s_{ij'}}} \right]_{j \neq j'}, \, 
    \m{\Phi}:= \left[ 
    \frac{\langle \varphi_{i_{1}j_{1}k_{1}}, \varphi_{i_{2}j_{2}k_{2}} \rangle_{\b{H}}}{\sqrt{r_{i_{1}}s_{i_{1}j_{1}}} \sqrt{r_{i_{2}}s_{i_{2}j_{2}}}} \right], \label{eq:mean:Gram} \\
    &\m{\Phi}^{\odot}:= \left[ 
    \frac{\langle \varphi_{i_{1}j_{1}k_{1}} \otimes \varphi_{i_{1}j'_{1}k'_{1}}, \varphi_{i_{2}j_{2}k_{2}} \otimes \varphi_{i_{2}j'_{2}k'_{2}} \rangle_{\b{H} \otimes \b{H}} 
    }{\sqrt{r_{i_{1}}(r_{i_{1}}-1)s_{i_{1}j_{1}}s_{i_{1}j'_{1}}} \sqrt{r_{i_{2}}(r_{i_{2}}-1)s_{i_{2}j_{2}}s_{i_{2}j'_{2}}}} \right]_{\substack{j_{1} \neq j'_{1} \\ j_{2} \neq j'_{2}}}. \label{eq:tens:Gram}
\end{align}
Here, the symbol $\odot$ stands for a block-wise tensor product, also called the Khatri-Rao product \cite{khatri1968solutions, neudecker1995hadamard}.
Note that $\m{\Phi}$ and $\m{\Phi}^{\odot}$ are both s.p.d, with their components calculated using \eqref{eq:Gram:elem}.

\begin{theorem}\label{thm:repre:mean}
Let $\nu, \eta > 0$ be penalty parameters. Then, the minimizer of  $\hat{L}_{\nu}:\b{H} \rightarrow \b{R}$ in \eqref{eq:emp:risk:mean} is given by
\begin{equation*}
    \hat{\m{\mu}}_{\nu} = \sum_{i=1}^{n} \frac{1}{\sqrt{r_{i}}} \sum_{j=1}^{r_{i}} \frac{1}{\sqrt{s_{ij}}} \sum_{k=1}^{s_{ij}} \hat{\alpha}^{\nu}_{ijk} \varphi(\m{R}_{ij}, \m{X}_{ijk}), \quad \hat{\m{\alpha}}_{\nu}:=\mathrm{vec}[\hat{\alpha}^{\nu}_{ijk}] = \left( \m{\Phi}+ \frac{n \nu}{\omega_{d_{r}, d}} \m{I} \right)^{-1} \m{w}.
\end{equation*}
The minimizer of $\hat{L}_{\eta}^{\odot}: \b{H} \otimes \b{H} \rightarrow \b{R}$ in \eqref{eqn:emp:risk:cov} is given by
\begin{equation*}
    \hat{\m{\Gamma}}_{\eta} = \sum_{i=1}^{n} \frac{1}{\sqrt{r_{i}(r_{i}-1)}} \sum_{1 \le j \neq j' \le r_{i}} \frac{1}{\sqrt{s_{ij}s_{ij'}}} \sum_{\substack{1 \le k \le s_{ij} \\ 1 \le k' \le s_{ij'}}} \hat{\alpha}^{\odot}_{ijkj'k'} \varphi(\m{R}_{ij}, \m{X}_{ijk}) \otimes \varphi(\m{R}_{ij'}, \m{X}_{ij'k'}),
\end{equation*}
where $\hat{\m{\alpha}}_{\eta}^{\odot}:=\mathrm{vec}[\hat{\alpha}^{\odot}_{ijkj'k'}] = ( \m{\Phi}^{\odot}+ n \eta/\omega_{d_{r}, d}^{2} \m{I} )^{-1} \m{w}^{\odot}$.
\end{theorem}

By \cref{prop:A1:check}, the evaluations of $\hat{\m{\mu}}_{\nu}: \Omega^{d} \rightarrow \b{R}$ and $\hat{\m{\Gamma}}_{\eta}: \Omega^{d} \times \Omega^{d} \rightarrow \b{R}$ are given by
\begin{align*}
    &\hat{\m{\mu}}_{\nu} (\m{z}) = \sum_{i=1}^{n} \frac{1}{\sqrt{r_{i}}} \sum_{j=1}^{r_{i}} \frac{1}{\sqrt{s_{ij}}} \sum_{k=1}^{s_{ij}} \hat{\alpha}^{\nu}_{ijk} \s{P}_{\m{R}_{ij}} k_{\m{z}} (\m{X}_{ijk}), \\
    &\hat{\m{\Gamma}}_{\eta} (\m{z}, \m{z}') = \sum_{i=1}^{n} \sum_{1 \le j \neq j' \le r_{i}}  \sum_{\substack{1 \le k \le s_{ij} \\ 1 \le k' \le s_{ij'}}} \frac{\hat{\alpha}^{\odot}_{ijkj'k'}}{\sqrt{r_{i}(r_{i}-1) s_{ij}s_{ij'}}}  \s{P}_{\m{R}_{ij}} k_{\m{z}} (\m{X}_{ijk}) \s{P}_{\m{R}_{ij'}} k_{\m{z}'} (\m{X}_{ij'k'}),
\end{align*}
and thus, $\hat{\m{\Gamma}}_{\eta}: \Omega^{d} \times \Omega^{d} \rightarrow \b{R}$ is a symmetric bivariate function: $\hat{\m{\Gamma}}_{\eta} (\m{z}, \m{z}') = \hat{\m{\Gamma}}_{\eta} (\m{z}', \m{z})$.

Once we obtain $\hat{\m{\Gamma}}_{\eta}$ and $\hat{\m{\mu}}_{\nu}$ for some $\nu, \eta>0$, we estimate the covariance using the plug-in formula: $\hat{\m{\Sigma}}_{\nu, \eta} = \hat{\m{\Gamma}}_{\eta} - \hat{\m{\mu}}_{\nu} \otimes \hat{\m{\mu}}_{\nu}$, or equivalently,
$\hat{\m{\Sigma}}_{\nu, \eta} (\m{z}, \m{z}') = \hat{\m{\Gamma}}_{\eta} (\m{z}, \m{z}')- \hat{\m{\mu}}_{\nu} (\m{z}) \hat{\m{\mu}}_{\nu} (\m{z}')$, which remains symmetric. 
However, a key issue arises with this approach. As discussed in \cref{rmk:spd:fail}, the second moment estimator $\hat{\m{\Gamma}}_{\eta}$ is not s.p.d. unless it is identically zero \cite{yun2025fast}. Consequently, the covariance estimator $\hat{\m{\Sigma}}_{\nu, \eta}$ also fails to be s.p.d. in general. This problem stems from the kernel method with the exclusive use of off-diagonal product pairs in the estimation procedure, and \emph{not} from the inverse problem setup.
In practice, however, this deviation from the s.p.d. property is negligible as our estimators are consistent. Still, in finite-sample scenarios, one may project $\hat{\m{\Sigma}}_{\nu, \eta}$ onto the closed convex cone of s.p.d. tensors, which preserves the statistical consistency and does not alter the convergence rate in \cref{sec:asym:cov}.

\subsection{Numerical Challenges}\label{ssec:num:chal}

While \cref{thm:repre:mean} provides elegant least-squares solutions, a na\"{i}ve implementation is computationally infeasible. Two main bottlenecks arise: solving the normal equation for the covariance and performing the subsequent kernel PCA. This section briefly discusses these challenges and our numerical strategies for overcoming them, with further details available in Supplement \ref{sec:num:consid} and recent work \cite{yun2025fast}.

Recall from \cref{thm:repre:mean} that the second-moment estimation requires solving the linear system: $( \m{\Phi}^{\odot}+ n \eta/\omega_{d_{r}, d}^{2} \m{I} ) \hat{\m{\alpha}}_{\eta}^{\odot} = \m{w}^{\odot}$.
The problem with a direct, factorization-based approach is that it requires explicitly forming the covariance Gram matrix $\m{\Phi}^{\odot}$, which is often computationally prohibitive. For example, in our simulation with $n=100$ particles, $r=5$ tilts, and $s=10$ locations per tilt, this Gram matrix has over $50$ billion entries. Storing it in double-precision format would require $377$ GB of memory, compared to just $191$ MB for the mean Gram matrix $\m{\Phi}$.
The solution is to use a \emph{matrix-free} iterative solver -- we never explicitly form the Khatri-Rao product $\m{\Phi}^{\odot}$, as its action can be implemented only using much smaller matrix $\m{\Phi}$. The \texttt{TReK} algorithm \cite{yun2025fast}, a tensorized version of Krylov subspace method \cite{saad2003iterative, greenbaum1997iterative, bjorck2024numerical}, is designed to solve such least-squares problems.

A second challenge arises when performing kernel PCA to analyze the estimated covariance $\hat{\m{\Sigma}}_{\nu, \eta} = \hat{\m{\Gamma}}_{\eta} - \hat{\m{\mu}}_{\nu} \otimes \hat{\m{\mu}}_{\nu}$. After arranging the solution vector $\hat{\m{\alpha}}_{\eta}^{\odot}$ into the block-diagonal matrix $\hat{\m{A}}_{\eta}$,
the kernel PCA reduces to the generalized eigenvalue problem, see \cref{thm:fpca}:
\begin{equation*}
    [\hat{\m{A}}_{\eta} - \hat{\m{a}}_{\nu} \hat{\m{a}}_{\nu}^{\top}] \m{\Phi} \m{V} = \m{V} \m{\Lambda}, \quad \m{V}^{*} \m{\Phi} \m{V} = \m{I}, 
\end{equation*}
an equation arising frequently in statistics. A common two-step approach to this problem is to first computing the usual eigendecomposition $\m{\Phi}^{1/2} [\hat{\m{A}}_{\eta} - \hat{\m{a}}_{\nu} \hat{\m{a}}_{\nu}^{\top}] \m{\Phi}^{1/2} = \m{U} \m{\Lambda} \m{U}^{\top}$, and then recover the generalized eigenvectors $\m{V}$ via $\m{U} = \m{\Phi}^{1/2} \m{V}$ \cite{caponera2022functional, ramsay2002applied}. However, this requires computing the matrix square root $\m{\Phi}^{1/2}$, which is slow and can be numerically unstable if $\m{\Phi}$ is ill-conditioned or has a low numerical rank \cite{nocedal1999numerical}. 
To avoid this, we again use a matrix-free iterative method based on the Lanczos tridiagonalization \cite{lanczos1950iteration, paige1972computational}, which is faster and more stable than the two-step approach, especially when we only need to extract the leading eigenvalues and eigenvectors $(\lambda_{k}, \m{v}_{k})$, often called the Ritz pairs \cite{jia2020low, chung2024computational}.

\section{Mercer's Theory in Inverse Problems}\label{sec:Mercer:inv}
This section describes the interplay between the Mercer decomposition and the inverse problem, which is essential for developing asymptotic theory. In \cref{ssec:direct:msr}, we explain how the theory reduces to the classical FDA setup in the absence of inverse problems.

Given a sinogram kernel $K$ on $\Omega^{d}$, we first define the induced reproducing kernel for the space of projection images:
\begin{align}\label{eq:induced:RKHS}
    \tilde{K} ((R, \m{x}), (R', \m{x}')) := \langle \varphi(R, \m{x}), \varphi(R', \m{x}') \rangle_{\b{H}}, \quad R, R' \in \b{G}, \, \m{x}, \m{x}' \in \Omega^{d-d_{r}},
\end{align}
Although not required for the asymptotic theory, we shortly remark another powerful advantage of the RKHS framework. In \cref{thm:induced:RKHS}, we show that $\tilde{K}$ is also a $\b{G}$-invariant Mercer kernel. More importantly, when considered as an operator between RKHSs, $\s{P}: \b{H}(K) \rightarrow \b{H}(\tilde{K})$ is a well-defined, contractive, and surjective map whose adjoint is an isometry.

Consider the integral operator $\s{T}_{K} \in \c{B}_{1}(\c{L}_{2} (\Omega^{d}, W_{0, d}))$ associated with $K$:
\begin{equation*}
    \s{T}_{K} f(\cdot) := \int_{\Omega^{d}} K(\m{z}, \cdot) f(\m{z}) W_{0,d}(\m{z}) \rd \m{z}.
\end{equation*}
The Mercer theorem reveals that $\s{T}_{K}$ is s.p.d. and trace-class, and $\b{H} = \b{H}(K)$ is the range space of the HS operator
$(\s{T}_{K})^{1/2} \in \c{B}_{2}(\c{L}_{2} (\Omega^{d}, W_{0, d}))$, equipped with the range norm \cite{sarason2007complex}:
\begin{align*}
    \b{H} = \{(\s{T}_{K})^{1/2} h : h \in \c{L}_{2} (\Omega^{d}, W_{0, d}) \}, \quad \|f\|_{\b{H}} = \inf\{\|h\|_{\c{L}_{2} (\Omega^{d}, W_{0, d})}: f = (\s{T}_{K})^{1/2} h\},
\end{align*}
It follows that the continuous embedding $\imath: \b{H} \hookrightarrow \c{L}_{2} (\Omega^{d}, W_{0, d})$ is a HS operator satisfying 
\begin{equation*}
    \s{T}_{K} = \imath \circ \imath^{*}, \quad \VERT \imath \VERT_{2}^{2} = \VERT \imath^{*} \circ \imath \VERT_{1} = \VERT \s{T}_{K} \VERT_{1}= \int_{\Omega^{d}} K(\m{z}, \m{z}) W_{0,d}(\m{z}) \rd \m{z} < +\infty.
\end{equation*}
This implies that when we restrict the sinogram operator's domain to $\b{H}(K)$, the embedded operator is also a HS operator:
\begin{equation*}
    \s{P} \circ \imath : \b{H}(K) \hookrightarrow \c{L}_{2} (\Omega^{d}, W_{0, d}) \rightarrow \c{L}_{2} (\b{G} \times \Omega^{d-d_{r}}, W_{d_{r}, d}), f \mapsto \s{P} f.
\end{equation*}

\begin{proposition}\label{prop:swap:int:op}
Let $\s{P}$ be a sinogram operator, and $K: \Omega^{d} \times \Omega^{d} \rightarrow \b{R}$ be a sinogram kernel, and $\tilde{K}: (\b{G} \times \Omega^{d-d_{r}}) \times (\b{G} \times \Omega^{d-d_{r}}) \rightarrow \b{R}$ be an induced kernel as in \eqref{eq:induced:RKHS}. Then,
\begin{equation*}
    \s{T}_{\tilde{K}} = (\s{P} \circ \imath) (\s{P} \circ \imath)^{*} = \s{P} \s{T}_{K} \s{P}^{*} \in \c{B}_{1}(\c{L}_{2} (\b{G} \times \Omega^{d-d_{r}}, W_{d_{r}, d})).
\end{equation*}
\end{proposition}

Consequently, consider the singular value decomposition (SVD) of $\s{P} \circ \imath$:
\begin{equation*}
    \s{P} \circ \imath = \sum_{l=1}^{\infty} \tau_{l} \sum_{m=1}^{N(l)} g_{lm} \otimes e_{lm} : \b{H}(K) \rightarrow \c{L}_{2} (\b{G} \times \Omega^{d-d_{r}}, W_{d_{r}, d}),
\end{equation*}
where $\tau_{1} \ge \tau_{2} \ge \dots \ge 0$ are the non-zero singular values, $N(l)$ is the multiplicity of $\tau_{l}$, and $e_{lm} \in \b{H}(K)$, $g_{lm} \in \c{L}_{2} (\b{G} \times \Omega^{d-d_{r}}, W_{d_{r}, d})$ have unit norms. Observe that $\{g_{lm} : l \in \b{N}, 1 \le m \le N(l) \}$ and $\{e_{lm} : l \in \b{N}, 1 \le m \le N(l) \}$ form CONSs for $\overline{\c{R}(\s{T}_{\tilde{K}})}$ and $\overline{\c{R}(\s{K})}$, respectively. However, $e_{lm}$'s may not be orthogonal in $\c{L}_{2} (\Omega^{d}, W_{0, d})$, unless $\s{P}^{*} \s{P}$ and $\s{T}_{K}$ commute. Additionally, note that $g_{lm} = \tau_{l}^{-1} \s{P} e_{lm}$, hence
\begin{equation}\label{eq:bdd:representer}
    \sup_{R, \m{x}} |g_{lm} (R, \m{x})| = \tau_{l}^{-1} \sup_{R, \m{x}} |\s{P} e_{lm} (R, \m{x})| \le \tau_{l}^{-1} \|e_{lm}\|_{\b{H}} \sup_{R, \m{x}} \|\varphi(R, \m{x})\|_{\b{H}} = \tau_{l}^{-1} B.
\end{equation}
Also, observe from the intertwining property in \cref{prop:unit:repn} that
\begin{equation*}
    g_{lm} (R, \m{x}) = \rho(R) g_{lm} (I, \m{x}) = \tau_{l}^{-1} \s{P} [\rho(R) e_{lm}] (I, \m{x}) = \tau_{l}^{-1} \s{P}_{I} [\rho(R) e_{lm}] ( \m{x}).
\end{equation*}

It follows from \cref{prop:swap:int:op} that
\begin{equation*}
    \sum_{l} \tau_{l}^{2} N(l) = \VERT \s{T}_{\tilde{K}} \VERT_{1} = \int_{\Omega^{d-d_{r}}} \tilde{K}((R, \m{x}), (R, \m{x})) W_{d_{r}, d}(\m{x}) \rd \m{x} < +\infty,
\end{equation*}
which is independent of the choice of $R \in \b{G}$ due to $\b{G}$-invariance of $\tilde{K}$. Consequently, $\s{K} := (\s{P} \circ \imath)^{*}(\s{P} \circ \imath) \in \c{B}_{1}(\b{H})$, a central operator in asymptotic theory, admits the spectral decomposition:
\begin{equation*}
    \s{K} = \sum_{l=1}^{\infty} \tau_{l}^{2} \sum_{m=1}^{N(l)} e_{lm} \otimes_{\b{H}} e_{lm}.
\end{equation*}
This operator gives the $\c{L}_{2}$-inner product between sinograms in the following sense:
\begin{equation}\label{eq:isomet}
    \langle f_{1} , \s{K} f_{2} \rangle_{\b{H}} =  \langle \s{P} f_{1} , \s{P} f_{2} \rangle_{\c{L}_{2} (\b{G} \times \Omega^{d-d_{r}}, W_{d_{r}, d})}, \quad f_{1}, f_{2} \in \b{H}.
\end{equation}
Additionally, $\s{K}/\omega_{d_{r}, d}$ is precisely the expectation of the tensorized random feature map:

\begin{proposition}\label{prop:isometry}
Let $\s{P}$ be a sinogram operator and $K$ be a sinogram kernel. Let $\m{R} \sim \b{P}_{\m{R}}$ and $\m{X} \sim \b{P}_{\m{X}}$ be independent. Then, the operator $\s{K}$ satisfies
\begin{align*}
    \s{K} = \omega_{d_{r}, d} \b{E}[\varphi(I, \m{X}) \otimes \varphi(I, \m{X})]
    = \omega_{d_{r}, d} \b{E}[\varphi(\m{R}, \m{X}) \otimes \varphi(\m{R}, \m{X})].
\end{align*}
Moreover, $\s{K} \in \c{B}_{1}(\b{H})$ is s.p.d. with the norm $\VERT \s{K} \VERT_{1} = \omega_{d_{r}, d} \b{E}[\|\varphi(\m{R}, \m{X})\|_{\b{H}}^{2}]$.
\end{proposition}

We now explain how $\b{G}$-invariance yields the orthogonal decompositions associated with the SVD of $\s{P} \circ \imath$. For each $l \in \b{N}$, define the finite-dimensional subspaces
\begin{align*}
    \b{H}_{l} :&= \{f \in \b{H}: \s{K} f = \tau_{l}^{2} f \}= \spann \{e_{lm} : m = 1, 2, \dots, N(l) \}, \\
    \tilde{\c{H}}_{l} :&= \{g \in \c{L}_{2} (\b{G} \times \Omega^{d-d_{r}}, W_{d_{r}, d}) : \s{T}_{\tilde{K}} g = \tau_{l}^{2} g \}= \spann \{g_{lm} : m = 1, 2, \dots, N(l) \} \nonumber.
\end{align*}
Since $\b{H}_{l}$'s are $\b{G}$-invariant subspaces with respect to the unitary representation $(\rho, \b{H})$ from \cref{prop:unit:repn}, $[\rho(R)]_{\{e_{lm} : 1 \le m \le N(l)\}} \in SO(N(l))$ for any $R \in \b{G}$, similar to the Wigner D-matrix \cite{wigner1931gruppentheorie, dai2013approximation}. We can draw the same conclusion for $\tilde{\c{H}}_{l}:=\s{P}(\b{H}_{l})$, hence the following orthogonal decomposition of the output space:
\begin{equation*}
    \c{L}_{2} (\b{G} \times \Omega^{d-d_{r}}, W_{d_{r}, d}) = \bigoplus_{l =1}^{\infty} \tilde{\c{H}}_{l} \oplus \c{R}(\s{P} \circ \imath)^{\perp}.
\end{equation*}
If we view the finite-dimensional $\b{G}$-invariant subspace $(\tilde{\c{H}}_{l}, W_{d_{r}, d})$ as an RKHS, equipped with the weighted $\c{L}_{2}$ norm, the corresponding reproducing kernel $\tilde{K}_{l}: (\b{G} \times \Omega^{d-d_{r}}) \times (\b{G} \times \Omega^{d-d_{r}}) \rightarrow \b{R}$ is given by
\begin{equation}\label{eq:daughter:L2:rk}
    \tilde{K}_{l}((R_{1}, \m{x}_{1}), (R_{2}, \m{x}_{2})) := \sum_{m=1}^{N(l)} g_{lm}(R_{1}, \m{x}_{1}) g_{lm} (R_{2}, \m{x}_{2}),
\end{equation}
which is well-defined by virtue of \eqref{eq:bdd:representer}. Then, the Mercer decomposition of the induced reproducing kernel can be expressed as $\tilde{K}=\sum_{l=1}^{\infty} \tau_{l}^{2} \tilde{K}_{l}$.

\begin{proposition}\label{prop:daughter:rk}
For each $l \in \b{N}$, $\tilde{K}_{l}$ satisfies the reproducing property:
\begin{align*}
    &\tilde{K}_{l}((R_{1}, \m{x}_{1}), (R_{2}, \m{x}_{2})) 
    = \langle \tilde{K}_{l}((R_{1}, \m{x}_{1}), \cdot), \tilde{K}_{l}((R_{2}, \m{x}_{2}), \cdot) \rangle_{\c{L}_{2} (\b{G} \times \Omega^{d-d_{r}}, W_{d_{r}, d})} \\
    &= \int_{\b{G}} \int_{\Omega^{d-d_{r}}} \tilde{K}_{l}((R_{1}, \m{x}_{1}), (R_{3}, \m{x}_{3})) \tilde{K}_{l}((R_{2}, \m{x}_{2}), (R_{3}, \m{x}_{3})) W_{d_{r}, d}(\m{x}_{3}) \rd \m{x}_{3} \rd R_{3}.
\end{align*}
Moreover, the associated integral operator $\s{T}_{\tilde{K}_{l}} \in \c{B}_{\infty}(\c{L}_{2} (\b{G} \times \Omega^{d-d_{r}}, W_{d_{r}, d}))$ is the orthogonal projection onto $\tilde{\c{H}}_{l}$. The kernel $\tilde{K}_{l}$ is $\b{G}$-invariant, and satisfies 
\begin{equation*}
    \int_{\Omega^{d-d_{r}}} \tilde{K}_{l}((R, \m{x}), (R, \m{x})) W_{d_{r}, d}(\m{x}) \rd \m{x} = N(l).
\end{equation*}
\end{proposition}

\section{Asymptotic Analysis}\label{sec:asym}
To provide the uniform rate of convergence of our estimators, we require an additional assumption:
\begin{enumerate}
\item [(A)] There exists some $\tilde{B} >0$ such that, the $l$-th kernel $\tilde{K}_{l}$ in \eqref{eq:daughter:L2:rk} induced by $K$ satisfies
\begin{equation*}
    \frac{\tilde{K}_{l}((I, \m{x}), (I, \m{x}))}{N(l)} \le \tilde{B}, \quad l \in \b{N}, \ \m{x} \in \Omega^{d-d_{r}}. 
\end{equation*}
\end{enumerate}
Since the kernel $\tilde{K}_{l}$ is $\b{G}$-invariant and s.p.d., condition (A) is equivalent to
\begin{equation*}
    \frac{\tilde{K}_{l}((R, \m{x}), (R', \m{x}'))}{N(l)} \le \tilde{B}, \quad l \in \b{N}, \ R, R' \in \b{G}, \ \m{x}, \m{x}' \in \Omega^{d-d_{r}}. 
\end{equation*}

In view of \cref{prop:daughter:rk}, condition (A) amounts to requiring that the supremum norm of $\tilde{K}_{l}$ is uniformly bounded across $l \in \b{N}$ in terms of its $\c{L}_{1}$-norm along diagonal pairs. Heuristically, as $l \in \b{N}$ grows, each singular function $g_{lm}$ exhibits higher frequency oscillations; however, the growth in the dimension $N(l)$ counterbalances this fluctuation in the construction of $\tilde{K}_{l}$. For instance, in \cref{rmk:transact:action}, although each spherical harmonic of degree $l$ becomes more oscillatory with increasing degree $l$, the associated reproducing kernel still satisfies (A). More generally, if the group action of $\b{G}$ is transitive, then $\tilde{K}_{l}((I, \m{x}), (I, \m{x}))$ is constant along diagonal pairs. In this case, (A) is automatically satisfied for any sinogram kernel $K$:
\begin{equation*}
    \tilde{K}_{l}((I, \m{x}), (I, \m{x})) = \omega_{d_{r}, d}^{-1} N(l).
\end{equation*}

On the other hand, when the unitary group action on the domain is not transitive (like the action of $SO(d)$ on $\b{B}^{d}$ for the X-ray transform), one may take a different approach. If we have explicit knowledge of the SVD of $\s{P}: \c{L}_{2} (\Omega^{d}, W_{0, d}) \to \c{L}_{2} (\b{G} \times \Omega^{d-d_{r}}, W_{d_{r}, d})$:
\begin{equation*}
    \s{P} = \sum_{l=1}^{\infty} \lambda_{l} \sum_{m=1}^{N(l)} g_{lm} \otimes f_{lm}.
\end{equation*}
we can construct a compatible Mercer kernel using the singular functions $f_{lm}$:
\begin{equation*}
    K(\m{z}_{1}, \m{z}_{2})=\sum_{l=1}^{\infty} \zeta_{l}^{2} \sum_{m=1}^{N(l)} f_{lm}(\m{z}_{1}) f_{lm}(\m{z}_{2}),
\end{equation*}
where $\sum_{l=1}^{\infty} \zeta_{l}^{2} N(l) < + \infty$, so that $\s{P}$ and $\s{T}_{K}$ are simultaneously singular value decomposable, i.e., $\s{P}^{*} \s{P}$ and $\s{T}_{K}$ commute. In this setup, the coefficients satisfy $\tau_{l} = \lambda_{l} \zeta_{l}$ and the validity of the assumption (A) for the $l$-th kernel $\tilde{K}_{l}$ in \eqref{eq:daughter:L2:rk} depends solely on the SVD of $\s{P}$, as illustrated in \cref{ex:radon:kernel} for the Radon transform.


\subsection{Asymptotic Behavior of the Mean Estimator}\label{sec:asym:mean}
We present the uniform convergence rates of our mean estimator both in input and output levels, which are $\|\hat{\m{\mu}}_{\nu_{n}}-\m{\mu} \|_{\b{H}}$ and $\| \s{P} \hat{\m{\mu}}_{\nu_{n}}- \s{P} \m{\mu} \|_{\c{L}_{2} (\b{G} \times \Omega^{d-d_{r}}, W_{d_{r}, d})}$, respectively. Let us define the empirical operator
\begin{align*}
    \hat{\s{K}} := \frac{\omega_{d_{r}, d}}{n} \sum_{i=1}^{n} \frac{1}{r_{i}} \sum_{j=1}^{r_{i}} \frac{1}{s_{ij}} \sum_{k=1}^{s_{ij}} \varphi_{ijk} \otimes \varphi_{ijk} \in \c{B}_{\infty}(\b{H}).
\end{align*}
Then, it is straightforward from \cref{prop:isometry} that $\s{K} = \b{E}[\hat{\s{K}}]$. The decomposition
\begin{align*}
    \b{E}[Z_{ijk}-\langle \varphi_{ijk}, f \rangle_{\b{H}}]^{2} 
    &= \b{E}[(Z_{ijk}-\langle \varphi_{ijk}, \m{\mu} \rangle_{\b{H}}) + \langle \varphi_{ijk}, \m{\mu} - f \rangle_{\b{H}}]^{2} \\
    &= \b{E}[(Z_{ijk}-\langle \varphi_{ijk}, \m{\mu} \rangle_{\b{H}})^{2}] + \b{E}[\langle \varphi_{ijk}, \m{\mu} - f \rangle_{\b{H}}^{2}],
\end{align*}
yields the expectation of empirical risk functional $\hat{L}_{\nu}$ in \eqref{eq:emp:risk:mean} as
\begin{align*}
    L_{\nu}(f)
    :&= \b{E}[ \hat{L}_{\nu}(f) ] 
    = \omega_{d_{r}, d} \b{E}[\left(W_{111}-\langle \varphi_{111}, \m{\mu} \rangle_{\b{H}} \right)^{2}] + \langle \s{K}(\m{\mu}-f), \m{\mu}-f \rangle_{\b{H}} + \nu \|f\|^{2}_{\b{H}} \nonumber \\
    &= \omega_{d_{r}, d} \b{E}[\left(W_{111}-\s{P}_{\m{R}_{11}} \m{\mu}(\m{X}_{111}) \right)^{2}] + \langle \s{K}(\m{\mu}-f), \m{\mu}-f \rangle_{\b{H}} + \nu \|f\|^{2}_{\b{H}},
\end{align*}
where the first term does not depend on $f \in \b{H}$. By strong convexity, its minimizer is unique:
\begin{equation*}
    \bar{\m{\mu}}_{\nu} := \argmin_{f \in \b{H}} L_{\nu}(f).
\end{equation*}
Subsequently, we get explicit expressions for $\hat{\m{\mu}}_{\nu}$ and $\bar{\m{\mu}}_{\nu}$ similar to ridge regression \cite{hsing2015theoretical}:
\begin{equation*}
    \hat{\m{\mu}}_{\nu} = (\hat{\s{K}} + \nu I)^{-1} \c{Z}_{n}, \quad
    \bar{\m{\mu}}_{\nu} = (\s{K} + \nu I)^{-1} \s{K} \m{\mu},
\end{equation*}
where
\begin{equation*}
    \c{Z}_{n} := \frac{\omega_{d_{r}, d}}{n} \sum_{i=1}^{n} \frac{1}{r_{i}} \sum_{j=1}^{r_{i}} \frac{1}{s_{ij}} \sum_{k=1}^{s_{ij}} Z_{ijk} \varphi_{ijk} \in \b{H}.
\end{equation*}
Thus, it follows that
\begin{align}\label{eq:mean:asym:1st}
    \m{\mu} - \bar{\m{\mu}}_{\nu} = (\s{K} + \nu I)^{-1} (\s{K} + \nu I) \m{\mu} - \nu (\s{K} + \nu I)^{-1} \m{\mu} = \nu (\s{K} + \nu I)^{-1} \m{\mu}.
\end{align}

Before presenting the results, we argue why a source condition on $\m{\mu}$ is necessary if we aim to obtain an asymptotic result at the input level $\|\hat{\m{\mu}}_{\nu_{n}}-\m{\mu} \|_{\b{H}}$. Observe from \cref{prop:repre::proj} that $\varphi(R, \m{x}) \in \c{N}(\s{K})^{\perp} = \overline{\c{R}(\s{K})}$ for any $R \in \b{G}$ and $\m{x} \in \Omega^{d-d_{r}}$, and therefore their finite linear combination $\hat{\m{\mu}}_{\nu_{n}}$ must also lie in $\overline{\c{R}(\s{K})}$. Thus, a \emph{minimal} requirement for the consistency of $\hat{\m{\mu}}_{\nu_{n}}$ is that $\m{\mu} \in \overline{\c{R}(\s{K})}$; otherwise, the model would necessarily be misspecified.

Although it is true that $\|\bar{\m{\mu}}_{\nu_{n}}-\m{\mu} \|_{\b{H}} \to 0$ as the $\nu_{n} \rightarrow 0$ for a \emph{single} fixed $\m{\mu} \in \overline{\c{R}(\s{K})}$ (see Theorem 6.2.3 in \cite{hsing2015theoretical}), a closer examination using \eqref{eq:mean:asym:1st} shows that
\begin{equation*}
    \sup \{ \|\bar{\m{\mu}}_{\nu}-\m{\mu} \|_{\b{H}}^{2} : \m{\mu} \in \overline{\c{R}(\s{K})}, \|\m{\mu}\|_{\b{H}}^{2} \le M_{1} \}  = M_{1}, \quad M_{1} >0, \nu>0,
\end{equation*}
which implies that mere boundedness of $\m{\mu} \in \b{H}$ is insufficient to guarantee \emph{uniform} consistency of $\bar{\m{\mu}}_{\nu_{n}}$. This motivates the introduction of a source condition, parameterized by $\kappa_{1}, M_{1}$ to control the regularity of $\m{\mu}$, and by $M_{2}$ to control the heterogeneity of $\c{Y}_{i}$:

\begin{definition}[Source regularity]\label{def:srce:reg}
For $0 \le \kappa_{1} \le 1/2$ and $M_{1}, M_{2}>0$, we say that a second-order $\b{H}$-valued process $\c{Y}$ belongs to the class $\c{S}_{1}(\kappa_{1}, M_{1}, M_{2})$ if
    \begin{equation*}
        \m{\mu} \in \{\s{K}^{\kappa_{1}} \delta: \delta \in \b{H}, \|\delta\|_{\b{H}}^{2} \le M_{1}\} \subsetneq \b{H}, \quad \VERT \m{\Gamma} \VERT_{1} = \b{E} [\|\c{Y}\|_{\b{H}}^{2}] \le M_{2}.
    \end{equation*}
When $\kappa_{1} = 0$, we simply write $\c{S}_{1}(\kappa_{1}, M_{1}, M_{2})$ in place of $\c{S}_{1}(M_{1}, M_{2})$.
\end{definition}

Since $\s{K} \in \c{B}_{1}(\b{H})$ is a trace-class operator, increasing the value of $\kappa_{1} \ge 0$ imposes greater smoothness on the mean function $\m{\mu} \in \b{H}$. However, once $\kappa_{1} \ge 1/2$, further increasing $\kappa_{1}$ does not improve the convergence rate; see the discussion preceding \cref{sec:proof:asym:mean} for details. Moreover, the Douglas factorization lemma \cite{douglas1966majorization} implies that
\begin{equation}\label{eq:Doug:equiv}
    \m{\mu} \in \{\s{K}^{\kappa_{1}} \delta: \delta \in \b{H}, \|\delta\|_{\b{H}}^{2} \le M_{1}\} \quad \Longleftrightarrow \quad M_{1} \s{K}^{2 \kappa_{1}} \succeq \m{\mu} \otimes \m{\mu},
\end{equation}
where $\succeq$ denotes the Loewner order for operators. In the maximally smooth case where $\kappa_{1}=1/2$, relation \eqref{eq:Doug:equiv} can be further expressed, using \eqref{eq:isomet}, as
\begin{equation*}
    \m{\mu} \in \{\s{K}^{1/2} \delta: \delta \in \b{H}, \|\delta\|_{\b{H}}^{2} \le M_{1}\} \, \Leftrightarrow \, 
    \langle \m{\mu}, f \rangle_{\b{H}}^{2} \le M_{1} \|\s{P} f\|_{\c{L}_{2} (\b{G} \times \Omega^{d-d_{r}}, W_{d_{r}, d})}^{2}, \, f \in \b{H}(K).
\end{equation*}

The heteroscedasticity of the errors introduces a dependence of the convergence rate on two key quantities:
\begin{equation*}
    \overline{r} := \left[ \frac{1}{n} \sum_{i=1}^{n} \frac{1}{r_{i}} \right]^{-1}, \quad
    \overline{r}^{\sigma^{2}} := \left[ \frac{1}{n} \sum_{i=1}^{n} \frac{1}{r_{i}} \left(\frac{\sigma_{r_{i}}^{2}}{\sigma^{2}} \right) \right]^{-1}.
\end{equation*}
We refer to these quantities as the \emph{effective tilt number} and the \emph{noise-weighted effective tilt number}, respectively. They represent the harmonic mean of the number of tilts, with the second term adjusted for the noise variance in each measurement. In the second term, we normalize by a baseline variance $\sigma^{2}$, which is set to $\sigma_{1}^{2}$ for reference.

To make the asymptotic results more interpretable, we assume that the eigenvalues $\tau_{l}^{2}$ of $\s{K}$ (or $\s{T}_{\tilde{K}}$) decay in a polynomial speed,  while their multiplicities $N(l)$ grow polynomially:
\begin{equation*}
    \tau_{l} \asymp l^{-p}, \quad N(l) \asymp l^{q}, \quad p, q \ge 0.
\end{equation*}
Recall that $\VERT \s{K} \VERT_{1} = \sum_{l} \tau_{l}^{2} N(l) < \infty$ must hold, which is equivalent to the condition $2p > q+1$. This incorporates the familiar Sobolev-type condition $p > 1/2$, as discussed in \cref{ex:unit:Sobo}. The parameter $p$ quantifies the degree of smoothness, while $q$ typically scales with the intrinsic dimension $d$ of the domain, manifesting the curse of dimensionality in our asymptotic bounds. For general asymptotic results stated in terms of Schatten norms, without assuming specific decay or growth rates for $\tau_l$ and $N(l)$, see \cref{thm:gen:asym:mean}. In the case where $\s{P}_{I} = I$, the following theorem subsumes the known results for kernel regression \cite{cai2010nonparametric,caponera2022functional,caponera2025two}, matching the optimal convergence rates established in \cite{cai2011optimal}.

\begin{theorem}\label{thm:mean:asym:L2}
Let $M_{1}, M_{2}>0$. Assume (A) and that the eigenvalues $\tau_{l}^{2}$  of $\s{T}_{\tilde{K}}$ satisfies $\tau_{l} \asymp l^{-p}$ with multiplicities $N(l) \asymp l^{q}$ for some $p, q \ge 0$ with $2p > q+1$. If the decaying speed of the penalty parameter is $\nu_{n} \asymp [n \cdot \min(\overline{r},
\overline{r}^{\sigma^{2}})]^{-2p/(2p+q+1)}$, then it holds that 
\begin{align*}
    \lim_{D \rightarrow \infty} \limsup_{n \rightarrow \infty} \sup_{\b{P}_{\c{Y}} \in \c{S}_{1}(M_{1}, M_{2})} \b{P}_{\c{Y}} &\left[ \| \s{P} \hat{\m{\mu}}_{\nu_{n}}- \s{P} \m{\mu} \|_{\c{L}_{2} (\b{G} \times \Omega^{d-d_{r}}, W_{d_{r}, d})}^{2} > \right. \\
    &\hspace{2em} \left. D ((n \cdot \min(\overline{r},
    \overline{r}^{\sigma^{2}}))^{-\frac{2p}{2p+q+1}} + n^{-1}) \right] = 0.
\end{align*}   
\end{theorem}

We now unpack the implications of \cref{thm:mean:asym:L2}. First, the convergence rate for the estimated sinogram depends on the effective tilt numbers, \emph{not} on the locational sampling frequency. 
Intuitively, if only a single viewing angle is available, one cannot recover the object's structure nor its full sinogram across different orientations, no matter how densely that single view is sampled. Nevertheless, a careful look at the proof shows that higher sampling leads to sharper non-asymptotic bounds.

Additionally, \cref{thm:mean:asym:L2} reveals a \emph{phase transition} based on how the number of tilts compares to the number of functions $n$ \cite{cai2010nonparametric, caponera2022functional, cai2011optimal}. As $1/2 < 2p/(2p+q+1) < 1$, the convergence rate is always slower than $(n \cdot \min(\overline{r},
\overline{r}^{\sigma^{2}}))^{-1} + n^{-1}$, yet faster than $(n \cdot \min(\overline{r},
\overline{r}^{\sigma^{2}}))^{-1/2} + n^{-1}$. Then, we see two regimes depending on the tilt numbers: 
\begin{itemize}
\item Sparse Tilts: In the regime where $\min(\overline{r},
\overline{r}^{\sigma^{2}}) = O(n^{\frac{q+1}{2p}})$, the convergence rate is
\begin{equation*}
    \| \s{P} \hat{\m{\mu}}_{\nu_{n}}- \s{P} \m{\mu} \|_{\c{L}_{2} (\b{G} \times \Omega^{d-d_{r}}, W_{d_{r}, d})}^{2} = O_{\b{P}} \left( (\min(\overline{r},
    \overline{r}^{\sigma^{2}}))^{-2p/(2p+q+1)} \right).
\end{equation*}
\item Dense Tilts: In the regime where $\min(\overline{r},
\overline{r}^{\sigma^{2}}) \gg n^{\frac{q+1}{2p}}$, we have enough tilt information and achieve a parametric rate that no longer depends on the effective tilt numbers:
\begin{equation*}
    \| \s{P} \hat{\m{\mu}}_{\nu_{n}}- \s{P} \m{\mu} \|_{\c{L}_{2} (\b{G} \times \Omega^{d-d_{r}}, W_{d_{r}, d})}^{2} = O_{\b{P}} \left( n^{-1} \right).
\end{equation*}
\end{itemize}

Finally, the theorem demonstrates another phase transition, this one related to the heteroscedastic noise. When the noise function $r \mapsto \sigma_{r}^{2}$ is increasing, the convergence rate depends on the \emph{noise-weighted} effective tilt number ($\overline{r}^{\sigma^{2}}$), making the rate slower than in the homoscedastic case. This is reasonable, as collecting more tilts introduces more noise into the overall dataset. In the extreme case where the noise variance grows faster than linearly with the number of tilts, the convergence rate can actually degrade as more tilts are added. This result provides a theoretical basis for the practical trade-off between acquiring more viewing angles and the corresponding degradation in image quality. As we will show later, the covariance estimation is even more susceptible to this noise growth.

Note that the preceding analysis of the output sinogram error does not require the source regularity parameter $\kappa_{1}$. However, to analyze the convergence rate of the object itself in the input space, this parameter becomes necessary.

\begin{theorem}\label{thm:mean:asym:RKHS}
Let $0 \le \kappa_{1} \le 1/2$ and $M_{1}, M_{2}>0$. Assume (A) and that the eigenvalues $\tau_{l}^{2}$ of $\s{T}_{\tilde{K}}$ satisfies $\tau_{l} \asymp l^{-p}$ with multiplicities $N(l) \asymp l^{q}$ for some $p, q \ge 0$ with $2p > q+1$. If the decaying speed of the penalty parameter is $\nu_{n} \asymp (n \cdot \min(\overline{r},
\overline{r}^{\sigma^{2}}))^{-(\frac{2p+q+1}{2p} + 2\kappa_{1})^{-1}}$, then
\begin{align*}
    \lim_{D \rightarrow \infty} \limsup_{n \rightarrow \infty} \sup_{\b{P}_{\c{Y}} \in \c{S}_{1}(\kappa_{1}, M_{1}, M_{2})} \b{P} & \left[ \| \hat{\m{\mu}}_{\nu_{n}}-\m{\mu} \|_{\b{H}}^{2} > \right. \\
    &\hspace{2em} \left. D ( (n \cdot \min(\overline{r},
    \overline{r}^{\sigma^{2}}))^{- 2\kappa_{1}(\frac{2p+q+1}{2p} + 2\kappa_{1})^{-1}} + n^{-1}) \right] = 0.
\end{align*}
\end{theorem}

As mentioned, while setting $\kappa_{1} = 0$ in \cref{thm:mean:asym:RKHS} does not yield a meaningful consistency rate at the input level, once we have $\kappa_{1} > 0$, the convergence rate is bounded between $(n \cdot \min(\overline{r}, \overline{r}^{\sigma^{2}}))^{-\frac{2\kappa_{1}}{2\kappa_{1}+1}} + n^{-1}$ and $(n \cdot \min(\overline{r}, \overline{r}^{\sigma^{2}}))^{-\frac{2\kappa_{1}}{2\kappa_{1}+1/2}} + n^{-1}$ as $2p > q+1$.

Similar to the output-level analysis, \cref{thm:mean:asym:RKHS} also exhibits a phase transition with respect to the effective tilt numbers, though the boundary now occurs at $n^{(2\kappa_{1})^{-1} (\frac{2p+q+1}{2p})}$. In the sparse tilts regime where $\min(\overline{r},
\overline{r}^{\sigma^{2}}) = O( n^{(2\kappa_{1})^{-1}(\frac{2p+q+1}{2p})})$, we obtain
\begin{equation*}
    \| \hat{\m{\mu}}_{\nu_{n}}-\m{\mu} \|_{\b{H}}^{2} = O_{\b{P}} ( (n \overline{r})^{- 2\kappa_{1} (\frac{2p+q+1}{2p} + 2\kappa_{1})^{-1}} ),
\end{equation*}
In the dense tilts regime where $\min(\overline{r},
\overline{r}^{\sigma^{2}}) \gg n^{(2\kappa_{1})^{-1}(\frac{2p+q+1}{2p})}$, a parametric rate is achieved. The phase transition with respect to heteroscedastic noise is the same as before, with the boundary occurring when the noise function grows linearly.

\subsection{Asymptotic Behavior of the Covariance Estimator}\label{sec:asym:cov}
As mentioned in \cref{sec:notation}, we denote the tensor product between two elements in $\b{H} \otimes \b{H}$ by $\otimes_{2}$ when we reshuffle the fourth-order tensors. Define the following empirical and population functionals:
\begin{align}\label{eq:emp:tensor:UE}
    &\hat{\s{K}}^{\otimes} := \frac{\omega_{d_{r}, d}^{2}}{n} \sum_{i=1}^{n} \frac{1}{r_{i}(r_{i}-1)} \sum_{1 \le j \neq j' \le r_{i}} \frac{1}{s_{ij}s_{ij'}} \sum_{\substack{1 \le k \le s_{ij} \\ 1 \le k' \le s_{ij'}}} \varphi_{iJJ'}^{\otimes} \otimes_{2} \varphi_{iJJ'}^{\otimes}, \nonumber\\
    &\s{K}^{\otimes} := \b{E}[\hat{\s{K}}^{\otimes}] = \omega_{d_{r}, d}^{2} \b{E}[ (\varphi_{ijk} \otimes \varphi_{ij'k'}) \otimes_{2} (\varphi_{ijk} \otimes \varphi_{ij'k'})] \in \c{B}_{\infty}(\b{H} \otimes \b{H}), \, j \neq j'. 
\end{align}

Since $\varphi_{ijk}$ and $\varphi_{ij'k'}$ are independent for $j \neq j'$, it is important to note that 
\begin{align*}
    \s{K}^{\otimes} = \s{K} \otimes \s{K} = \sum_{l, l'=1}^{\infty} \tau_{ll'}^{2} \sum_{m, m'=1}^{N(l)} e^{\otimes}_{lm, l'm'}  \otimes_{2} e^{\otimes}_{lm, l'm'},
\end{align*}
where $\tau_{ll'} := \tau_{l} \tau_{l'}$ and $e^{\otimes}_{lm, l'm'} := e_{lm} \otimes e_{l'm'} \in \b{H} \otimes \b{H}$ forms a CONS for $\b{H} \otimes \b{H}$. Consequently, the s.p.d. operator $\s{K}^{\otimes}$ satisfies $\VERT\s{K}^{\otimes} \VERT_{\infty} = \VERT\s{K} \VERT_{\infty}^{2}$ and $\VERT(\s{K}^{\otimes})^{\theta} \VERT_{1} = \VERT\s{K}^{2 \theta} \VERT_{1}$ for $\theta \in (0, \infty)$
whenever either side is finite, which is always true for $\theta \ge 1/2$.


For any $j \neq j'$, since $\b{E} [\left. Z_{iJJ'} \right\vert \varphi] = \langle \m{\Gamma}, \varphi_{iJJ'}^{\otimes}  \rangle_{\b{H} \otimes \b{H}} $ by \eqref{eq:WW:cond:exp}, we have
\begin{align*}
    &\b{E} \left(Z_{iJJ'}- \langle G, \varphi_{iJJ'}^{\otimes}  \rangle_{\b{H} \otimes \b{H}} \right)^{2} = \b{E} \left(Z_{iJJ'}- \langle \m{\Gamma}, \varphi_{iJJ'}^{\otimes}  \rangle_{\b{H} \otimes \b{H}} \right)^{2} + \langle \s{K}^{\otimes} (\m{\Gamma}-G), \m{\Gamma}-G \rangle_{\b{H} \otimes \b{H}}.
\end{align*}
Therefore, the expectation of empirical risk functional in \eqref{eqn:emp:risk:cov} becomes
\begin{align*}
    L_{\eta}^{\otimes}(G)  
    =\b{E} \left[ \left(Z_{iJJ'}- \langle \m{\Gamma}, \varphi_{iJJ'}^{\otimes}  \rangle_{\b{H} \otimes \b{H}} \right)^{2} \right] + \langle \s{K}^{\otimes} (\m{\Gamma}-G), \m{\Gamma}-G \rangle_{\b{H} \otimes \b{H}} + \eta \|G\|^{2}_{\b{H} \otimes \b{H}}.
\end{align*}
Similar to \cref{sec:asym:mean}, both $\hat{L}_{\eta}^{\odot}$ and $L_{\eta}^{\otimes}$ admit unique minimizers, given by
\begin{equation*}
    \hat{\m{\Gamma}}_{\eta} := \argmin_{G \in \b{H} \otimes \b{H}} \hat{L}_{\eta}^{\odot}(G) = (\hat{\s{K}}^{\otimes} + \eta I)^{-1} \c{Z}^{\otimes}_{n}, \quad \bar{\m{\Gamma}}_{\eta} := \argmin_{G \in \b{H} \otimes \b{H}} L_{\eta}^{\otimes}(G) = (\s{K}^{\otimes} + \nu I)^{-1} \s{K}^{\otimes} \m{\Gamma},
\end{equation*}
where
\begin{align*}
    \c{Z}^{\otimes}_{n}:= \frac{\omega_{d_{r}, d}^{2}}{n} \sum_{i=1}^{n} \frac{1}{r_{i}(r_{i}-1)} \sum_{1 \le j \neq j' \le r_{i}} \frac{1}{s_{ij}s_{ij'}} \sum_{\substack{1 \le k \le s_{ij} \\ 1 \le k' \le s_{ij'}}} Z_{iJJ'} \varphi_{iJJ'}^{\otimes} \in \b{H} \otimes \b{H}.
\end{align*}
Similar to \eqref{eq:mean:asym:1st}, it follows that $\m{\Gamma} - \bar{\m{\Gamma}}_{\eta} = \eta (\s{K}^{\otimes} + \eta I)^{-1} \m{\Gamma}$.
Following the same reasoning as in \cref{sec:asym:mean}, we introduce the source regularity condition:
\begin{definition}[Source regularity]
For $0 \le \kappa_{2} \le 1/2$ and $M_{3}, M_{4}>0$, we say that a fourth-order $\b{H}$-valued process $\c{Y}$ belongs to the class $\c{S}_{2}(\kappa_{2}, M_{3}, M_{4})$ if
    \begin{equation*}
        \m{\Gamma} \in \{(\s{K}^{\otimes})^{\kappa_{2}} \delta^{\otimes}: \delta^{\otimes} \in \b{H} \otimes \b{H}, \|\delta^{\otimes}\|_{\b{H} \otimes \b{H}}^{2} \le M_{3}\}, \quad \b{E}[\|\c{Y}\|_{\b{H}}^{4}] \le M_{4}.
    \end{equation*}
When $\kappa_{2} = 0$, we simply denote $\c{S}_{2}(\kappa_{2}, M_{3}, M_{4})$ instead of $\c{S}_{2}(M_{3}, M_{4})$.
\end{definition}
Again, the Douglas factorization lemma \cite{douglas1966majorization} leads to
\begin{equation*}
    \m{\Gamma} \in \{(\s{K}^{\otimes})^{\kappa_{2}} \delta^{\otimes}: \delta^{\otimes} \in \b{H} \otimes \b{H}, \|\delta^{\otimes}\|_{\b{H} \otimes \b{H}}^{2} \le M_{3}\} \quad \Longleftrightarrow \quad M_{3} (\s{K}^{\otimes})^{2 \kappa_{2}} \succeq \m{\Gamma} \otimes_{2} \m{\Gamma}.
\end{equation*}

We formalize the intuition that the source condition for covariance estimation is stronger than that for mean estimation:
\begin{proposition}\label{prop:srce:cond}
Let $0 \le \kappa_{2} \le 1/2$ and $M_{3}, M_{4}>0$. Then 
\begin{equation*}
    \c{S}_{2}(\kappa_{2}, M_{3}, M_{4}) \subset \c{S}_{1}(\kappa_{2}, \sqrt{M_{3}}, \sqrt{M_{4}}).
\end{equation*}
\end{proposition}

As in \cref{sec:asym:mean}, we present our results under the assumption of polynomial decay of the eigenvalues and polynomial growth of their multiplicities, and we refer to \cref{thm:gen:asym:cov,cor:gen:asym:cov} in the Supplement for general results. In contrast to \cref{sec:asym:mean}, a key difference appears when handling heteroscedastic errors. The  rate for the covariance is more sensitive to noise and depends on the \emph{squared-noise-weighted effective tilt number}:
\begin{equation*}
    \overline{r}^{\sigma^{4}} := \left[ \frac{1}{n} \sum_{i=1}^{n} \frac{1}{r_{i}} \left(\frac{\sigma_{r_{i}}^{4}}{\sigma^{4}} \right) \right]^{-1}.
\end{equation*}
This term, which depends on the fourth power of the noise ($\sigma^{4}$), replaces the $\sigma^{2}$-dependent term used in the mean estimation, signaling a stronger penalty for noisy data. This is because noise is amplified to a greater degree when estimating higher-order moments.

\begin{theorem}\label{thm:cov:asym:L2}
Let $M_{3}, M_{4}>0$. Assume (A), and suppose that the eigenvalues $\tau_{l}^{2}$ of $\s{T}_{\tilde{K}}$ satisfies $\tau_{l} \asymp l^{-p}$ and the dimension of the corresponding eigenspace $N(l) \asymp l^{q}$ for some $p, q \ge 0$ with $2p > q+1$. If the penalty parameter satisfies $\eta_{n} \asymp (n \cdot \min(\overline{r},
\overline{r}^{\sigma^{4}})/ \log n)^{-(\frac{2p+q+1}{2p})^{-1}}$, then
\begin{align*}
    \lim_{D \rightarrow \infty} \limsup_{n \rightarrow \infty} \sup_{\b{P}_{\c{Y}} \in \c{S}_{2}(M_{3}, M_{4})} \b{P}_{\c{Y}} &\left[ \| (\s{P} \otimes \s{P}) (\hat{\m{\Gamma}}_{\eta_{n}}-\m{\Gamma}) \|_{\c{L}_{2} (\b{G} \times \Omega^{d-d_{r}}, W_{d_{r}, d})}^{2} > \right. \\
    &\hspace{6em} \left. D \left( \left( \frac{n \cdot \min(\overline{r},
    \overline{r}^{\sigma^{4}})}{\log n} \right)^{-\frac{2p}{2p+q+1}} + \frac{1}{n} \right) \right] = 0.
\end{align*}
Moreover, the covariance estimator $\hat{\m{\Sigma}}_{\eta_{n}} := \hat{\m{\Gamma}}_{\eta_{n}} - \hat{\m{\mu}}_{\eta_{n}} \otimes \hat{\m{\mu}}_{\eta_{n}}$ achieves the same uniform convergence rate as $\hat{\m{\Gamma}}_{\eta_{n}}$.
\end{theorem}

In \cref{thm:cov:asym:L2} above, the uniform convergence rate for $\m{\Gamma}$ or $\m{\Sigma}$ in the output $\c{L}_{2}$ space, like its counterpart \cref{thm:mean:asym:L2} for the mean, does not require a source regularity parameter ($\kappa_{2}$). However, as in \cref{thm:mean:asym:RKHS}, the convergence rate of the covariance operator itself in the input space $\b{H} \otimes \b{H}$ is not meaningful when $\kappa_{2} = 0$:

\begin{theorem}\label{thm:cov:asym:RKHS}
Let $0 \le \kappa_{2} \le 1/2$ and $M_{3}, M_{4}>0$. Assume (A), and suppose that the eigenvalues $\tau_{l}^{2}$ of $\s{T}_{\tilde{K}}$ satisfies $\tau_{l} \asymp l^{-p}$ and the dimension of the corresponding eigenspace $N(l) \asymp l^{q}$ for some $p, q \ge 0$ with $2p > q+1$. If the penalty parameter satisfies $\eta_{n} \asymp (n \cdot \min(\overline{r}, \overline{r}^{\sigma^{4}})/ \log n)^{-(\frac{2p+q+1}{2p} + 2\kappa_{2})^{-1}}$, then it holds that 
\begin{align*}
    \lim_{D \rightarrow \infty} \limsup_{n \rightarrow \infty} \sup_{\b{P}_{\c{Y}} \in \c{S}_{2}(\kappa_{2}, M_{3}, M_{4})} \b{P}_{\c{Y}} &\left[ \| \hat{\m{\Gamma}}_{\eta_{n}}-\m{\Gamma} \|_{\b{H} \otimes \b{H}}^{2} > \right. \\
    &\hspace{1em} \left. D \left( \left( \frac{n \cdot \min(\overline{r},
    \overline{r}^{\sigma^{4}})}{\log n} \right)^{- 2\kappa_{2}(\frac{2p+q+1}{2p} + 2\kappa_{2})^{-1}} + \frac{1}{n} \right) \right] = 0.
\end{align*}  
Moreover, the covariance estimator $\hat{\m{\Sigma}}_{\eta_{n}} := \hat{\m{\Gamma}}_{\eta_{n}} - \hat{\m{\mu}}_{\eta_{n}} \otimes \hat{\m{\mu}}_{\eta_{n}}$ achieves the same uniform convergence rate as $\hat{\m{\Gamma}}_{\eta_{n}}$.
\end{theorem}

Just as with  mean estimation in \cref{sec:asym:mean}, the asymptotic convergence rate in \cref{thm:cov:asym:L2,thm:cov:asym:RKHS} for the covariance depends on the effective tilt numbers ($\overline{r},
\overline{r}^{\sigma^{4}}$), not on the frequency of sampling at different locations within a projection. While more location samples can lead to sharper non-asymptotic bounds, they do not alter the overall uniform convergence rate.
We again observe a phase transition between sparse and dense sampling regimes:
\begin{itemize}
\item The rate in \cref{thm:cov:asym:L2} is always bounded between $(\min(\overline{r},
\overline{r}^{\sigma^{4}}) / \log n)^{-1} + n^{-1}$ and $(\min(\overline{r},
\overline{r}^{\sigma^{4}}) / \log n)^{-1/2} + n^{-1}$. As for \cref{thm:cov:asym:RKHS}, the  rate is always bounded between $(\min(\overline{r},
\overline{r}^{\sigma^{4}})/ \log n)^{-\frac{2\kappa_{2}}{2\kappa_{2}+1}} + n^{-1}$ and $(\min(\overline{r},
\overline{r}^{\sigma^{4}})/ \log n)^{-\frac{2\kappa_{2}}{2\kappa_{2}+1/2}} + n^{-1}$.

\item The critical boundaries where the regime switches from being limited by the effective tilt numbers (sparse tilts) to being limited by the number of particles (dense tilts) are:
\begin{equation*}
    \min(\overline{r},
    \overline{r}^{\sigma^{4}}) \asymp n^{\frac{q+1}{2p}} \log n, \quad \min(\overline{r},
    \overline{r}^{\sigma^{4}}) \asymp n^{(2\kappa_{2})^{-1}(\frac{2p+q+1}{2p})} \log n,
\end{equation*}
for \cref{thm:cov:asym:L2,thm:cov:asym:RKHS}, respectively.
\end{itemize}

The most significant difference from the mean estimation is the covariance estimator's heightened sensitivity to heteroscedastic noise:
\begin{itemize}
\item For the mean estimation in \cref{thm:mean:asym:L2,thm:mean:asym:RKHS}, the phase transition's critical boundary was where noise variance $\sigma_{r}^{2}$ grew linearly with the number of tilts $r$.
\item For the covariance estimation in \cref{thm:cov:asym:L2,thm:cov:asym:RKHS}, the boundary occurs when $\sigma_{r}^{2}$ grows with the square root of $r$.
\end{itemize}
This makes covariance estimation far more fragile. If the noise variance per image increases even moderately with the number of acquired tilts (i.e., faster than $\sqrt{r}$), adding more views will actually degrade the quality of the covariance estimate. This imposes a much stricter practical limit on the number of tilts that can be beneficially used.

\section{Nuisance Orientations}\label{sec:cryo}
In cryo-EM, each biomolecule $\c{Y}_{i}$ is flash-frozen in a random orientation $\m{R}_{i}$, unknown to the experimenter \cite{singer2018mathematics}. 
While the microscope applies known, controlled tilts $\tilde{\m{R}}_{ij}$ 
(if the imaging was conducted simultaneously across particles, then $\tilde{\m{R}}_{ij} \equiv \tilde{\m{R}}_{j}$) 
during the $j$-th image acquisition, the final viewing orientation is a composite:
\begin{equation}\label{eq:cryo:orient}
    \m{R}_{ij} = \tilde{\m{R}}_{ij} \m{R}_{i}
\end{equation}
Because the initial orientation $\m{R}_{i}$ is a latent variable, the absolute first and second moments of the particles are unidentifiable.
Fortunately, key structural features of biomolecules -- such as the icosahedral symmetry of spherical viruses \cite{zandi2004origin}, the pseudo-twofold symmetry of nucleosomes \cite{luger1997crystal}, and the tetrameric arrangement of aquaporins \cite{murata2000structural} -- are preserved under random orientations. Consequently, it is both reasonable and biologically motivated to focus on $\b{G}$-invariant statistics. Treating the initial orientations as nuisance parameters, we formally define these orientationally averaged quantities as:
\begin{align*}
    &\overline{\m{\mu}} := \b{E}_{\m{R}} [\rho(\m{R})\m{\mu}] \in \b{H}, \quad
    \overline{\m{\Gamma}} := \b{E}_{\m{R}} [\rho(\m{R}) \m{\Gamma} \rho(\m{R}^{-1})] \in \b{H} \otimes \b{H}, \\
    &\overline{\m{\Sigma}} := \b{E}_{\m{R}} [\rho(\m{R}) \m{\Sigma} \rho(\m{R}^{-1})] = \overline{\m{\Gamma}} - \overline{\m{\mu}} \otimes \overline{\m{\mu}} \in \b{H} \otimes \b{H},
\end{align*}
where $\m{R} \sim \b{P}_{\m{R}}$ is a uniform distribution on $\b{G}$ \cite{katsevich2015covariance, faraut2008analysis}. By construction, these averaged statistics are also $\b{G}$-invariant:
\begin{equation*}
    \overline{\m{\mu}}(\m{z}) = \overline{\m{\mu}}(R \m{z}), \quad \overline{\m{\Gamma}}(\m{z}, \m{z}') = \overline{\m{\Gamma}}(R \m{z}, R \m{z}'), \quad \overline{\m{\Sigma}}(\m{z}, \m{z}') = \overline{\m{\Sigma}}(R \m{z}, R \m{z}'), \quad R \in \b{G}, \, \m{z}, \m{z}' \in \Omega^{d}.
\end{equation*}
From an identifiability perspective, adopting this framework is statistically equivalent to assuming the random process $\c{Y}$ itself is $\b{G}$-invariant, i.e., $\c{Y} \stackrel{d}{\equiv} \rho(R) \c{Y}$ for any $R \in \b{G}$. This assumption simplifies the estimation procedure significantly: the specific unknown orientations $\m{R}_{i}$
become irrelevant, and all that matters are the known, relative orientations $\tilde{\m{R}}_{ij}$ between the tilted images.

To proceed with estimation, we define $\b{G}$-invariant feature maps and their tensor products by averaging over the orientations:
\begin{align*}
    \overline{\varphi} : &\b{G} \times \Omega^{d-d_{r}} \rightarrow \b{H}, (R, \m{x}) \mapsto 
    \b{E}_{\m{R}} [\rho(\m{R}) \varphi(R, \m{x})], \\
    \overline{\varphi \otimes \varphi} : &(\b{G} \times \Omega^{d-d_{r}}) \times (\b{G} \times \Omega^{d-d_{r}}) \rightarrow \b{H} \otimes \b{H}, \\
    &((R, \m{x}), (R', \m{x}')) \mapsto \b{E}_{\m{R}} [\rho(\m{R}) \circ (\varphi(R, \m{x}) \otimes \varphi(R', \m{x})) \circ \rho(\m{R}^{-1})],
\end{align*}
which inherits $\b{G}$-invariance from the Haar measure:
\begin{equation*}
    \overline{\varphi}(R, \m{x}) = \overline{\varphi}(I, \m{x}), \quad 
    \overline{\varphi \otimes \varphi}(R, \m{x}, R', \m{x}') = \overline{\varphi \otimes \varphi}(R (R')^{-1}, \m{x}, I, \m{x}').
\end{equation*}
Hence, we simply denote $\overline{\varphi}(\m{x})$ and $\overline{\varphi \otimes \varphi}(R, \m{x}, \m{x}')$, instead of $\overline{\varphi}(R, \m{x})$ and $\overline{\varphi \otimes \varphi}(R, \m{x}, I, \m{x}')$, respectively. 

\begin{remark}
Recall from \cref{prop:A1:check} that $\varphi(R, \m{x})(\m{z}) = \s{P}_{R} k_{\m{z}} (\m{x})$, hence
\begin{equation*}
    \overline{\varphi}(\m{x})(\m{z}) = \int_{\b{G}} \varphi(R, \m{x})(\m{z}) \rd R = \int_{\b{G}} \s{P}_{R} k_{\m{z}} (\m{x}) \rd R = \left[\s{P}_{I} \left( \int_{\b{G}} k_{R \m{z}} \rd R \right) \right] (\m{x}).
\end{equation*}
For the X-ray transform, this can be computed explicitly, see Proposition B.2. in \cite{yun2025computerized}:
\begin{align*}
    \overline{\varphi}(\m{x})(\m{z}) 
    = \frac{1}{|\b{S}^{d-1}|} \int_{\b{B}^{d} \backslash \|\m{x}\| \cdot \b{B}^{d}} \frac{K(\m{z}, \m{z}')}{\|\m{z}'\|^{d-2} (\|\m{z}'\| -\|\m{x}\|^{2})^{1/2}}  \rd \m{z}',
\end{align*}
which only depends on $\|\m{x}\|$ and $\|\m{z}\|$.
\end{remark}

Crucially, even though $\m{R}_{i} \in \b{G}$ are unknown, these $\b{G}$-invariant feature maps are still computable random elements from the known experimental parameters $\tilde{\m{R}}_{ij}$ and $\m{X}_{ijk}$:
\begin{align*}
    &\overline{\varphi}_{ijk} := \overline{\varphi}(\m{R}_{ij}, \m{X}_{ijk}) = \overline{\varphi}(\m{X}_{ijk}), \\
    &\overline{\varphi}_{iJJ'}^{\otimes} := \overline{\varphi \otimes \varphi}(\m{R}_{ij}, \m{X}_{ijk}, \m{R}_{ij'}, \m{X}_{ij'k'}) = \overline{\varphi \otimes \varphi}(\tilde{\m{R}}_{ij}, \m{X}_{ijk}, \tilde{\m{R}}_{ij'}, \m{X}_{ij'k'}).
\end{align*}

\begin{proposition}\label{prop:cryo:cond:exp}
For the observation model in \eqref{eq:obs:scheme} with nuisance orientations from \eqref{eq:cryo:orient}, the conditional expectations are given by:
\begin{align}\label{eq:cryo:cond:2nd}
    &\b{E}[Z_{ijk} \vert (\tilde{\m{R}}_{ij}, \m{X}_{ijk}) ] = \langle \m{\mu}, \overline{\varphi}_{ijk} \rangle_{\b{H}}, \nonumber \\
    &\b{E}[Z_{ijk} Z_{i'j'k'} \vert (\tilde{\m{R}}_{ij}, \m{X}_{ijk}) , (\tilde{\m{R}}_{i'j'}, \m{X}_{i'j'k'}) ] 
    =
    \begin{cases}
        \langle \m{\mu}, \overline{\varphi}_{ijk} \rangle_{\b{H}} \langle \m{\mu}, \overline{\varphi}_{i'j'k'} \rangle_{\b{H}}, &i \neq i', \\
        \langle \m{\Gamma}, \overline{\varphi}_{iJJ'}^{\otimes} \rangle_{\b{H} \otimes \b{H}} + \sigma_{r_{i}}^{2} \delta_{jj'} \delta_{kk'}, &i = i'.
    \end{cases}
\end{align}
Additionally, for any $\m{\mu} \in \b{H}$ and $\m{\Gamma} \in \b{H} \otimes \b{H}$, the inner products are equivalent to those with their $\b{G}$-invariant counterparts:
\begin{equation*}
    \langle \m{\mu}, \overline{\varphi}_{ijk} \rangle_{\b{H}} = \langle \overline{\m{\mu}}, \overline{\varphi}_{ijk} \rangle_{\b{H}}, \quad \langle \m{\Gamma}, \overline{\varphi}_{iJJ'}^{\otimes} \rangle_{\b{H} \otimes \b{H}} = \langle \overline{\m{\Gamma}}, \overline{\varphi}_{iJJ'}^{\otimes} \rangle_{\b{H} \otimes \b{H}}.
\end{equation*}
\end{proposition}

Similar to the standard case, if we define
\begin{align*}
    \overline{L}_{\nu}(f)
    = \frac{\omega_{d_{r}, d}}{n} \sum_{i=1}^{n} \frac{1}{r_{i}} \sum_{j=1}^{r_{i}} \frac{1}{s_{ij}} \sum_{k=1}^{s_{ij}}  \left(Z_{ijk}-\langle f, \overline{\varphi}_{ijk} \rangle_{\b{H}} \right)^{2} +\nu \|f\|^{2}_{\b{H}},
\end{align*}
we recover the representer theorem to estimate $\overline{\m{\mu}}$:
\begin{align*}
    \argmin_{f \in \b{H}} \overline{L}_{\nu}(f) \in \text{span} \{\overline{\varphi}_{ijk}: 1 \le i \le n, 1 \le j \le r_{i}, 1 \le k \le s_{ij} \},
\end{align*}
which is guaranteed to be a $\b{G}$-invariant function. The estimator's coefficients are found by solving a linear system with a modified Gram matrix, by replacing $\varphi_{ijk}$ with $\overline{\varphi}_{ijk}$ in \eqref{eq:mean:Gram}:
\begin{equation*}
    \overline{\m{\Phi}}:= \left[ \overline{\m{\Phi}}_{i_{1}j_{1}k_{1}, i_{2}j_{2}k_{2}}:= \frac{1}{\sqrt{r_{i_{1}}s_{i_{1}j_{1}}}} \frac{1}{\sqrt{r_{i_{2}}s_{i_{2}j_{2}}}} \langle \overline{\varphi}_{i_{1}j_{1}k_{1}}, \overline{\varphi}_{i_{2}j_{2}k_{2}} \rangle_{\b{H}} \right].
\end{equation*}
The estimation of the second moment $\overline{\m{\Gamma}}$ follows the exact same logic, replacing the corresponding quantities in \eqref{eq:tens:Gram} with their $\b{G}$-invariant versions, so we omit the detailed formulation.
Because this $\b{G}$-invariant framework depends only on relative orientations, consistent estimation is feasible, and all asymptotic results from \cref{sec:asym} apply directly.

\section{Simulation Study}\label{sec:simul}
We set $d=2$ and generate $n=100$ random functions. To address low-dose tomography setups \cite{bhagya21, sanaat2020projection, dong19, tong20, simeoni19acoustic}, we sample $r=5$ angles for each function, and $s=10$ locations for each angle.
In homogeneous inverse problems, the 2D Shepp-Logan phantom -- formed by the overlap of several ellipses -- serves as a standard benchmark image for tomographic simulations. Analogously, we generate random functions by superimposing several fixed ellipsoidal-shaped bump functions, each modulated by a random intensity, as illustrated in \cref{fig:rand:phantoms}. Specifically, let $Q = 6$ denote the number of bump functions, each characterized by a location parameter $\m{l}_{q} \in \b{B}^{2}$, and a dispersion matrix $\m{S}_{q} \in \b{R}^{2 \times 2}$. We then define $n$ realizations of random functions as:
\begin{equation*}
    \c{Y}_{i} (\m{z}) = \sum_{q=1}^{Q} \xi_{q}^{i} \Psi[(\m{z} - \m{l}_{q})^{\top} \m{S}_{q}^{-1} (\m{z} - \m{l}_{q})], \quad \m{z} \in \b{B}^{2}, \, 1 \le i \le n,
\end{equation*}
where $\Psi$ is a compactly supported radial function
\begin{equation*}
    \Psi(r) = \frac{\tanh (5 (1 - r))^{2}}{\tanh (5 (1 - r))^{2} + \tanh (5 r)^{2}}  \cdot I(r \le 1).
\end{equation*}
For intensity, the i.i.d. random coefficients $\boldsymbol{\xi}^{i} = (\xi_{1}^{i}, \dots, \xi_{Q}^{i})^{\top}$ are drawn from a multivariate normal distribution $\c{N}_{Q}(\m{m}, \m{C})$ with mean vector $\m{m} \in \b{R}^{Q}$ and covariance matrix $\m{C} \in \b{R}^{Q \times Q}$. 
Consequently, the mean and covariance functions are given by
\begin{align*}
    &\m{\mu} (\m{z}) = \sum_{q=1}^{Q} \m{m}_{q} \Psi[(\m{z} - \m{l}_{q})^{\top} \m{S}_{q}^{-1} (\m{z} - \m{l}_{q})], \quad \m{z} \in \b{B}^{2}, \\
    &\m{\Sigma} (\m{z}, \m{z}') = \sum_{q, q'=1}^{Q} \m{C}_{qq'} \Psi[(\m{z} - \m{l}_{q})^{\top} \m{S}_{q}^{-1} (\m{z} - \m{l}_{q})] \Psi[(\m{z}' - \m{l}_{q'})^{\top} \m{S}_{q'}^{-1} (\m{z}' - \m{l}_{q'})], \quad \m{z}, \m{z}' \in \b{B}^{2}.
\end{align*}


\begin{figure}[h]
\centering
\includegraphics[width=.8\textwidth, height=7cm]{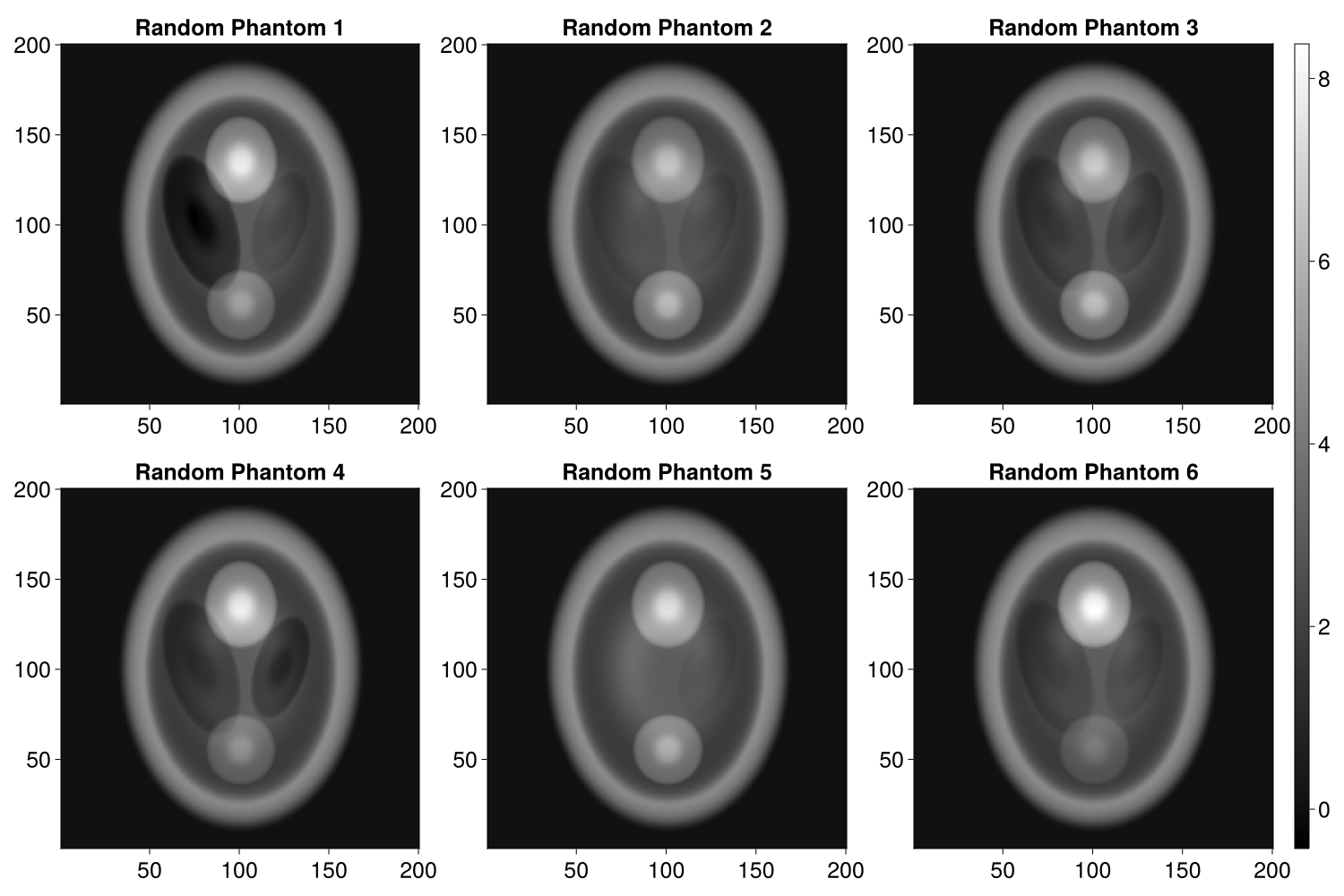}
\caption{Six realizations out of random phantoms, composed of ellipsoidal bump functions with random intensities.
}
\label{fig:rand:phantoms}
\end{figure}


Since $d = 2$, the $k$-plane, Radon, and X-ray transforms all coincide, which is the forward operator used to generate the observed data in \cref{fig:hetero:sino}, where we take the sampling distributions $\b{P}_{\m{R}}$ and $\b{P}_{\m{X}}$ to be uniform on $SO(2)$ and $[-1, +1]$, respectively, with the noise level $\sigma = 0.1$.
In two dimensions, $SO(2)$ is isomorphic to the circle $\b{R}/(2 \pi \b{Z}) $:
\begin{equation*}
    \phi \in \b{R}/(2 \pi \b{Z}) 
    \ \cong \ \m{E}(\phi)= \begin{pmatrix}
    \cos \phi & -\sin \phi \\
    \sin \phi & \cos \phi
    \end{pmatrix} \in SO(2).
\end{equation*}
Since the transform satisfies \emph{antipodal} symmetry, i.e., $\s{P} f (\m{E}(\phi), x) = \s{P} f (\m{E}(\phi + \pi), -x)$, instead of sampling from $\b{P}_{\m{R}}$, we equivalently sample $\phi \sim \mathrm{Unif}[0, \pi)$, which avoids redundancy in data acquisition.

\begin{figure}[h]
\centering
\includegraphics[width=\textwidth, height=4.5cm]{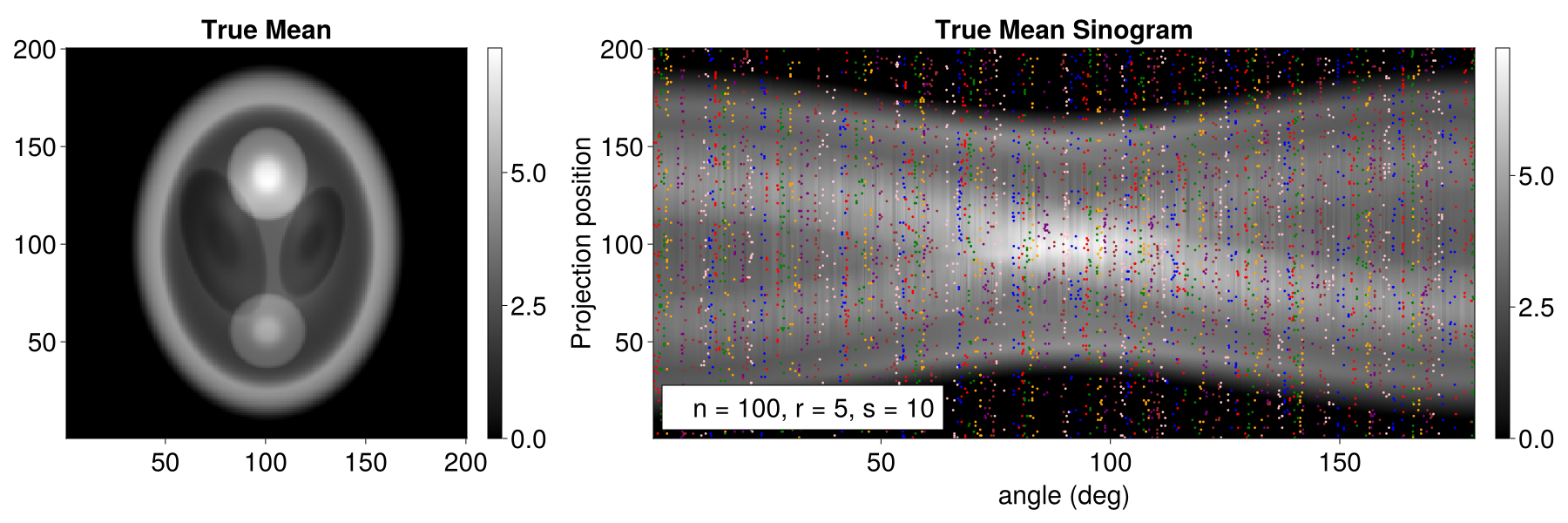}
\caption{The mean function $\m{\mu}$ is depicted on the left panel, with the corresponding sinogram $\s{P} \m{\mu}$ is shown on the right panel in grayscale. Overlaid are colored dots representing i.i.d. pairs of viewing angle and positions $(\boldsymbol{\phi}_{ij}, \m{X}_{ijk}) \sim \mathrm{Unif}([0, \pi) \times [-1, +1])$, with seven different colors indicating different random functions.}
\label{fig:hetero:sino}
\end{figure}

For the estimation, we employ the Gaussian kernel 
\begin{equation*}
    K_{\gamma}(\m{z}_{1}, \m{z}_{2}) = \exp(-\gamma \|\m{z}_{1}-\m{z}_{2}\|^{2}), \quad \gamma >0, \, \m{z}_{1}, \m{z}_{2} \in \b{B}^{2}.
\end{equation*}
with bandwidth parameter $\gamma =2^{8}$. The induced kernel $\tilde{K}$ used to construct the mean Gram matrix $\m{\Phi}$ can be computed explicitly; see \cite{yun2025computerized} for details. For both the mean and second-moment estimations shown in \cref{fig:second}, we set the regularization parameters to $\nu = \eta = 2^{-8}$ for the representer theorem in \cref{thm:repre:mean}, using the implementation detailed in \cref{sec:num:consid}.

\begin{figure}[h]
\centering
\includegraphics[width=.8\textwidth, height=6.5cm]{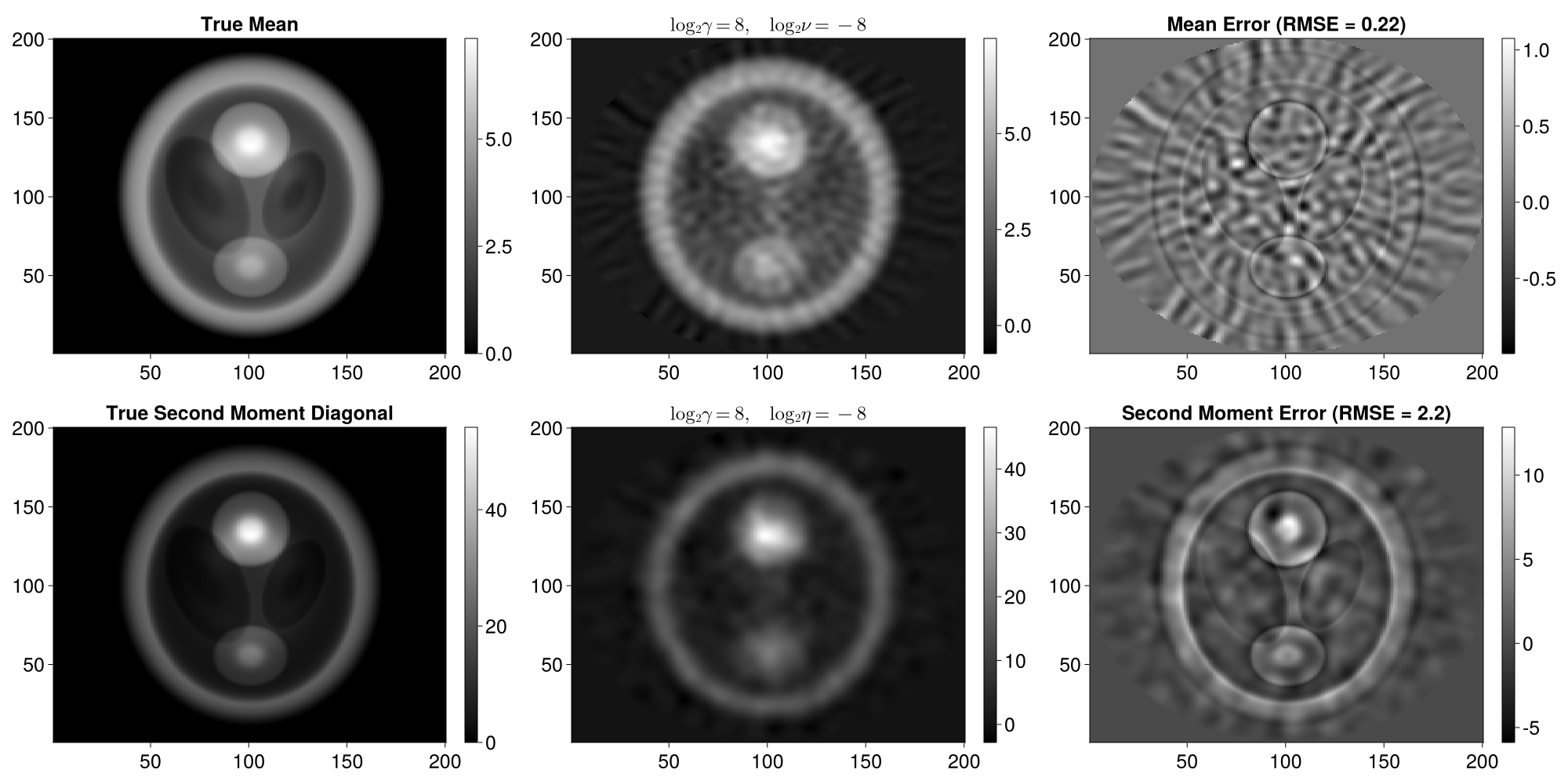}
\caption{Top row: the ground truth and estimated mean function $\m{\mu}(\m{z})$. Bottom row: the diagonal of the second-moment function $\m{\Gamma}(\m{z}, \m{z})$, for $\m{z} \in \b{B}^{2}$.  
Left column: ground truth. Middle column: estimators $\hat{\m{\mu}}_{\nu}(\m{z})$ and $\hat{\m{\Gamma}}_{\eta}(\m{z}, \m{z})$ computed using the Gaussian kernel $K_{\gamma}$. Right column: pointwise absolute error between the estimates and ground truth, with reported Root Mean Squared Error (RMSE) for quantitative comparison.}
\label{fig:second}
\end{figure}

After obtaining $\hat{\m{\mu}}_{\nu}$ and $\hat{\m{\Gamma}}_{\eta}$, we perform kernel PCA using \cref{thm:fpca}:
\begin{equation*}
    \hat{\m{\Sigma}}_{\nu, \eta} = \hat{\m{\Gamma}}_{\eta} - \hat{\m{\mu}}_{\nu} \otimes \hat{\m{\mu}}_{\nu} = \sum_{l=1}^{\mathrm{rank}(\hat{\m{\Sigma}}_{\nu, \eta})} \hat{\lambda}_{l} \ \hat{\psi}_{l} \otimes \hat{\psi}_{l} \in \b{H} \otimes \b{H},
\end{equation*}
where each principal component $\hat{\psi}_{l} \in \b{H}$ is an orthonormal eigenfunction corresponding to the eigenvalue $\hat{\lambda}_{l} \in \b{R}$. The results are illustrated in \cref{fig:fpca}, where we observe that the first principal component effectively captures the dominant covariance structure in our simulation. 

\begin{figure}[h]
\centering
\includegraphics[width=\textwidth, height=7cm]{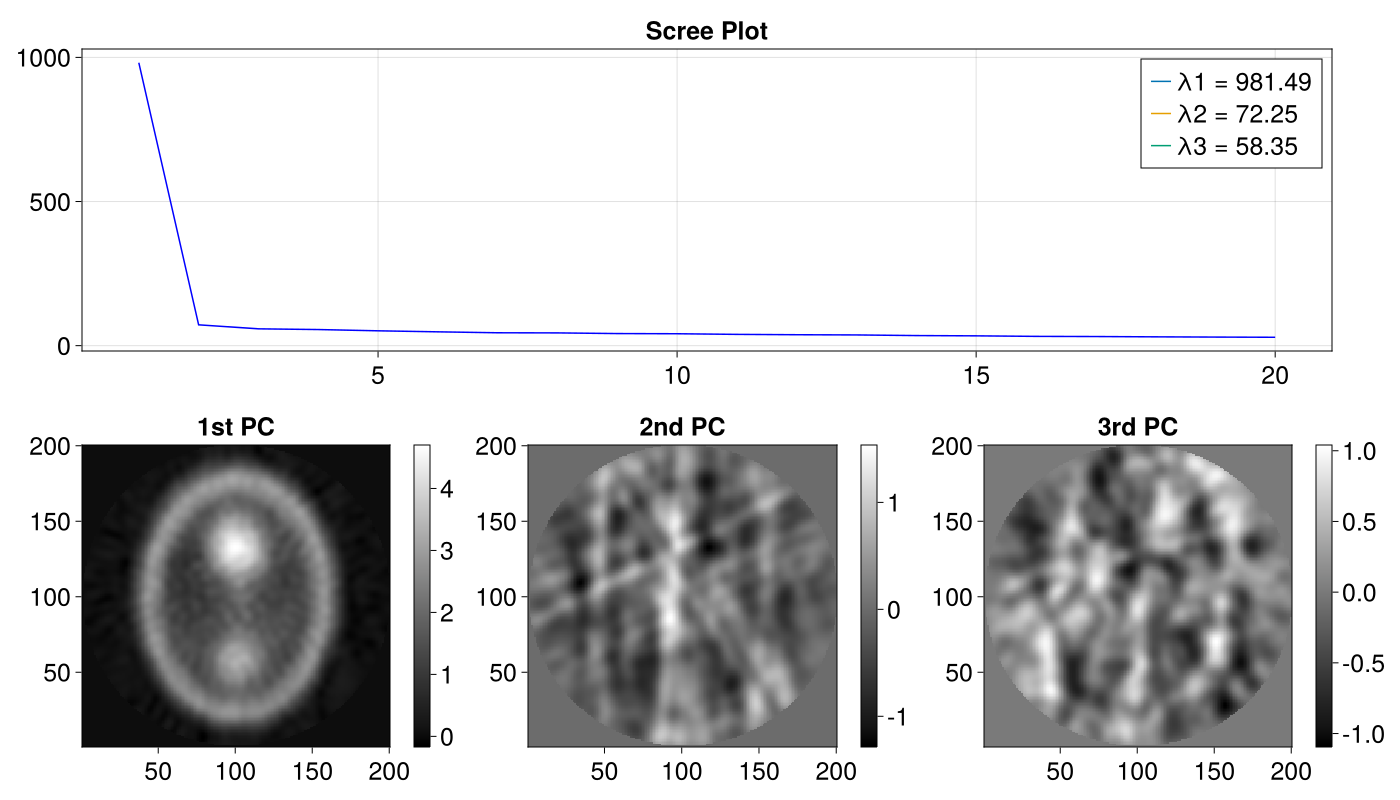}
\caption{Top panel: Scree plot of the first 20 eigenvalues of the estimated covariance operator. A clear spectral gap is observed after the leading eigenvalue, indicating that the first principal component dominates the variability. Bottom panels: The first three principal components of the covariance operator.}
\label{fig:fpca}
\end{figure}


\section{Conclusion}
Our work combines methods from inverse problems \cite{hanke2017taste, kaipio2006statistical, engl1996regularization} and FDA \cite{hsing2015theoretical, ramsay2002applied} to provide statistical foundations for the central challenge of continuous heterogeneity. Our framework provides a consistent, theoretically-grounded solution by using the representer theorems in a RKHS setup, which reduces the problem to a finite-dimensional regression made practical by our efficient \texttt{TReK} algorithm to overcome the limitations of traditional FBP \cite{natterer2001mathematics}. By enabling the principled estimation of covariance, our method facilitates the quantification of conformational variability from pooled data, making it especially valuable in low-dose, high-noise imaging regimes. The framework encompasses several tomographic modalities -- including electron, photoacoustic, and optical tomography \cite{bhagya21, sanaat2020projection, dong19, tong20, simeoni19acoustic} -- and is supported by rigorous asymptotic guarantees.





\begin{supplement}
\stitle{A. Sinogram Operators}
\sdescription{It shows how our theoretical results, based on orientational duality, extend beyond the X-ray transform to general linear inverse problems.}
\end{supplement}
\begin{supplement}
\stitle{B. Implementation}
\sdescription{It provides detailed implementation notes for our estimators. It focuses on the Krylov methods used to efficiently solve the large linear system arising from the covariance estimation.}
\end{supplement}
\begin{supplement}
\stitle{C. Proofs}
\sdescription{It contains the detailed mathematical proofs for all \emph{theorems, propositions, and lemmas}.}
\end{supplement}
\begin{supplement}
\stitle{D. Index}
\sdescription{It provides a comprehensive \emph{list of notations}, serving as a quick reference.}
\end{supplement}

\bibliographystyle{imsart-number}
\bibliography{bibliography}

\newpage

\section*{}

\renewcommand{\thesection}{\Alph{section}}
\renewcommand{\theequation}{S.\arabic{equation}}

\setcounter{page}{1}

\setcounter{equation}{0}
\setcounter{section}{0}
\setcounter{subsection}{0}

\centerline{\bf \large Supplementary Material}


\bigskip

\section{Sinogram Operators}\label{sec:sino:op}
Throughout the Supplement, we consider our problem setup and the associated proofs in a more general setup. First, we let an input $f$ to be a $d$-dimensional functions defined on a shared support $\Omega^{d}$, a compact subset of $\b{R}^{d}$, and output $g$ as $(d-d_{r})$-dimensional functions defined on a shared compact support $\Omega^{d-d_{r}} \subset \b{R}^{d-d_{r}}$, where $0 \leq d_{r} \leq d$. For simplicity, we consider the Lie group $\b{G}$ to be compact, so that $\b{G}$ ensures the uniqueness of the normalized Haar measure. Additionally, we conceptualize the set of orientations as a Lie group acting on $\Omega^{d}$ from the left, denoted as $\b{G}$. We denote the left action of $R \in \b{G}$ by $\rho(R): \Omega^{d} \rightarrow \Omega^{d}, \m{z} \mapsto R \m{z}$. The forward operator $\s{P}_{R}$ at orientation $R \in \b{G}$  generates a $(d-d_{r})$-dimensional projection image of the $d$-dimensional input function.

We present what type of forward operators $\{\s{P}_{R}: R \in \b{G}\}$ across orientations are considered in the context of inverse problems. 
While the forward operator $f \mapsto \s{P}_{R} f$ could be well-defined on different domains or ranges, the analytical properties of the operator $\s{P}_{R}$ depend on these choices, affecting the quality of reconstructed image \cite{hanke2017taste, yun2025computerized}. As a starting point, we first assume that $\s{P}_{R}$ is a well-defined continuous linear operator between some weighted $\c{L}_{2}$ spaces for any orientation $R \in \b{G}$. Then, we restrict the domain by an RKHS, and discuss how this alteration changes the properties of the forward operator. 
Reliable reconstruction necessitates sectional images from multiple orientations, and they should be systematically intertwined, leading to the following definition:

\begin{definition}[Sinogram Operator]\label{def:sino:op}
A collection of forward operators $\s{P} = \{\s{P}_{R} : R \in \b{G}\}$ is said to be a \emph{sinogram} operator if
\begin{enumerate}
\item $\s{P}_{I} \in \c{B}_{\infty}(\c{L}_{2} (\Omega^{d}, W_{0, d}), \c{L}_{2} (\Omega^{d-d_{r}}, W_{d_{r}, d}))$ is a continuous linear operator.
\item If $f: \Omega^{d} \rightarrow \b{R}$ is a continuous function, then $\s{P}_{I}f: \Omega^{d-d_{r}} \rightarrow \b{R}$ is also a continuous function.
\item For any $R \in \b{G}$, it holds that $\s{P}_{R} = \s{P}_{I} \circ \rho(R)$.
\end{enumerate}
In this case, we identify the sinogram operator $\s{P}: \c{L}_{2} (\Omega^{d}, W_{0, d}) \rightarrow \c{L}_{2} (\b{G} \times \Omega^{d-d_{r}}, W_{d_{r}, d}))$ through $\s{P} f (R, \m{x}) := \s{P}_{R} f (\m{x})$.
\end{definition}

$\s{P}_{R} = \s{P}_{I} \circ \rho(R)$ dictates the dual perspective—rotating the viewing orientation is equivalent to rotating the input function in the opposite direction. This leads to $\s{P}^{*}_{R} = \rho(R^{-1}) \circ \s{P}^{*}_{I}$ and that the sinogram operator $\s{P}$ is continuous, with its adjoint given by
\begin{equation*}
    \s{P}^{*} g = \int_{\b{G}} \s{P}^{*}_{R} g(R, \cdot) \rd R, \quad g \in \c{L}_{2} (\b{G} \times \Omega^{d-d_{r}}, W_{d_{r}, d}),
\end{equation*}
in the sense of Bochner integration \cite{da2014stochastic, conway2019course}.

\begin{remark}\label{rmk:kplane}
The $k$-plane transform is a special case of \cref{def:sino:op} with $\b{G}=SO(d)$, $\Omega^{d} = \b{B}^{d}$, $\Omega^{d-d_{r}}=\b{B}^{d-d_{r}}$, and the weight functions $W_{0, d} = 1$ and $W_{d_{r}, d} = (1-\|x\|^{2})^{-d_{r}/2}$. 
However, if we let $\Omega^{d} = \b{R}^{d}$, then $\s{P}_{I}$ is no longer continuous \cite{natterer2001mathematics}, which shows that the compactness of $\Omega^{d}$ is necessary. Later on, the compactness of $\Omega^{d}$ ensures the Mercer decomposition in \cref{sec:Mercer:inv}.
\end{remark} 

We next explain how the Lie group $\b{G}$ induces unitary representations on the input and output function spaces, and how these are connected through duality.

\begin{definition}[Representation]\label{def:repre}
Let $\b{G}$ be a Lie group, and $\c{H}$ be a Hilbert space. A representation $\rho$ of $\b{G}$ on $\c{H}$ is a map $\rho: \b{G} \rightarrow \c{B}_{\infty}(\c{H}), R \mapsto \rho(R)$ such that
\begin{enumerate}
    \item $\rho(I)$ is the identity map and $\rho(R_{1} R_{2}) = \rho(R_{1}) \rho(R_{2})$ for any $R_{1}, R_{2} \in \b{G}$.
    \item (strong operator topology (SOT)) For any $f \in \c{H}$, the map $\b{G} \rightarrow \c{H}, R \mapsto \rho(R) f$ is continuous.
\end{enumerate}
$\rho$ is said to be \emph{unitary} if $\| \rho(R) f \|_{\c{H}} = \| f \|_{\c{H}}$ for any $f \in \c{H}$ and $R \in \b{G}$.
A subspace $\c{H}_{0} \subset \c{H}$ is said to be $\rho$-invariant if $\rho(R) \c{H}_{0} = \c{H}_{0}$ for any $R \in \b{G}$.
Given two representations $(\rho_{1}, \c{H}_{1})$ and $(\rho_{2}, \c{H}_{2})$ of $\b{G}$, a continuous linear operator $\s{P} \in \c{B}_{\infty}(\c{H}_{1}, \c{H}_{2})$ is said to be \emph{intertwining} if $\s{P} \circ \rho_{1}(R) = \rho_{2}(R) \circ \s{P}$ for any $R \in \b{G}$.
\end{definition}

If $\dim \c{H} = \infty$, then the Banach space $\c{B}_{\infty}(\c{H})$ equipped with the norm topology is not separable \cite{da2014stochastic}, hence the SOT is a suitable topology to work with the representation theory, which is weaker than the norm topology \cite{faraut2008analysis, stone1932one}.

\begin{proposition}\label{prop:unit:repn}
Let $(\rho, \c{L}_{2} (\Omega^{d}, W_{0, d}) )$ and $(\rho, \c{L}_{2} (\b{G} \times \Omega^{d-d_{r}}, W_{d_{r}, d}) )$ be defined as in \eqref{eq:left:reg:def1} and \eqref{eq:left:reg:def2}, respectively. Then, they are both unitary representations of $\b{G}$. Moreover, the sinogram operator $\s{P}: \c{L}_{2} (\Omega^{d}, W_{0, d}) \rightarrow \c{L}_{2} (\b{G} \times \Omega^{d-d_{r}}, W_{d_{r}, d})$ is an intertwining operator.
\end{proposition}
\begin{proof}
Denote by $\c{H}_{1} = \c{L}_{2} (\Omega^{d}, W_{0, d})$ and $\c{H}_{2} = \c{L}_{2} (\b{G} \times \Omega^{d-d_{r}}, W_{d_{r}, d})$. For $\rho: \b{G} \rightarrow \c{B}_{\infty}(\c{H}_{1})$ given in \eqref{eq:left:reg:def1}, it is clear that $\rho(I)$ is the identity map and $\rho(R_{1} R_{2}) = \rho(R_{1}) \rho(R_{2})$ for any $R_{1}, R_{2} \in \b{G}$. Also, they are unitary since the weight $W_{0, d}$ is $\b{G}$-invariant. Therefore, it remains to show that $\rho: \b{G} \rightarrow \c{B}_{\infty}(\c{H}_{1})$ is continuous at $I \in \b{G}$ with respect to the strong operator topology. Let $f \in \c{H}_{1}$ and $\varepsilon>0$ be given. Since $C(\Omega^{d})$ is dense in $\c{H}_{1}$ by the Stone-Weierstrasse theorem, there is some $\tilde{f} \in C(\Omega^{d})$ such that $\| f - \tilde{f} \|_{\c{H}} < \varepsilon$. Also, $\tilde{f}$ is uniformly continuous as $\Omega^{d}$ is compact, so there exists some $\delta > 0$ such that $\|R - I\|<\delta$ implies
\begin{equation*}
    \sup_{\m{z} \in \Omega^{d}} \left| \tilde{f} (R^{-1} \m{z}) - \tilde{f} (\m{z}) \right| < \varepsilon.
\end{equation*}
Thus, whenever $\|R - I\|<\delta$, 
\begin{align*}
    \|\rho(R) \tilde{f} - \tilde{f} \|_{\c{H}}^{2} = \int_{\Omega^{d}} [\tilde{f}(R^{-1} \m{z}) - \tilde{f}(\m{z})]^{2} W_{0, d}(\m{z}) \rd \m{z} \le \omega_{0, d} \cdot \varepsilon^{2},
\end{align*}
where $\omega_{0, d}$ is defined in \eqref{eq:normal:const}, which leads to
\begin{equation*}
    \|\rho(R) f - f \|_{\c{H}} \le \|\rho(R) (f - \tilde{f}) \|_{\c{H}} + \|\rho(R) \tilde{f} - \tilde{f} \|_{\c{H}} + \|\tilde{f} - f \|_{\c{H}} < \varepsilon (2 + \sqrt{\omega_{0, d}}).
\end{equation*}
This demonstrates that $\rho: \b{G} \rightarrow \c{B}_{\infty}(\c{H}_{1})$ is a unitary representation, and one can also prove in the same way that $\rho: \b{G} \rightarrow \c{B}_{\infty}(\c{H}_{2})$ is also a unitary representation. Finally, for any $f \in \c{H}_{1}$, we have
\begin{align*}
    [\s{P} \rho(R_{1}) f](R_{2}, \cdot) &= \s{P}_{R_{2}} \circ \rho(R_{1}) f = \s{P}_{I} \rho(R_{2}) \rho(R_{1}) f = \s{P}_{I} \rho(R_{2} R_{1}) f \\
    &= \s{P}_{R_{2} R_{1}} f = \s{P} f(R_{2} R_{1}, \cdot) = [\rho(R_{1}) \s{P} f](R_{2}, \cdot),
\end{align*}
thus $\s{P}: \c{H}_{1} \rightarrow \c{H}_{2}$ is intertwining.
\end{proof}

Even when the sinogram operator $\s{P}$ is injective, $\s{P}^{\dagger}$ is unbounded unless $\c{R}(\s{P})$ is closed in $\c{L}_{2} (\b{G} \times \Omega^{d-d_{r}}, W_{d_{r}, d})$. This implies that even a small perturbation could have a fatal impact on our reconstruction. The $k$-plane transform is compact with infinite-dimensional range, hence $\s{P}^{\dagger}$ is necessarily unbounded \cite{conway2019course}. 
Moreover, with only finitely many orientations $R_{1}, \dots, R_{r} \in \b{G}$, it is often the case that $\cap_{j=1}^{r} \c{N}(\s{P}_{R_{j}}) \neq \{0\}$—as is true for the $k$-plane transform \cite{natterer2001mathematics}. This underscores the inherent ill-posedness: reconstruction from limited views admits infinitely many solutions \cite{hanke2017taste, hadamard1902problemes, natterer1980sobolev}. Regularization techniques, such as truncated SVD \cite{izen1988series, maass1987x}, are therefore essential. Since the sinogram operator $\s{P}$ leads to an ill-posed inverse problem, we seek kernels that induce a reproducing property compatible with $\s{P}$.

\begin{definition}[Sinogram Kernel]\label{def:sino:ker}
Let $\s{P}$ be a sinogram operator as in \cref{def:sino:op}. A Mercer kernel $K: \Omega^{d} \times \Omega^{d} \rightarrow \b{R}$ on $\Omega^{d}$ is said to be a sinogram kernel if
\begin{enumerate}
\item $K$ is $\b{G}$-invariant, i.e. 
\begin{equation}\label{eq:G-inv}
    K(\m{z}_{1}, \m{z}_{2}) = K(R \m{z}_{1}, R \m{z}_{2}), \quad R \in \b{G}, \m{z}_{1}, \m{z}_{2} \in \Omega^{d}.
\end{equation}
\item At orientation $I \in \b{G}$, there is some constant $B_{\m{x}} > 0$ for every $\m{x} \in \Omega^{d-d_{r}}$ such that 
\begin{equation}\label{eq:sino:kernel}
    |\s{P}_{I} f ( \m{x})| \le B_{\m{x}} \|f\|_{\b{H}}, \quad f \in \b{H}(K).
\end{equation}
\end{enumerate}
\end{definition}

Note that \eqref{eq:G-inv} is equivalent to $\rho(R) k_{\m{z}} = K(\m{z}, R^{-1} \cdot) = K (R \m{z}, \cdot) = k_{R \m{z}}$,
thus $(\rho, \b{H})$ is a unitary representation of $\b{G}$. Indeed, \eqref{eq:sino:kernel} holds for any orientation $R \in \b{G}$:
\begin{equation}\label{eq:bdd:eval}
    |\s{P}_{R} f ( \m{x})| = |\s{P}_{I} [\rho(R) f] ( \m{x})| \le B_{\m{x}} \|\rho(R) f\|_{\b{H}} = B_{\m{x}} \|f\|_{\b{H}}, \quad f \in \b{H}.
\end{equation}
Due to the Riesz representation theorem, \eqref{eq:bdd:eval} is equivalent to the fact that,
for each $R \in \b{G}$ and $\m{x} \in \Omega^{d-d_{r}}$, there is some $\varphi(R, \m{x}) \in \b{H}$ with $\| \varphi(R, \m{x}) \|_{\b{H}} \le B_{\m{x}}$ with
\begin{equation}\label{eq:eval:ftnal}
    \langle \varphi(R, \m{x}), f \rangle_{\b{H}} = \s{P} f (R, \m{x}), \quad f \in \b{H}.
\end{equation}
Using the unitary representation $(\rho, \b{H})$, it is immediate that
\begin{equation}\label{eq:gen:cond:indep}
    \varphi(R, \m{x}) = \rho(R^{-1}) \varphi(I, \m{x}), \quad R \in \b{G}, \m{x} \in \Omega^{d-d_{r}}.
\end{equation}

Since $K$ is a Mercer kernel, the feature map $\varphi(R, \m{x}) \in \b{H}(K)$ is a continuous function. The following proposition offers a description how to evaluate this map, which is a fundamental building block to obtain representer theorems in \cref{sec:scheme}: 

\begin{proposition}\label{prop:A1:check}
Let $\s{P}$ be a sinogram operator and $K: \Omega^{d} \times \Omega^{d} \rightarrow \b{R}$ be a sinogram kernel as in \cref{def:sino:ker}. 
Then $\varphi(R, \m{x}): \Omega^{d} \rightarrow \b{R}$ in \eqref{eq:eval:ftnal} is characterized by
\begin{equation*}
    \varphi(R, \m{x})(\m{z}) = \s{P}_{R} k_{\m{z}} (\m{x}), \quad \m{x} \in \Omega^{d-d_{r}}, \m{z} \in \Omega^{d}, R \in \b{G},
\end{equation*}
and for any $R, R' \in \b{G}$ and $\m{x}, \m{x}' \in \Omega^{d-d_{r}}$,
\begin{equation}\label{eq:Gram:elem}
    \langle \varphi(R, \m{x}), \varphi(R', \m{x}') \rangle_{\b{H}} = \langle \varphi(I, \m{x}), \varphi(R'R^{-1}, \m{x}') \rangle_{\b{H}} = [\s{P}_{I} \varphi(R'R^{-1}, \m{x}')](\m{x}).
\end{equation}
Additionally, $\varphi: \b{G} \times \Omega^{d-d_{r}} \rightarrow \b{H}, (R, \m{x}) \mapsto \varphi(R, \m{x})$ is a continuous mapping.
\end{proposition}
\begin{proof}
First, note that
\begin{equation*}
    \varphi(R, \m{x})(\m{z}) = \langle \varphi(R, \m{x}), k_{\m{z}} \rangle_{\b{H}} =\s{P} k_{\m{z}}(R, \m{x}) = \s{P}_{R} k_{\m{z}} (\m{x}), \quad \m{x} \in \Omega^{d-d_{r}}, \m{z} \in \Omega^{d}, R \in \b{G}.
\end{equation*}
Additionally, due to \eqref{eq:gen:cond:indep}, we have
\begin{equation*}
    \langle \varphi(R, \m{x}), \varphi(R', \m{x}') \rangle_{\b{H}} = \langle \varphi(I, \m{x}), \varphi(R'R^{-1}, \m{x}') \rangle_{\b{H}} = [\s{P}_{I} \varphi(R'R^{-1}, \m{x}')](\m{x}).
\end{equation*}
To show that $\varphi: \b{G} \times \Omega^{d-d_{r}} \rightarrow \b{H}, (R, \m{x}) \mapsto \varphi(R, \m{x})$ is continuous, it suffices to show the continuity with $R=I$ fixed due to \cref{prop:unit:repn}. For $\m{x}, \m{x}' \in \Omega^{d-d_{r}}$,
\begin{small}
\begin{align*}
    \|\varphi(I, \m{x}) - \varphi(I, \m{x}')\|_{\b{H}}^{2} = \left([\s{P} \varphi(I, \m{x})](I, \m{x}) - [\s{P} \varphi(I, \m{x})](I, \m{x}') \right) - \left([\s{P} \varphi(I, \m{x}')](I, \m{x}) - [\s{P} \varphi(I, \m{x}')](I, \m{x}') \right).
\end{align*}    
\end{small}
Since $\varphi(I, \m{x}), \varphi(I, \m{x}')$ are continuous functions, both $\s{P} \varphi(I, \m{x}), \s{P} \varphi(I, \m{x}')$ are also continuous functions, so the above equation converges to $0$ whenever $|\m{x} - \m{x}'| \rightarrow 0$.
\end{proof}

Since $\b{G} \times \Omega^{d-d_{r}}$ is compact, the continuity of $\varphi: \b{G} \times \Omega^{d-d_{r}} \rightarrow \b{H}, (R, \m{x}) \mapsto \varphi(R, \m{x})$
ensures that $\sup_{R, \m{x}} \|\varphi(R, \m{x})\|_{\b{H}}$ is finite, say $B > 0$. Henceforth, we will use $B > 0$ instead of $B_{\m{x}}$ in \eqref{eq:sino:kernel}.

\begin{remark}\label{rmk:xray:arbi:kern}
Any $SO(d)$-invariant Mercer kernel $K: \b{B}^{d} \times \b{B}^{d} \rightarrow \b{R}$ on $\b{B}^{d}$ is a sinogram kernel associated to the $k$-plane transform in \cref{rmk:kplane}. By the compactness of $\b{B}^{d}$, $K$ is bounded by some constant, say $(B/|\b{B}^{d_{r}}|)^{2} >0$. Then, the Cauchy-Schwarz inequality yields
\begin{align*}
    \sup_{\m{x} \in \b{B}^{d-d_{r}}} |\s{P}_{I} f (\m{x})| 
    &= \left| \int_{(1-\|\m{x}\|^{2})^{1/2} \cdot \b{B}^{d_{r}}} f( [\m{x}:z]) \rd z \right|
    = \int_{(1-\|\m{x}\|^{2})^{1/2} \cdot \b{B}^{d_{r}}} |\langle f, k_{[\m{x}:z]} \rangle_{\b{H}}| \rd z\\
    &\le \int_{(1-\|\m{x}\|^{2})^{1/2} \cdot \b{B}^{d_{r}}} \|f\|_{\b{H}} \frac{B}{|\b{B}^{d_{r}}|} \rd z  
    \le B \|f\|_{\b{H}}.
\end{align*}   
\end{remark}

We show that the range of the operator $\s{P}$ is again an RKHS. \cref{thm:induced:RKHS} generalizes the central theorem in \cite{yun2025computerized} with a simpler proof, and also incorporates the pull-back theorem (Theorem 5.7 in \cite{paulsen2016introduction}).

\begin{theorem}\label{thm:induced:RKHS}
Let $\s{P}$ be a sinogram operator and $K: \Omega^{d} \times \Omega^{d} \rightarrow \b{R}$ be a sinogram kernel. Define bivariate functions $\tilde{K}^{I}: \Omega^{d-d_{r}} \times \Omega^{d-d_{r}} \rightarrow \b{R}$ and $\tilde{K}: (\b{G} \times \Omega^{d-d_{r}}) \times (\b{G} \times \Omega^{d-d_{r}}) \rightarrow \b{R}$ by
\begin{align*}
    \tilde{K}^{I} (\m{x}, \m{x}') := \langle \varphi(I, \m{x}), \varphi(I, \m{x}') \rangle_{\b{H}}, \quad
    \tilde{K} ((R, \m{x}), (R', \m{x}')) := \langle \varphi(R, \m{x}), \varphi(R', \m{x}') \rangle_{\b{H}}.
\end{align*}
Then, both $\tilde{K}^{I}$ and $\tilde{K}$ are Mercer kernels. Also, if we restrict the domain of $\s{P}_{R}$ and $\s{P}$ to $\b{H}(K)$, then both $\s{P}_{R}: \b{H}(K) \rightarrow \b{H}(\tilde{K}^{I})$ and $\s{P}: \b{H}(K) \rightarrow \b{H}(\tilde{K})$ are well-defined contractive, surjective operators. Considering the sinogram operator as an operator between RKHSs, the following holds for any $R \in \b{G}$:
\begin{enumerate}
\item $\s{P}_{R}^{*}: \b{H}(\tilde{K}^{I}) \rightarrow \b{H}(K)$ and $\s{P}^{*}: \b{H}(\tilde{K}) \rightarrow \b{H}(K)$ are isometries, i.e. $\s{P}_{R} \circ \s{P}_{R}^{*}$ are identity operators on $\b{H}(\tilde{K}^{I})$ and $\b{H}(\tilde{K})$, respectively. Additionally, they share the same range given by
\begin{equation*}
    \c{R}(\s{P}^{*}) = \c{R}(\s{P}_{R}^{*}) = \overline{\spann\{\varphi(I, \m{x}) : \m{x} \in \Omega^{d-d_{r}} \}}^{\b{H}(K)},
\end{equation*}
where $\overline{\cdot}^{\b{H}(K)}$ denotes the closure in $\b{H}(K)$.

\item For $\tilde{K}^{I}$, it holds that $\tilde{k}^{I}_{\m{x}} = \s{P}_{R} \varphi(R, \m{x})$ and $\s{P}_{R}^{*} \tilde{k}^{I}_{\m{x}} =  \varphi(R, \m{x})$. For $\tilde{K}$, it holds that $\tilde{k}_{(R, \m{x})} = \s{P} \varphi(R, \m{x})$ and $\s{P}^{*} \tilde{k}_{(R, \m{x})} =  \varphi(R, \m{x})$.

\item Both $\b{H}(\tilde{K}^{I})$ and $\b{H}(\tilde{K})$ are equipped with the range norm in the following sense \cite{sarason2007complex}:
\begin{align*}
    &\|g_{R}\|_{\b{H}(\tilde{K}^{I})}=\|\s{P}_{R}^{*} g_{R}\|_{\b{H}(K)}= \min \left\{\|f\|_{\b{H}(K)} : g_{R}=\s{P}_{R}f \right\}, \quad g_{R} \in \b{H}(\tilde{K}^{I}), \\
    &\|g\|_{\b{H}(\tilde{K})}=\|\s{P}^{*} g\|_{\b{H}(K)}= \min \left\{\|f\|_{\b{H}(K)} : g=\s{P} f \right\}, \quad g \in \b{H}(\tilde{K}).
\end{align*}
\end{enumerate}
\end{theorem}

\begin{proof}
First, note that for any $R \in \b{G}$, $\tilde{K}^{I} ( \m{x}, \m{x}') = \langle \varphi(R, \m{x}), \varphi(R, \m{x}') \rangle_{\b{H}}$ from \eqref{eq:Gram:elem}. It follows from \cref{prop:A1:check} that $\tilde{K}^{I}$ is a Mercer kernel. Also, let $\b{V} := \spann \{\tilde{k}^{I}_{\m{x}} : \m{x} \in \Omega^{d-d_{r}} \}$, then $\b{V}^{\perp} = \{0\}$, thus $\overline{\b{V}} = \b{H}(\tilde{K}^{I})$.
Now, for any $f \in \b{H}(K)$,
\begin{equation}\label{eq:contract}
    \|f\|_{\b{H}(K)}^{2} \tilde{K}^{I} ( \m{x}, \m{x}') - \s{P}_{R} f(\m{x}) \s{P}_{R} f(\m{x}') = \|f\|_{\b{H}(K)}^{2} \langle \varphi(R, \m{x}), \varphi(R, \m{x}') \rangle_{\b{H}} - \langle f, \varphi(R, \m{x})\rangle_{\b{H}} \langle f, \varphi(R, \m{x}')\rangle_{\b{H}}
\end{equation}
is a kernel function on $\Omega^{d-d_{r}}$ due to Cauchy-Scwharz inequality. Then Theorem 3.11 in \cite{paulsen2016introduction} reveals that $\s{P}_{R} f \in \b{H}(\tilde{K}^{I})$ and $\|P_{R} f\|_{\b{H}(\tilde{K}^{I})} \le \|f\|_{\b{H}(K)}$, i.e. $\s{P}_{R}: \b{H}(K) \rightarrow \b{H}(\tilde{K}^{I})$ is well-defined and is contractive.

For any $R \in \b{G}$ and $\m{x} \in \Omega^{d-d_{r}}$, it holds from \cref{prop:A1:check} that 
\begin{align*}
    &\tilde{k}^{I}_{\m{x}}(\m{x}') = \tilde{K}^{I} ( \m{x}, \m{x}') = \langle \varphi(R, \m{x}), \varphi(R, \m{x}') \rangle_{\b{H}(K)} = \s{P}_{R} \varphi(R, \m{x}) (\m{x}'), \\
    &\varphi(R, \m{x}) (\m{z}) = \s{P}_{R} k_{\m{z}} (\m{x}) = \langle \s{P}_{R} k_{\m{z}}, \tilde{k}^{I}_{\m{x}} \rangle_{\b{H}(\tilde{K}^{I})} = \langle k_{\m{z}}, \s{P}_{R}^{*} \tilde{k}^{I}_{\m{x}} \rangle_{\b{H}(\tilde{K})} = \s{P}_{R}^{*} \tilde{k}^{I}_{\m{x}} (\m{z}),
\end{align*}
i.e., $\tilde{k}^{I}_{\m{x}} = \s{P}_{R} \varphi(R, \m{x})$ and $\s{P}_{R}^{*} \tilde{k}^{I}_{\m{x}} =  \varphi(R, \m{x})$.
Hence, $\s{P}_{R} \s{P}_{R}^{*} \tilde{k}^{I}_{\m{x}} = \tilde{k}^{I}_{\m{x}}$ and because $\b{V}$ is dense in $\b{H}(\tilde{K}^{I})$, we deduce that $\s{P}_{R} \s{P}_{R}^{*}$ is the identity on $\b{H}(\tilde{K}^{I})$, or equivalently, $\s{P}_{R}^{*}$ is a linear isometry.
As a consequence, we have that $\c{R}(\s{P}_{R})$ is closed in $\b{H}(\tilde{K}^{I})$, and $\c{R}(\s{P}_{R}^{*})$ is closed in $\b{H}(K)$ due to Proposition 6.15 in \cite{paulsen2016introduction}. Therefore,
\begin{equation*}
    \c{R}(\s{P}_{R}) = \overline{\c{R}(\s{P}_{R})} \supset \overline{\b{V}} = \b{H}(\tilde{K}^{I}),
\end{equation*}
i.e. $\s{P}_{R}$ is surjective. Also,
\begin{equation*}
    \overline{\s{P}_{R}^{*} \b{V}} \subset \overline{\c{R}(\s{P}_{R}^{*})}= \c{R}(\s{P}_{R}^{*}) = \s{P}_{R}^{*}(\overline{\b{V}}) \subset \overline{\s{P}_{R}^{*} \b{V}},
\end{equation*}
hence $\s{P}_{R}^{*}(\overline{\b{V}}) = \overline{\s{P}_{R}^{*} \b{V}} = \overline{\spann\{\varphi(R, \m{x}) : \m{x} \in \Omega^{d-d_{r}} \}}$.
\begin{equation*}
    \c{R}(\s{P}_{R}^{*}) = \overline{\s{P}_{R}^{*} \b{V}} = \overline{\spann\{\varphi(R, \m{x}) : \m{x} \in \Omega^{d-d_{r}} \}}.
\end{equation*}

Now, for a sinogram operator $\s{P}$, \eqref{eq:contract} and Theorem 3.11 in \cite{paulsen2016introduction} shows again that $\s{P}: \b{H}(K) \rightarrow \b{H}(\tilde{K})$ is a well-defined contractive operator. Also, same as above using \cref{prop:A1:check}, we get $\tilde{k}_{(R, \m{x})} = \s{P} \varphi(R, \m{x})$ and $\s{P}^{*} \tilde{k}_{(R, \m{x})} =  \varphi(R, \m{x})$, and thus $\s{P}^{*}$ is a linear isometry. Continuing the same argument, we also conclude that $\s{P}$ is surjective, and $\c{R}(\s{P}^{*}) = \overline{\spann\{\varphi(R, \m{x}) : \m{x} \in \Omega^{d-d_{r}} \}}$.
\end{proof}

Note that the induced kernel $\tilde{K}$ is also $\b{G}$-invariant by \eqref{eq:Gram:elem}: for any $R \in \b{G}$, we have
\begin{equation}\label{eq:induc:rk:inv}
    \tilde{K} ((R_{1}, \m{x}_{1}), (R_{2}, \m{x}_{2})) := \langle \varphi(R_{1}, \m{x}_{1}), \varphi(R_{2}, \m{x}_{2}) \rangle_{\b{H}} = \tilde{K} ((R R_{1}, \m{x}_{1}), (R R_{2}, \m{x}_{2})).
\end{equation}
Additionally,  $\tilde{K}, \tilde{K}^{I}$ only depend on the sinogram operator $\s{P}$ and the original kernel $K$, not on the choice of the weight functions. For instance, when $\s{P}$ is the $k$-plane transform and $K(\m{z}_{1}, \m{z}_{2}) = \exp (-\gamma \|\m{z}_{1} - \m{z}_{2}\|^{2})$ is a Gaussian kernel with the parameter $\gamma >0$, \eqref{eq:Gram:elem} yields that
\begin{align*}
    \tilde{K}^{I}(\m{x}_{1}, \m{x}_{2}) = \exp (-\gamma \|\m{x}_{1} - \m{x}_{2}\|^{2}) \times \int_{(1-\|\m{x}\|^{2})^{1/2} \cdot \b{B}^{d_{r}}} \int_{(1-\|\m{x}'\|^{2})^{1/2} \cdot \b{B}^{d_{r}}} \exp (-\gamma \|z_{1} - z_{2}\|^{2}) \rd z_{1} \rd z_{2},
\end{align*}
where the latter integral can be evaluated using the CDF of a multivariate normal distribution \cite{yun2025computerized}.

Also, we remark that $\s{P} \circ \imath$ in \cref{sec:Mercer:inv} is \emph{not} equal to $\s{P}:\b{H}(K) \rightarrow \b{H}(\tilde{K})$ in \cref{thm:induced:RKHS} as the codomains and the inner product structures differ. Indeed, $\s{P}:\b{H}(K) \rightarrow \b{H}(\tilde{K})$ is not even compact as its adjoint is isometric. To avoid confusion from now on, $\s{P}$ always implies the operator between $\c{L}_{2}$ spaces.

\subsection{Direct Measurements}\label{ssec:direct:msr}
The standard FDA setting under direct measurements corresponds to the case where both $\b{G} = \{I\}$ and $\s{P}_{I} = I$ are trivial, thus every Mercer kernel on $\Omega^{d}$ qualifies as a valid sinogram kernel. The feature map $\varphi(R, \m{x})$ reduces to the kernel generator $k_{\m{x}}$, and the induced reproducing kernel $\tilde{K}$ coincides with the original kernel. We briefly review Mercer’s theory here as a baseline for comparison in \cref{sec:Mercer:inv}.

Under suitable boundary conditions, a common approach to constructing the kernel is via the Green’s function $G(\m{z}, \cdot) \in \c{L}_{2} (\Omega^{d})$ associated with a differential operator $\s{D}$ on $\c{L}_{2} (\Omega^{d})$, which is injective and unbounded:
\begin{equation*}
    K(\m{z}_{1}, \m{z}_{2}) = \int_{\Omega^{d}} G(\m{z}_{1}, \m{z}_{3}) G(\m{z}_{2}, \m{z}_{3}) \rd \m{z}_{3}, \quad \m{z}_{1}, \m{z}_{2} \in \Omega^{d}.
\end{equation*}
Fredholm theory \cite{fredholm1903classe, paulsen2016introduction} guarantees that the composition $\s{D} \s{T}_{G}$ is the identity on $\c{L}_{2} (\Omega^{d})$, yielding:
\begin{equation*}
    \b{H}(K) = \text{Dom}(\s{D}), \quad \|f\|_{\b{H}(K)} = \| \s{D} f\|_{\c{L}_{2} (\Omega^{d})},       
\end{equation*}
which provides an explicit interpretation of smoothness in terms of $\s{D}$.
Furthermore, the s.p.d. compact operator $\s{T}_{K} = \s{T}_{G}\s{T}_{G}^{*}$ admits an eigen-decomposition with eigenvalue-eigenfunction pairs $(\tau_{l}^{2}, f_{lm})$:
\begin{equation*}
    K(\m{z}_{1}, \m{z}_{2})=\sum_{l=1}^{\infty} \tau_{l}^{2} \sum_{m=1}^{N(l)} f_{lm}(\m{z}_{1}) f_{lm}(\m{z}_{2}), \quad \m{z}_{1}, \m{z}_{2} \in \Omega^{d},
\end{equation*}
where $\tau_{1} > \tau_{2} > \dots > 0$ and $\{f_{lm}: l \in \b{N}, 1 \le m \le N(l) \}$ forms a CONS in $\c{L}_{2} (\Omega^{d})$.

\begin{example}\label{ex:unit:Sobo}
Let $\s{D} = \frac{\rd}{\rd z}$ on $\Omega^{d} = [0, 1]$ with the boundary condition $f(0) = 0$. Then the Green’s function yields:
\begin{equation*}
    K(z_{1}, z_{2}) = \int_{0}^{1} I(z_{3} < z_{1}) I(z_{3} < z_{2}) \rd z_{3} = \min (z_{1}, z_{2}), \quad 0 \le z_{1}, z_{2} \le 1.
\end{equation*}
The Mercer decomposition gives $N(l) = 1$, $\tau_{l} = [(l-1/2)\pi]^{-1}$, and $f_{l}(z) = \sqrt{2} \sin ((l-1/2)\pi z)$ for each $l \in \b{N}$. More generally, let $p \in \b{N}$, and $\s{D} = \frac{\rd^{p}}{\rd z^{p}}$ with boundary conditions $f(0)= f'(0) = \cdots = f^{(p-1)}(0) = 0$. Then, integration by parts leads to $\s{D}^{*} = (-1)^{p} \frac{\rd^{p}}{\rd z^{p}}$ with boundary conditions $f(1)= f'(1) = \cdots = f^{(p-1)}(1) = 0$, and $f_{l}$'s are the eigenfunctions of $\s{D}^{*} \s{D} = (-1)^{p} \frac{\rd^{2p}}{\rd z^{2p}}$ with corresponding boundary conditions \cite{conway2019course}. Consequently, the kernel for the Sobolev space of order $p$ becomes
\begin{align*}
    K(z_{1}, z_{2}) 
    = \frac{B_{p}(z_{1}) B_{p}(z_{2})}{p!} - \frac{B_{2p}(|z_{1} - z_{2}|)}{(2p)!}, \quad 0 \le z_{1}, z_{2} \le 1,
\end{align*}
where $B_{p}$ is the Bernoulli polynomial of degree $p$, which is the setup considered in \cite{rice1991estimating, cai2010nonparametric}.
\end{example}

If the operator $\s{D}$ and its boundary conditions are invariant under a group action, then the resulting kernel inherits this symmetry:
\begin{example}\label{ex:sphe:Sobo} 
\cite{simeoni2020functional, caponera2022functional} consider the spherical Sobolev space of order $p > d/2$, arising from the pseudo-differential operator $\s{D} = (I - \Laplace_{\b{S}^{d-1}})^{p/2}$. Note that $\Omega_{d}= \b{S}^{d-1}$ has no boundary and the Laplace-Beltrami operator $\Laplace_{\b{S}^{d-1}}$ is isotropic.
In this case, the eigenfunctions $f_{lm} = Y_{lm}$ are the spherical harmonics of degree $l$, with corresponding eigenvalues and multiplicities:
\begin{equation}\label{eq:sphe:har:svd}
    \tau_{l} = [1 + l (l+d-2)]^{-p/2} \asymp l^{-p}, \quad
    N(l)=
    \begin{pmatrix}
        l+d-1 \\ l
    \end{pmatrix} -
    \begin{pmatrix}
        l+d-3 \\ l-2
    \end{pmatrix} \asymp l^{d-1}.
\end{equation}
Thus, the reproducing kernel becomes isotropic and admits the expansion:
\begin{align*}
    K(\theta_{1}, \theta_{2})= \frac{1}{|\b{S}^{d-1}|} \sum_{l=1}^{\infty} \tau_{l}^{2} N(l) C_{l}^{d/2}(\theta_{1}^{\top} \theta_{2}), \quad \theta_{1}, \theta_{2} \in \b{S}^{d-1},
\end{align*}
where $C_{l}^{d/2-1}$ is the Gegenbauer polynomial of degree $l$, normalized such that $C_{l}^{d/2}(1)=1$ \cite{dai2013approximation}.
\end{example}

However, when $\partial \Omega^{d} \neq \emptyset$, the boundary conditions can break symmetries. For instance, in \cref{ex:unit:Sobo}, although $\s{D}$ is translation-invariant, its domain $[0, 1]$ and boundary conditions are not. Consequently, the resulting reproducing kernel does not inherit translation invariance. If we consider $\b{G} = SO(d)$ instead of $\b{G} = \{I\}$ in \cref{ex:sphe:Sobo}, we get the following interpretation:

\begin{remark}\label{rmk:transact:action}
Consider again the setting in \cref{ex:sphe:Sobo}, where $\Omega^{d} = \b{S}^{d-1}$, $\b{G} = SO(d)$, and $\s{P}_{I} = I$ so that $\s{P} = \rho$. In this case, the eigenvalues $\tau_{l}^{2}$ and multiplicities $N(l)$ are as given in \eqref{eq:sphe:har:svd}, and the eigenfunctions are $e_{lm} = \tau_{l} Y_{lm}$, where $Y_{lm}$ are the spherical harmonics of degree $l$. The irreducible sub-representation $[\rho(R)]_{\{e_{lm} : 1 \le m \le N(l)\}} \in SO(N(l))$ corresponds precisely to the Wigner D-matrix \cite{wigner1931gruppentheorie, dai2013approximation}, and
\begin{equation*}
    g_{lm}(R, \theta) = \rho(R) Y_{lm}(\theta) = Y_{lm}(R^{-1} \theta).
\end{equation*}

Since $\tilde{K}_{l}$ in \eqref{eq:daughter:L2:rk} is isotropic and $SO(d)$ acts transitively on $\b{S}^{d-1}$, it follows that $\tilde{K}_{l}((R_{1}, \theta_{1}), (R_{2}, \theta_{2}))$ depends only on the inner product between $R_{1}^{-1} \theta_{1}, R_{2}^{-1} \theta_{2} \in \b{S}^{d-1}$. As a result, \cref{prop:daughter:rk} implies
\begin{equation*}
    \tilde{K}_{l}((R, \theta), (R, \theta)) \equiv \frac{N(l)}{|\b{S}^{d-1}|}, \quad l \in \b{N}, \, R \in SO(d), \, \theta \in \b{S}^{d-1}.
\end{equation*}
In general, the $l$-th kernel takes the form of the Funk-Hecke formula \cite{dai2013approximation}:
\begin{align*}
    \tilde{K}_{l}((R_{1}, \theta_{1}), (R_{2}, \theta_{2})) = \frac{N(l)}{|\b{S}^{d-1}|} C_{l}^{d/2-1}(\theta_{1}^{\top} R_{1} R_{2}^{\top} \theta_{2}), \quad R_{1}, R_{2} \in SO(d), \, \theta_{1}, \theta_{2} \in \b{S}^{d-1}.
\end{align*}
This recovers the original kernel as $K(\theta_{1}, \theta_{2}) = \tilde{K} ((I, \theta_{1}), (I, \theta_{2})) =\sum_{l=1}^{\infty} \tau_{l}^{2} \tilde{K}_{l} ((I, \theta_{1}), (I, \theta_{2}))$.
\end{remark}

\section{Implementation}\label{sec:num:consid}
This section, heavily relies on our recent work \cite{yun2025fast}, outlines the practical implementation of our mean and covariance estimators, guided by the representer theorem in \cref{thm:repre:mean}. For simplicity, we will assume a uniform number of tilts ($r_{i} \equiv r$) and locations ($s_{ij} \equiv s$) per particle, though the methods extend to non-uniform cases. Our code and column-major indexing conventions follow the order of $i$ (particle) → $j$ (tilt) → $k$ (location). Throughout, block partitioning is performed with respect to the index $i=1, \cdots, n$, corresponding to each random function.

\subsection{Mean Estimation}
Modify the vectorization of observations and the mean Gram matrix introduced in \cref{sec:scheme} as:
\begin{small}
\begin{align*}
    \m{w}:= \mathrm{vec}\left[Z_{ijk} \right] = 
    \left( \begin{array}{c}
    \m{w}_{1} \\
    \hline
    \m{w}_{2} \\
    \hline
    \vdots \\
    \hline
    \m{w}_{n}
    \end{array} \right) \in \b{R}^{nrs}, \, 
    \m{\Phi}:= \left[ \langle \varphi_{i_{1}j_{1}k_{1}}, \varphi_{i_{2}j_{2}k_{2}} \rangle_{\b{H}} \right] = 
    \left( \begin{array}{c|c|c|c}
    \m{\Phi}_{11} & \m{\Phi}_{12} & \cdots & \m{\Phi}_{1n} \\
    \hline
    \m{\Phi}_{21} & \m{\Phi}_{22} & \cdots & \m{\Phi}_{2n} \\
    \hline
    \vdots & \vdots & \ddots & \vdots \\
    \hline
    \m{\Phi}_{n1} & \m{\Phi}_{n2} & \cdots & \m{\Phi}_{nn}
    \end{array} \right) \in \b{R}^{nrs \times nrs}.
\end{align*}    
\end{small}

We then consider the \emph{rescaled} empirical risk functional $\hat{R}_{\nu}:\b{H} \rightarrow \b{R}$ with tuning parameter $\nu > 0$:
\begin{align*}
    \hat{R}_{\nu}(f)
    = \sum_{i=1}^{n} \sum_{j=1}^{r} \sum_{k=1}^{s}  \left(Z_{ijk}-\langle f, \varphi_{ijk} \rangle_{\b{H}} \right)^{2} +\nu \|f\|^{2}_{\b{H}}.
\end{align*}
Since $\hat{R}_{\nu \cdot nrs / \omega_{d_{r}, d}}$ is proportional to $\hat{L}_{\nu}$ defined in \eqref{eq:emp:risk:mean}, no essential change is made but it leads to a cleaner normal equation:
\begin{align*}
    &\hat{\m{\mu}}_{\nu} = \sum_{i=1}^{n} \sum_{j=1}^{r} \sum_{k=1}^{s} \hat{\alpha}^{\nu}_{ijk} \varphi_{ijk}, \quad \hat{\m{\alpha}}_{\nu}:=\mathrm{vec}[\hat{\alpha}^{\nu}_{ijk}] = \left( \m{\Phi}+ \nu \m{I} \right)^{-1} \m{w} \in \b{R}^{nrs}.
\end{align*}
If $\m{R}_{ij} \in \b{G}$ exhibits a parallel imaging geometry, $\m{\Phi}$ has a $n \times n$ block-circulant structure, allowing for even faster solves using FFTs, see Section 5 in \cite{yun2025computerized}.
To visualize the result, we can evaluate $\hat{\m{\mu}}_{\nu} \in \b{H}$ on a grid of points $\c{G} := \{\m{z}_{t}: 1 \le t \le T\} \subset \Omega^{d}$. This is done by first computing a \emph{frame} matrix:
\begin{equation*}
    \m{F} := [F_{t, ijk} := \varphi_{ijk}(\m{z}_{t})] \in \b{R}^{T \times nrs}
\end{equation*}
with entries computed using \eqref{eq:eval:ftnal}. Then, the evaluation of the mean estimator over the grid is simply:
\begin{equation*}
    \hat{\m{\mu}}_{\nu} := [\hat{\m{\mu}}_{\nu} (\m{z}_{1}), \hat{\m{\mu}}_{\nu} (\m{z}_{2}), \dots, \hat{\m{\mu}}_{\nu} (\m{z}_{T})]^{\top} = \m{F} \hat{\m{\alpha}}_{\nu} \in \b{R}^{T}.
\end{equation*}

\subsection{Covariance Estimation}
While the mean estimation is straightforward, a na\"{i}ve approach to the covariance estimation could result in a severe computational burden. The normal equation for the covariance coefficients involves the massive covariance Gram matrix $\m{\Phi}^{\odot}$. Although one may attempt to solve this large linear system via Cholesky or spectral decomposition of the positive definite matrix $\m{\Phi}^{\odot} + \eta \m{I}$, this approach is impractical: due to the involvement of the Khatri-Rao product \cite{khatri1968solutions, neudecker1995hadamard}, the decompositions of $\m{\Phi}$ do not translate into those of $\m{\Phi}^{\odot}$. 
Therefore, our solution is to never form the $\m{\Phi}^{\odot}$ matrix in solving \cref{thm:repre:mean}. Instead, we use the \texttt{TReK} algorithm \cite{yun2025fast} -- a tensorized variant of the Krylov-projection method under range restrictions -- to iteratively solve \cref{thm:repre:mean}, as we now briefly describe. These \emph{element-free} solvers only require a function that computes the matrix-vector product, but the tensor-matrix product in our case, which we now briefly describe.

To formalize the element-free computation, we first define $\odiag(\b{R}^{n(rs)} \otimes \b{R}^{n(rs)})$ as the space of $n \times n$ diagonal block matrices, where each block is an $\b{R}^{(rs) \times (rs)}$ matrix:
\begin{equation*}
    \m{A} = \odiag[\m{A}_{i}] = \left( \begin{array}{c|c|c|c}
    \m{A}_{1} & \m{0} & \dots & \m{0} \\
    \hline
    \m{0} & \m{A}_{2} & \dots & \m{0} \\
    \hline
    \vdots & \vdots & \ddots & \vdots \\
    \hline
    \m{0} & \m{0} & \dots & \m{A}_{n}
    \end{array} \right) \in  \odiag(\b{R}^{n(rs)} \otimes \b{R}^{n(rs)}) \subset \b{R}^{n(rs) \times n(rs)}.
\end{equation*}

The block outer product of the observation vector is then $\m{w} \odot \m{w}:= \odiag[\m{w}_{i} \m{w}_{i}^{\top}] \in  \odiag(\b{R}^{n(rs)} \otimes \b{R}^{n(rs)})$. We also define a block outer product for the Gram matrix $\m{\Phi} \in \b{R}^{n(rs) \times n(rs)}$ that maps a diagonal block matrix to another:
\begin{align*}
    \m{\Phi} \odot \m{\Phi} : &\odiag[\m{A}_{i}]_{i=1}^{n} \in \odiag(\b{R}^{n(rs)} \otimes \b{R}^{n(rs)}) \mapsto \\
    &\hspace{3em} \odiag \left[ \sum_{i'=1}^{n} \m{\Phi}_{i i'} \m{A}_{i'} \m{\Phi}_{i' i} \right]_{i=1}^{n} \in \odiag(\b{R}^{n(rs)} \otimes \b{R}^{n(rs)}).
\end{align*}
The pseudocode for this operation is presented in \cref{alg:lazy:khatri:block}.

\begin{algorithm}[h!]
\caption{Block Outer Product}\label{alg:lazy:khatri:block}
\begin{algorithmic}[1]
\Require $\m{\Phi} \in \b{R}^{\m{r} \times \m{r}}, \m{A} = \odiag[\m{A}_{i}] \in \odiag(\b{R}^{n(rs)} \otimes \b{R}^{n(rs)})$
\Ensure $\m{B}= (\m{\Phi} \odot \m{\Phi}) \m{A} \in \odiag(\b{R}^{n(rs)} \otimes \b{R}^{n(rs)})$
    \For{$i = 1$ to $n$}
        \State $\m{B}_i \gets \m{0}$
        \For{$i' = 1$ to $n$}
            \State $\m{B}_i \gets \m{B}_i + \m{K}_{i i'} \m{A}_{i'} \m{K}_{i i'}^\top$
        \EndFor
    \EndFor   
\end{algorithmic}
\end{algorithm}

Moving forward, we introduce the index set adhering to our off-diagonal pair observation scheme:
\begin{align}\label{eq:idx:non_redund}
    \c{J}_{i} := \{ (J_{1}, J_{2})=(j_{1}, k_{1}, j_{2}, k_{2}) : 1 \le j_{1} \neq j_{2} \le r, 1 \le k_{1}, k_{2} \le s \}, \quad |\c{J}_{i}|=r(r-1)s^{2}.
\end{align}
Then, we define the $i$-th elimination operator for each $i$-th block with respect to $\c{J}_{i}$: 
\begin{equation*}
    \s{J}_{i}: \b{R}^{(rs) \times (rs)} \to \b{R}^{(rs) \times (rs)}, \quad (\s{J}_{i} \m{A}_{i})[J_{1}, J_{2}] =
    \begin{cases}
        \m{A}_{i}[J_{1}, J_{2}] &, \quad (J_{1}, J_{2}) \in \c{J}_{i} \\
        0 &, \quad (J_{1}, J_{2}) \notin \c{J}_{i}
    \end{cases},
\end{equation*}
zeroes out the diagonal blocks of $\m{A}_{i}$. Accordingly, the elimination operator acts block-wise: 
\begin{equation*}
    \s{J} = \odiag[\s{J}_{i}] : \odiag[\m{A}_{i}] \in \odiag(\b{R}^{n(rs)} \otimes \b{R}^{n(rs)}) \mapsto \odiag[\s{J}_{i} \m{A}_{i}] \in \odiag(\b{R}^{n(rs)} \otimes \b{R}^{n(rs)}),
\end{equation*}
This is an orthogonal projection, with pseudocode provided in \cref{alg:diag:elim}.

\begin{algorithm}[h!]
\caption{Elimination of Irrelevant Pairs}\label{alg:diag:elim}
\begin{algorithmic}[1]
\Require $\m{A} = \odiag[\m{A}_{i}] \in \odiag(\b{R}^{n(rs)} \otimes \b{R}^{n(rs)})$
\Ensure $\m{B} = \s{J} \m{A} \in \odiag(\b{R}^{n(rs)} \otimes \b{R}^{n(rs)})$
    \State $\m{B} \gets \m{A}$
    \For{$i = 1$ to $n$}
        \For{$j = 1$ to $r$}
            \For{$k, k' = 1$ to $s$}
                \State $\m{B}_i[(j, k), (j, k')] \gets 0$
            \EndFor
        \EndFor
    \EndFor    
\end{algorithmic}
\end{algorithm}

The \emph{rescaled} empirical risk functional $\hat{R}_{\eta}^{\odot}:\b{H} \otimes \b{H} \rightarrow \b{R}$ for the second moment with tuning parameter $\eta > 0$ is:
\begin{align*}
    \hat{R}_{\eta}^{\odot}(G)
    = \sum_{i=1}^{n} \sum_{(J_{1}, J_{2}) \in \c{J}} \left(Z_{iJ_{1}J_{2}} - \left\langle G, \varphi_{iJ_{1}J_{2}}^{\otimes} \right\rangle_{\b{H} \otimes \b{H}} \right)^{2} + \eta \|G\|^{2}_{\b{H} \otimes \b{H}}.
\end{align*}
Note that $\hat{R}_{(\eta n|\c{J}|/ \omega_{d_{r}, d}^{2})}^{\odot}(G) \propto \hat{L}_{\eta}^{\odot}(G)$ in \eqref{eqn:emp:risk:cov} whenever $G \in \b{H} \otimes \b{H}$ is a self-adjoint operator due to the symmetry: $Z_{iJ_{1}J_{2}} = Z_{iJ_{2}J_{1}}$ and $(\varphi_{iJ_{1}J_{2}}^{\otimes})^{*} = \varphi_{iJ_{2}J_{1}}^{\otimes}$. The unique minimizer of $\hat{R}_{\eta}^{\odot}$ is a linear combination of symmetrized tensor products of feature maps:
\begin{align*}
    \hat{\m{\Gamma}}_{\eta} = \sum_{i=1}^{n} \sum_{(J_{1}, J_{2}) \in \c{J}} \hat{\alpha}^{\odot}_{iJ_{1}J_{2}} \frac{\varphi_{ij_{1}k_{1}} \otimes \varphi_{ij_{2}k_{2}} + \varphi_{ij_{2}k_{2}} \otimes \varphi_{ij_{1}k_{1}}}{2},
\end{align*}
If we arrange its coefficients into a block-diagonal matrix that lies in the range of $\s{J}$:
\begin{align}\label{eq:set:range:constr}
    \hat{\m{A}}_{\eta} := \odiag [ (\hat{\m{A}}_{\eta})_{i} ] \in \c{R}(\s{J}) \subset \odiag(\b{R}^{n(rs)} \otimes \b{R}^{n(rs)}),
\end{align}
then this coefficient matrix is the unique solution to the normal equation:
\begin{equation}\label{eq:normal:eq:cov}
    [\s{J} (\m{\Phi} \odot \m{\Phi}) \s{J} + \lambda \s{I}] \hat{\m{A}}_{\eta} = \s{J} (\m{w} \odot \m{w}).
\end{equation}
where $\s{I}$ is the identity operator on $\odiag(\b{R}^{n(rs)} \otimes \b{R}^{n(rs)})$, see Theorem 3.2 in \cite{yun2025fast}.
The action of the large positive definite operator on the left hand side can be implemented using $\s{J} (\m{\Phi} \odot \m{\Phi}) \s{J}$ using \cref{alg:lazy:khatri:block,alg:diag:elim}, allowing the system to be solved with an iterative Krylov method like CG or LSQR \cite{saad2003iterative, hestenes1952methods, paige1982lsqr}.
Once \eqref{eq:normal:eq:cov} is solved, the bivariate function $\hat{\m{\Gamma}}_{\eta}: \Omega^{d} \times \Omega^{d} \rightarrow \b{R}$ can be visualized over the square grid $\c{G} \times \c{G}$ via the following quadratic form:
\begin{equation*}
    [\hat{\m{\Gamma}}_{\eta}(\m{z}_{t_{1}}, \m{z}_{t_{2}})]_{1 \le t_{1}, t_{2} \le T} = \m{F} \hat{\m{A}}_{\eta} \m{F}^{\top} \in \b{R}^{T \times T}.
\end{equation*}
Accordingly, the covariance estimator is evaluated as $\m{F} (\hat{\m{A}}_{\eta} - \hat{\m{\alpha}}_{\nu} \hat{\m{\alpha}}_{\nu}^{\top}) \m{F}^{\top}$. 

\begin{remark}\label{rmk:spd:fail}
The constraint structure \eqref{eq:set:range:constr} reveals that $\tr (\hat{\m{A}}_{\eta}) = 0$, hence $\hat{\m{\Gamma}}_{\eta}$ fails to be s.p.d. unless it is zero. This issue arises fundamentally from using only off-diagonal pairs and is not an artifact of the inverse problem itself. Notably, the same phenomenon persists even when directly estimating the covariance function from centered observations. This underscores the necessity of the subsequent functional PCA step to obtain a s.p.d. covariance estimate by truncating the negative spectrum.
\end{remark}

The final step is to perform an eigendecomposition of the estimated covariance operator. The following theorem connects the kernel PCA to a generalized eigenvalue problem, see \cite{yun2025fast} for the proof.

\begin{theorem}\label{thm:fpca}
Let $\m{\Phi} \in \b{R}^{n(rs) \times n(rs)}$ be the mean Gram matrix, and consider the covariance tensor of the form
\begin{align*}
    \m{\Sigma} = \sum_{i_{1}, i_{2}=1}^{n} \sum_{j_{1}, j_{2}=1}^{r} \sum_{k_{1}, k_{2}=1}^{s} \m{C}_{i_{1}j_{1}k_{1}, i_{2}j_{2}k_{2}} \,  (\varphi_{i_{1}j_{1}k_{1}} \otimes \varphi_{i_{2}j_{2}k_{2}}),
\end{align*}
where $\m{C} \in \b{R}^{n(rs) \times n(rs)}$ is symmetric. Let $q = \text{rank} (\m{C} \m{\Phi})$, $(\lambda_{k}, \psi_{l} = \sum_{l=1}^{p} v_{ijk, l} \varphi_{ijk})$ denote the $l$-th eigenpair of $\m{\Sigma}$, i.e. $\m{\Sigma} = \sum_{l=1}^{q} \lambda_{l} \psi_{l} \otimes \psi_{l}$. If we define
\begin{equation*}
    \m{\Lambda} = \odiag[\lambda_{l}]_{1 \le l \le q} \in \b{R}^{q \times q}, \quad \m{V} = [v_{ijk, l}] = [ \m{v}_{1} \vert \cdots \vert \m{v}_{q} ] \in \b{R}^{n(rs) \times q},
\end{equation*}
then the diagonal matrix $\m{\Lambda}$ and $\m{G}$-orthogonal matrix $\m{V}$ satisfy the generalized eigenvalue problem:
\begin{equation*}
    \m{C} \m{\Phi} \m{V} = \m{V} \m{\Lambda}, \quad \m{V}^{\top} \m{\Phi} \m{V} = \m{I}
\end{equation*}
Also, when $\m{\Phi}$ is rank-deficient, the kernel PCA does not depend on the choice of $\m{V}$.
\end{theorem}

Applying \cref{thm:fpca} with $\m{C} = \hat{\m{A}}_{\eta} - \hat{\m{\alpha}}_{\nu} \hat{\m{\alpha}}_{\nu}^{\top}$ yields the spectral decomposition of the covariance estimator $\hat{\m{\Sigma}}_{\nu, \eta} \in \b{H} \otimes \b{H}$. If a s.p.d. estimator is desired, the components corresponding to negative eigenvalues can be truncated. If we are only interested in the dominant principal components, incomplete factorization can be done using a version of Lanczos tridiagonalization \cite{lanczos1950iteration, paige1972computational}, see \cite{yun2025fast} and references therein.

\section{Proofs}\label{sec:proof}
\subsection{Proofs in Section \ref{sec:setup}}\label{sec:proof:setup}

\begin{proof}[Proof of \cref{prop:repre::proj}]
\begin{enumerate}
\item For all $1 \le i \le n, 1 \le j \le r_{i}, 1 \le k \le s_{ij}$, it is straightforward that $0 = \s{P}_{\m{R}_{ij}} f(\m{X}_{ijk}) = \langle f, \varphi_{ijk} \rangle_{\b{H}}$ if and only if $f \in (\b{H}_{\m{F}})^{\perp}$.

\item For all $1 \le i \le n, 1 \le j \neq j' \le r_{i}, 1 \le k \le s_{ij}, 1 \le k' \le s_{ij'}$,
\begin{align*}
    0 &= (\s{P}_{\m{R}_{ij}} \otimes \s{P}_{\m{R}_{ij'}}) G(\m{X}_{ijk}, \m{X}_{ij'k'}) = \langle G, \varphi_{ijk} \otimes \varphi_{ij'k'} \rangle_{\b{H} \otimes \b{H}},
\end{align*}
if and only if $G \in ((\b{H} \otimes \b{H})_{\m{F}})^{\perp}$.
\end{enumerate}
\end{proof}

\begin{proof}[Proof of \cref{prop:repre}]
Given $f \in \b{H}$, decompose it by $f = f_{\m{F}} \oplus (f-f_{\m{F}}) \in \b{H}_{\m{F}} \oplus (\b{H}_{\m{F}})^{\perp}$. Then by \cref{prop:repre::proj}, we get
\begin{equation*}
    \hat{L}_{\nu}(f) = \hat{L}_{\nu}(f_{\m{F}}) + \nu \|f-f_{\m{F}}\|^{2} \ge \hat{L}_{\nu}(f_{\m{F}}),
\end{equation*}
thus any minimizer of $\hat{L}_{\nu}$ should belong to $\b{H}_{\m{F}} \subset \b{H}$. Moreover, the minimizer of $\hat{L}_{\nu}$ is unique since $\hat{L}_{\nu}$ is strictly convex in $f \in \b{H}$, provided that $\nu > 0$. Applying the same rationale, one can show that the minimizer of $\hat{L}_{\eta}^{\odot}$ is unique, and it belongs to $(\b{H} \otimes \b{H})_{\m{F}}$. Finally, since $\hat{L}_{\eta}^{\odot}(G) = \hat{L}_{\eta}^{\odot}(G^{*})$, $\hat{\m{\Gamma}}_{\eta}$ is self-adjoint due to its uniqueness.
\end{proof}

\begin{proof}[Proof of \cref{thm:repre:mean}]
Since
\begin{align*}
    &\mathrm{vec}\left[ \frac{\langle \hat{\m{\mu}}_{\nu}, \varphi_{ijk} \rangle_{\b{H}}}{\sqrt{r_{i}s_{ij}}} \right] = \mathrm{vec}\left[ \frac{1}{\sqrt{r_{i}s_{ij}}} \langle \sum_{i'=1}^{n} \sum_{j'=1}^{r_{i}} \sum_{k'=1}^{s_{i'j'}} \frac{1}{\sqrt{r_{i'}s_{i'j'}}} \hat{\alpha}^{\nu}_{i'j'k'} \varphi_{i'j'k'}, \varphi_{ijk} \rangle_{\b{H}} \right] = \m{\Phi} \hat{\m{\alpha}}_{\nu}, \\
    &\|\hat{\m{\mu}}_{\nu}\|^{2}_{\b{H}} = \left\| \sum_{i=1}^{n} \sum_{j=1}^{r_{i}} \sum_{k=1}^{s_{ij}} \frac{1}{\sqrt{r_{i}s_{ij}}} \hat{\alpha}^{\nu}_{ijk} \varphi_{ijk} \right\|^{2}_{\b{H}} = \hat{\m{\alpha}}_{\nu}^{\top} \m{\Phi} \hat{\m{\alpha}}_{\nu},
\end{align*}
we get
\begin{align*}
    \hat{L}_{\nu}(\hat{\m{\mu}}_{\nu})
    &:= \frac{\omega_{d_{r}, d}}{n} \sum_{i=1}^{n}  \sum_{j=1}^{r_{i}} \sum_{k=1}^{s_{ij}} \left(\frac{Z_{ijk}}{\sqrt{r_{i}s_{ij}}} - \frac{\langle \hat{\m{\mu}}_{\nu}, \varphi_{ijk} \rangle_{\b{H}}}{\sqrt{r_{i}s_{ij}}}  \right)^{2} +\nu \|\hat{\m{\mu}}_{\nu}\|^{2}_{\b{H}} \\
    &= \frac{\omega_{d_{r}, d}}{n} \|\m{w} - \m{\Phi} \hat{\m{\alpha}}_{\nu}\|^{2} + \nu \hat{\m{\alpha}}_{\nu}^{\top} \m{\Phi} \hat{\m{\alpha}}_{\nu}.
\end{align*}
This yields the normal equation as follows:
\begin{align*}
    \left( \m{\Phi}^{2}+ \frac{n \nu}{\omega_{d_{r}, d}} \m{\Phi} \right) \hat{\m{\alpha}}^{\nu}= \m{\Phi} \m{w} \,&\Longleftrightarrow \, \m{\Phi} \hat{\m{\alpha}}^{\nu}=\m{\Phi} \left( \m{\Phi}+ \frac{n \nu}{\omega_{d_{r}, d}} \m{I} \right)^{-1} \m{w}.
\end{align*}
When $\m{\Phi}$ is not strictly positive, then the normal equation does not have a unique solution, but nonetheless they all yield the same $\hat{\m{\mu}}_{\nu}$ by its uniqueness from \cref{prop:repre}. Thus, we might take $\hat{\m{\alpha}}^{\nu}=(\m{\Phi}+ \omega_{d_{r}, d}^{-1} n \nu \m{I})^{-1} \m{w}$.

For the covariance, note that $\langle \varphi_{i_{1}j_{1}k_{1}j'_{1}k'_{1}}^{\otimes}, \varphi_{i_{2}j_{2}k_{2}j'_{2}k'_{2}}^{\otimes} \rangle_{\b{H} \otimes \b{H}} = \langle \varphi_{i_{1}j_{1}k_{1}}, \varphi_{i_{2}j_{2}k_{2}} \rangle_{\b{H}} \langle \varphi_{i_{1}j'_{1}k'_{1}}, \varphi_{i_{2}j'_{2}k'_{2}} \rangle_{\b{H}}$. Similar to above, one can show that
\begin{align*}
    \hat{L}_{\eta}^{\odot}(G)
    &= \frac{\omega_{d_{r}, d}^{2}}{n} \|\m{w}^{\odot} - \m{\Phi}^{\odot} \hat{\m{\alpha}}_{\eta}^{\odot}\|^{2} + \eta (\hat{\m{\alpha}}_{\eta}^{\odot})^{\top} \m{\Phi} \hat{\m{\alpha}}_{\nu}^{\odot},
\end{align*}
thus
\begin{equation*}
    \hat{\m{\alpha}}_{\eta}^{\odot} = \left( \m{\Phi}^{\odot}+ \frac{n \eta}{\omega_{d_{r}, d}^{2}} \m{I} \right)^{-1} \m{w}^{\odot}.
\end{equation*}
\end{proof}

\subsection{Proofs in Section \ref{sec:Mercer:inv}}\label{sec:proof:Mercer:inv}

\begin{proof}[Proof of \cref{prop:swap:int:op}]
Let $g \in \c{L}_{2} (\b{G} \times \Omega^{d-d_{r}}, W_{d_{r}, d})$ be given. Since $\tilde{k}_{(R_{1}, \m{x}_{1})} = \s{P} \varphi(R_{1}, \m{x}_{1})$ from \cref{thm:induced:RKHS}, we have
\begin{align*}
    \s{T}_{\tilde{K}} g (R_{1}, \m{x}_{1}) &= \int_{\b{G}} \int_{\Omega^{d-d_{r}}} \tilde{K}((R_{1}, \m{x}_{1}), (R_{2}, \m{x}_{2})) g(R_{2}, \m{x}_{2}) W_{d_{r}, d}(\m{x}_{2}) \rd \m{x}_{2} \rd R_{2} \\
    &= \int_{\b{G}} \int_{\Omega^{d-d_{r}}} [\s{P} \varphi(R_{1}, \m{x}_{1})](R_{2}, \m{x}_{2}) g(R_{2}, \m{x}_{2}) W_{d_{r}, d}(\m{x}_{2}) \rd \m{x}_{2} \rd R_{2} \\
    &= \langle \s{P} \varphi(R_{1}, \m{x}_{1}), g \rangle_{\c{L}_{2} (\b{G} \times \Omega^{d-d_{r}}, W_{d_{r}, d})} = \langle \varphi(R_{1}, \m{x}_{1}), \s{P}^{*} g \rangle_{\c{L}_{2} (\Omega^{d}, W_{0, d})} \\
    &= \langle \varphi(R_{1}, \m{x}_{1}), \s{T}_{K} \s{P}^{*} g \rangle_{\b{H}}
    = [\s{P} \s{T}_{K} \s{P}^{*} g] (R_{1}, \m{x}_{1}),
\end{align*}
by the reproducing property. Therefore, $\s{T}_{\tilde{K}} = \s{P} \s{T}_{K} \s{P}^{*}$.
\end{proof}

\begin{proof}[Proof of \cref{prop:isometry}]
To show the uniqueness of $\s{K}$, it suffices to show \eqref{eq:isomet}. For any $f_{1}, f_{2} \in \b{H}$, by Theorem 7.2.5 in \cite{hsing2015theoretical},
\begin{align*}
    \omega_{d_{r}, d} \langle \b{E}[\varphi(\m{R}, \m{X}) \otimes \varphi(\m{R}, \m{X})] f_{1} , f_{2} \rangle_{\b{H}} 
    &= \omega_{d_{r}, d} \b{E}[ \langle \varphi(\m{R}, \m{X}) , f_{1} \rangle_{\b{H}} \langle \varphi(\m{R}, \m{X}) , f_{2} \rangle_{\b{H}} ] \\
    &= \omega_{d_{r}, d} \b{E}[ \s{P}_{\m{R}} f_{1} (\m{X}) \s{P}_{\m{R}} f_{2} (\m{X}) ] \\
    &= \int_{\b{G}} \int_{\Omega^{d-d_{r}}} \s{P}_{R} f_{1} ( \m{x}) \s{P}_{R} f_{2} ( \m{x}) W_{d_{r}, d}(\m{x}) \rd \m{x} \rd R \\
    &= \langle \s{P} f_{1} , \s{P} f_{2} \rangle_{\c{L}_{2} (\Omega^{d-d_{r}}, W_{d_{r}, d})}.
\end{align*}
Also, $\vertiii{\s{K}}_{1} = \b{E}[ \| \varphi(\m{R}, \m{X}) \|_{\b{H}}^{2}] \le B^{2}$ by Theorem 7.2.5 in \cite{hsing2015theoretical}, thus $\s{K}$ is of trace-class.
\end{proof}

\begin{proof}[Proof of \cref{prop:daughter:rk}]
For each $l \in \b{N}$,
\begin{align*}
    &\int_{\b{G}} \int_{\Omega^{d-d_{r}}} \tilde{K}_{l}((R_{1}, \m{x}_{1}), (R_{3}, \m{x}_{3})) \tilde{K}_{l}((R_{2}, \m{x}_{2}), (R_{3}, \m{x}_{3})) W_{d_{r}, d}(\m{x}) \rd \m{x} \rd R \\
    &= \sum_{m=1}^{N(l)} g_{lm}(R_{1}, \m{x}_{1}) \sum_{m'=1}^{N(l)} g_{lm'}(R_{2}, \m{x}_{2}) \langle g_{lm}, \s{P} g_{lm'} \rangle_{W_{d_{r}, d}} \\
    &= \sum_{m=1}^{N(l)} g_{lm}(R_{1}, \m{x}_{1}) g_{lm}(R_{2}, \m{x}_{2})
    = \tilde{K}_{l}((R_{1}, \m{x}_{1}), (R_{2}, \m{x}_{2})).
\end{align*}
For any $g \in \c{L}_{2} (\b{G} \times \Omega^{d-d_{r}}, W_{d_{r}, d})$, decompose it by
\begin{equation*}
    g = \sum_{l=1}^{\infty} g_{l} + g_{\perp} \in \bigoplus_{l =1}^{\infty} \tilde{\c{H}}_{l} \oplus \c{R}(\s{P} \circ \imath)^{\perp},
\end{equation*}
where $g_{l} \in \tilde{\c{H}}_{l}$ and $g_{\perp} \in \c{R}(\s{P} \circ \imath)^{\perp}$. Then
\begin{align*}
    \s{T}_{\tilde{K}_{l}} g (R_{1}, \m{x}_{1}) &= \int_{\b{G}} \int_{\Omega^{d-d_{r}}} \tilde{K}_{l}((R_{1}, \m{x}_{1}), (R_{2}, \m{x}_{2})) g_{l}(R_{2}, \m{x}_{2}) W_{d_{r}, d}(\m{x}) \rd \m{x} \rd R \\
    &= \sum_{m=1}^{N(l)} g_{lm}(R_{1}, \m{x}_{1}) \langle  g_{lm}, g \rangle_{W_{d_{r}, d}} = g_{l}(R_{1}, \m{x}_{1}),
\end{align*}
using the orthonormality of $\{g_{lm}: l \in \b{N}, 1 \le m \le N(l)\}$ in $\c{L}_{2} (\b{G} \times \Omega^{d-d_{r}}, W_{d_{r}, d})$.

Since $\b{H}_{l}$'s are invariant subspaces with respect to the unitary representation $(\rho, \b{H})$, it is trivial that $\tilde{K}_{l}((R, \m{x}), (R, \m{x}))$ is invariant to the choice of $R \in \b{G}$. Then,
\begin{align*}
    \int_{\Omega^{d-d_{r}}} \tilde{K}_{l}((R, \m{x}), (R, \m{x})) W_{d_{r}, d}(\m{x}) \rd \m{x} 
    &= \int_{\b{G}} \int_{\Omega^{d-d_{r}}} \tilde{K}_{l}((R, \m{x}), (R, \m{x})) W_{d_{r}, d}(\m{x}) \rd \m{x} \rd R \\
    &= \sum_{m=1}^{N(l)} \int_{\b{G}} \int_{\Omega^{d-d_{r}}} [g_{lm} (R, \m{x})]^{2} W_{d_{r}, d}(\m{x}) \rd \m{x} \rd R = N(l).
\end{align*}
\end{proof}

\subsection{Lemmas for Asymptotics}
\begin{example}\label{ex:radon:kernel}
Consider the Radon transform with $d \ge 2$, $d_{r} = d-1$, and $W_{d_{r}, d}(x) = w(x)^{-d_{r}}$, where $w(x) = \sqrt{1-x^{2}}$, along with a compatible kernel as described above.
For the Radon transform, it is more natural to index singular functions by a double array $(m, l)$ and $k$, where $m \in \b{N}_{0}$, $l \le m$ and $l-m$ is even, rather than using indices $l$ and $m$. The singular functions  are given by Zernike polynomials \cite{monard2019efficient, natterer2001mathematics}:
\begin{align*}
    g_{(m, l)k}(R, x) = c_{m}^{-1/2} Y_{lk}(R^{\top} e_{1}) C_{m}^{d/2}(x) w^{d-1}(x), \quad k=1, \dots, N(l),
\end{align*}
Here, $Y_{lk}: \b{S}^{d-1} \to \b{R}$ are the spherical harmonics of degree $l$ with the unit $\c{L}_{2}$-norm, $C_{m}^{d/2}$ is the Gegenbauer polynomial of degree $m$, and $c_{m}$ is the normalization constant
\begin{equation*}
    c_{m} = \int_{-1}^{+1} (C_{m}^{d/2}(x))^{2} w^{d-1}(x) \rd x = \frac{2^{d-1} \Gamma((d+1)/2)^{2} m!}{(m+d/2) \Gamma(m+d)},
\end{equation*}
The corresponding $(m, l)$-the kernel on the projection space $SO(d) \times [-1, +1]$ is then:
\begin{align*}
    \tilde{K}_{(m, l)} &((R_{1}, x_{1}), (R_{2}, x_{2})) \\
    &= \frac{N(l)}{c_{m} |\b{S}^{d-1}|} C_{l}^{d/2-1}(e_{1}^{\top} R_{1} R_{2}^{\top} e_{1}) [C_{m}^{d/2}(x_{1}) w^{d-1}(x_{1})] [C_{m}^{d/2}(x_{2}) w^{d-1}(x_{2})],
\end{align*}
and satisfies (A), as proven in \cref{lem:radon:kernel}. In particular, $\tilde{B}=\pi^{-2}$ when $d=2$.
\end{example}

\begin{lemma}\label{lem:radon:kernel}
Let $d \ge 2$, $m \in \b{N}_{0}$, $l \le m$ and $l-m$ is an even integer. Then the reproducing kernel $\tilde{K}_{(m, l)}$ defined in \cref{ex:radon:kernel} satisfy (A).
Moreover, in the case where $d=2$, we get the following sharp bound:
\begin{align*}
    \sup_{(m, l), R, x} \frac{\tilde{K}_{(m, l)} ((R, x), (R, x))}{N(l)} = \frac{1}{\pi^{2}}.
\end{align*}
\end{lemma}
\begin{proof}
Note that $\sum_{k=1}^{N(l)} Y_{lk}(\theta)^{2} = N(l)/|\b{S}^{d-1}|$ for any $\theta \in \b{S}^{d-1}$ \cite{dai2013approximation}, hence 
\begin{align*}
    \tilde{K}_{(m, l)} ((R, x), (R, x)) &= c_{m}^{-1} \left( \sum_{k=1}^{N(l)} Y_{lk}(R^{\top} e_{1})^{2} \right) [C_{m}^{d/2}(x) w^{d-1}(x)]^{2} \\
    &= c_{m}^{-1} N(l) [C_{m}^{d/2}(x) w^{d-1}(x)]^{2}/|\b{S}^{d-1}|.
\end{align*}
Using Legendre's Duplication formula $\Gamma((z+1)/2) \Gamma(z/2) = \sqrt{\pi} 2^{1-z} \Gamma(z)$ and (1.4) in \cite{forster1993inequalities} :
\begin{align*}
    \sup_{|x| \le 1} |C_{m}^{\lambda}(x) w^{2\lambda-1}(x)| \le \frac{\Gamma(m/2+\lambda)}{\Gamma(\lambda) \Gamma(m/2+1)} \frac{\Gamma(2\lambda) m!}{\Gamma(m+2\lambda)}, \quad \lambda \ge 1,
\end{align*}
we obtain
\begin{align*}
    &\sup_{(m, l), R, x} \frac{\tilde{K}_{(m, l)} ((R, x), (R, x))}{N(l)} \\ 
    &\le \sup_{m \in \b{N}_{0}} \left( \frac{\Gamma((m+d)/2)}{\Gamma(d/2) \Gamma(m/2+1)} \frac{\Gamma(d) m!}{\Gamma(m+d)} \right)^{2} \frac{(m+d/2) \Gamma(m+d)}{2^{d-1} (\Gamma((d+1)/2)^{2} m!} \frac{\Gamma(d/2)}{2 \pi^{d/2}} \\
    &= \frac{\Gamma(d/2)}{2 \pi^{d/2}} \frac{\Gamma(d)^{2}}{ [\Gamma((d+1)/2) \Gamma(d/2)]^{2}2^{d-1}} \sup_{m \in \b{N}_{0}} \left( \frac{\Gamma((m+d)/2)}{ \Gamma(m/2+1)} \right)^{2} \frac{m! (m+d/2)}{\Gamma(m+d)} \\
    &= \frac{\Gamma(d/2)}{2 \pi^{d/2}} \frac{2^{d-1}}{\pi} \sup_{m \in \b{N}_{0}}  \frac{(\Gamma((m+d)/2))^{2}}{ \Gamma(m+d)} \frac{m! (m+d/2)}{(\Gamma(m/2+1))^{2}} \\
    &= \frac{2^{d-2} \Gamma(d/2)}{\pi^{d/2+1}} \sup_{m \in \b{N}_{0}} \frac{\mathrm{Beta}((m+d)/2, (m+d)/2)}{\mathrm{Beta}(m/2+1, m/2+1)} \frac{m+d/2}{m+1} ,
\end{align*}
which is finite.
In particular, when $d=2$, $C_{m}^{1}$ is the Chebyshev polynomials of the second kind:
\begin{equation*}
    C_{m}^{1}(x) = \frac{\sin ((m+1) \arccos x)}{(m+1) \sin (\arccos x)} = \frac{\sin ((m+1) \arccos x)}{(m+1) w(x)},
\end{equation*}
and
\begin{align*}
    \sup_{(m, l), R, x} \frac{\tilde{K}_{(m, l)} ((R, x), (R, x))}{N(l)} &= \sup_{m, x} \frac{[\sin ((m+1) \arccos x)]^{2}}{c_{m} (m+1)^{2} |\b{S}^{1}|} = \sup_{m \in \b{N}_{0}} \frac{1}{2 \pi c_{m} (m+1)^{2}} \\
    &= \sup_{m \in \b{N}_{0}} \frac{(m+1) \Gamma(m+2)}{4 \pi (\Gamma(3/2))^{2} m! (m+1)^{2}}
    = \frac{1}{4 \pi (\Gamma(3/2))^{2}} = \frac{1}{\pi^{2}}.
\end{align*}
\end{proof}

\bigskip

\begin{lemma}\label{lem:srce:oper:norm}
Let $\s{K}$ be a compact, self-adjoint operator on a Hilbert space $\c{H}$. Then, for any $\kappa > 0$, $(\s{K} + \nu I)^{-1} \s{K}^{\kappa}$ is also a compact, self-adjoint operator with the operator norm satisfying
\begin{align*}
    \left\VERT (\s{K} + \nu I)^{-1} \s{K}^{\kappa} \right\VERT_{\infty} \le
    \begin{cases}
        \kappa^{\kappa} (1-\kappa)^{1-\kappa} \nu^{-(1-\kappa)}, \quad & 0 < \kappa < 1, \\
        \VERT \s{K} \VERT_{\infty}^{\kappa-1}, \quad & \kappa \ge 1.
    \end{cases}
\end{align*}
When $\kappa = 0$, $(\s{K} + \nu I)^{-1} \in \c{B}(\c{H})$ is a self-adjoint operator with $\VERT (\s{K} + \nu I)^{-1} \s{K}^{\kappa} \VERT_{\infty} \le \nu^{-1}$. Therefore, for a fixed $\kappa \ge 0$, $\VERT (\s{K} + \nu I)^{-1} \s{K}^{\kappa} \VERT_{\infty} \le \VERT \s{K} \VERT_{\infty}^{(\kappa-1) \vee 0} \nu^{(\kappa-1) \wedge 0}$.
\end{lemma}
\begin{proof}
Let
\begin{equation*}
    \s{K} = \sum_{l=0}^{\infty} \lambda_{l} e_{l} \otimes e_{l}
\end{equation*}
be the eigendecomposition of $\s{K}$, where the eigenfunctions $\{e_{l}\}_{l=0}^{\infty}$ is a CONS for $\c{H}$ with corresponding eigenvalues $\lambda_{1} \ge \lambda_{2} \ge \dots \ge 0$. Then for any $f = \sum_{l=0}^{\infty} f_{l} e_{l} \in \c{H}$,
\begin{equation*}
    (\s{K} + \nu I)^{-1} \s{K}^{\kappa} f = \sum_{l=0}^{\infty} \frac{\lambda_{l}^{\kappa}}{\lambda_{l} + \nu} f_{l} e_{l}.
\end{equation*}
If $0 < \kappa < 1$, H\"older's inequality yields
\begin{align*}
    \lambda_{l} + \nu = \frac{(\lambda_{l}^{\kappa}\kappa^{\kappa})^{1/\kappa}}{1/\kappa} + \frac{(\nu^{(1-\kappa)}(1-\kappa)^{(1-\kappa)})^{1/(1-\kappa)}}{1/(1-\kappa)} \ge \kappa^{\kappa} (1-\kappa)^{(1-\kappa)} \lambda_{l}^{\kappa} \nu^{(1-\kappa)},
\end{align*}
from which we obtain
\begin{equation*}
    \left\VERT (\s{K} + \nu I)^{-1} \s{K}^{\kappa} \right\VERT_{\infty} = \frac{\lambda_{1}^{\kappa}}{\lambda_{1} + \nu} \le \kappa^{\kappa} (1-\kappa)^{1-\kappa} \nu^{-(1-\kappa)}.
\end{equation*}
On the other hand, if $\kappa = 0$, then $1/(\lambda_{l} + \nu) \le \nu^{-1}$, and if $\kappa \ge 1$, then
\begin{equation*}
    \frac{\lambda_{l}^{\kappa}}{\lambda_{l} + \nu} \le \lambda_{l}^{\kappa-1}.
\end{equation*}
\end{proof}

For the covariance estimation, we recall that $J=(j,k) \, (1 \le j \le r_{i}, 1 \le k \le s_{ij})$ denote a merged index.
\begin{lemma}\label{lem:quad:cond:exp}
Define $\m{S}^{\otimes} := \b{E}[(\c{Y} \otimes \c{Y}) \otimes_{2} (\c{Y} \otimes \c{Y})] \in \c{B}_{1}(\b{H} \otimes \b{H})$, which is a trace-class operator with the norm
\begin{equation*}
    \vertiii{\m{S}^{\otimes}}_{1} = \b{E}[ \| \c{Y} \otimes \c{Y} \|_{\b{H} \otimes \b{H}}^{2}] = \b{E}[\|\c{Y}\|_{\b{H}}^{4}] \le M_{4}.
\end{equation*}
Then, for any $1 \le j_{1}, j'_{1}, j_{2}, j'_{2} \le r_{i}$ with $j_{1} \neq j'_{1}, j_{2} \neq j'_{2}$, we have 

\begin{align*} 
    &\b{E} \left[ \b{E} \left[ \left. Z_{iJ_{1}J'_{1}} Z_{iJ_{2}J'_{2}} \right\vert \varphi  \right] \langle \varphi_{iJ_{1}J'_{1}}^{\otimes} , e^{\otimes}_{lm, l'm'} \rangle_{\b{H} \otimes \b{H}} \langle \varphi_{ij_{2}j'_{2}}^{\otimes} , e^{\otimes}_{lm, l'm'} \rangle_{\b{H} \otimes \b{H}} \right] \\
    &\le
    \begin{cases}
    \omega_{d_{r}, d}^{-4} \langle \m{S}^{\otimes} e^{\otimes}_{lm, l'm'}, e^{\otimes}_{lm, l'm'} \rangle_{\b{H} \otimes \b{H}} \tau_{ll'}^{4}, \quad & |\{j_{1}, j'_{1}\} \cap \{j_{2}, j'_{2}\}| = 0, \\
    \omega_{d_{r}, d}^{-2} (B^{4} \vertiii{\m{S}^{\otimes}}_{1} + 2 \sigma_{r_{i}}^{2} B^{2} \vertiii{\m{\Gamma}}_{\infty} + \sigma_{r_{i}}^{4}) \tau_{ll'}^{2}, \quad & |\{j_{1}, j'_{1}\} \cap \{j_{2}, j'_{1}\}| \ge 1.
    \end{cases}
\end{align*}
\end{lemma}

\begin{proof}
Recall that for any $1 \le i \le n$ and $1 \le j \le r_{i}$, $Z_{iJ} = \langle \c{Y}_{i}, \varphi_{iJ} \rangle_{\b{H}} + \varepsilon_{iJ}$ where $\varepsilon_{iJ} \stackrel{iid}{\sim} \mathcal{N}(0, \sigma_{r_{i}}^{2})$.
\begin{enumerate}[(1)]
\item If $|\{j_{1}, j'_{1}\} \cap \{j_{2}, j'_{2}\}| = 0$,
\begin{align*} 
    \b{E} \left[ \left. Z_{iJ_{1}J'_{1}} Z_{iJ_{2}J'_{2}} \right\vert \varphi  \right] 
    &= \b{E} \left[ \left. Z_{iJ_{1}} Z_{iJ'_{1}} Z_{iJ_{2}} Z_{iJ'_{2}} \right\vert \varphi  \right] \\
    &= \b{E} \left[ \left. \langle \c{Y}_{i}, \varphi_{iJ_{1}} \rangle_{\b{H}} \langle \c{Y}_{i}, \varphi_{iJ'_{1}} \rangle_{\b{H}} \langle \c{Y}_{i}, \varphi_{iJ_{2}} \rangle_{\b{H}} \langle \c{Y}_{i}, \varphi_{iJ'_{2}} \rangle_{\b{H}} \right\vert \varphi  \right]
    = \langle \m{S}^{\otimes} \varphi_{iJ_{1}J'_{1}}^{\otimes}, \varphi_{iJ_{2}J'_{2}}^{\otimes} \rangle_{\b{H} \otimes \b{H}}.
\end{align*}
Hence,
\begin{align*}
    &\b{E} \left[ \b{E} \left[ \left. Z_{iJ_{1}J'_{1}} Z_{iJ_{2}J'_{2}} \right\vert \varphi  \right] \langle \varphi_{iJ_{1}J'_{1}}^{\otimes} , e^{\otimes}_{lm, l'm'} \rangle_{\b{H} \otimes \b{H}} \langle \varphi_{iJ_{2}J'_{2}}^{\otimes} , e^{\otimes}_{lm, l'm'} \rangle_{\b{H} \otimes \b{H}} \right] \\
    &= \langle \m{S}^{\otimes} \b{E}[ \langle \varphi_{iJ_{1}J'_{1}}^{\otimes} , e^{\otimes}_{lm, l'm'} \rangle_{\b{H} \otimes \b{H}} \varphi_{iJ_{1}J'_{1}}^{\otimes}], \b{E}[ \langle \varphi_{iJ_{2}J'_{2}}^{\otimes} , e^{\otimes}_{lm, l'm'} \rangle_{\b{H} \otimes \b{H}} \varphi_{iJ_{2}J'_{2}}^{\otimes}] \rangle_{\b{H} \otimes \b{H}} \\
    &= \omega_{d_{r}, d}^{-4} \langle \m{S}^{\otimes} \s{K}^{\otimes} e^{\otimes}_{lm, l'm'}, \s{K}^{\otimes} e^{\otimes}_{lm, l'm'} \rangle_{\b{H} \otimes \b{H}} = \omega_{d_{r}, d}^{-4} \tau_{ll'}^{4} \langle \m{S}^{\otimes} e^{\otimes}_{lm, l'm'}, e^{\otimes}_{lm, l'm'} \rangle_{\b{H} \otimes \b{H}}.
\end{align*}

\item If $|\{j_{1}, j'_{1}\} \cap \{j_{2}, j'_{2}\}| \ge 1$, by the Cauchy-Schwarz inequality,
\begin{align*} 
    \b{E} \left[ \left. Z_{iJ_{1}} Z_{iJ'_{1}} Z_{iJ_{2}} Z_{iJ'_{2}} \right\vert \varphi  \right]
    &\le \b{E} \left[ \left. Z_{iJ_{1}}^{2} Z_{iJ'_{1}}^{2}  \right\vert \varphi  \right] \\
    &= \b{E} \left[ \left. (\langle \c{Y}_{i}, \varphi_{iJ_{1}} \rangle_{\b{H}}^{2} + \varepsilon_{iJ_{1}}^{2}) (\langle \c{Y}_{i}, \varphi_{iJ'_{1}} \rangle_{\b{H}}^{2} + \varepsilon_{iJ'_{1}}^{2})   \right\vert \varphi  \right] \\
    &= \langle \m{S}^{\otimes} \varphi_{iJ_{1}J'_{1}}^{\otimes}, \varphi_{iJ_{1}J'_{1}}^{\otimes} \rangle_{\b{H} \otimes \b{H}} + \sigma_{r_{i}}^{2} \b{E} \left[ \left. \langle \c{Y}_{i}, \varphi_{iJ_{1}} \rangle_{\b{H}}^{2} + \langle \c{Y}_{i}, \varphi_{iJ'_{1}} \rangle_{\b{H}}^{2} \right\vert \varphi  \right] + \sigma_{r_{i}}^{4} \\
    &= \langle \m{S}^{\otimes} \varphi_{iJ_{1}J'_{1}}^{\otimes}, \varphi_{iJ_{1}J'_{1}}^{\otimes} \rangle_{\b{H} \otimes \b{H}} + \sigma_{r_{i}}^{2} \langle \m{\Gamma} \varphi_{iJ_{1}}, \varphi_{iJ_{1}} \rangle_{\b{H}} + \sigma_{r_{i}}^{2} \langle \m{\Gamma} \varphi_{iJ'_{1}}, \varphi_{iJ'_{1}} \rangle_{\b{H}} + \sigma_{r_{i}}^{4}  \\
    &\le B^{4} \vertiii{\m{S}^{\otimes}}_{1} + 2 \sigma_{r_{i}}^{2} B^{2} \vertiii{\m{\Gamma}}_{\infty} + \sigma_{r_{i}}^{4}.
\end{align*}

Hence,
\begin{align*}
    &\b{E} \left[ \b{E} \left[ \left. Z_{iJ_{1}J'_{1}}^{2} \right\vert \varphi  \right] \langle \varphi_{iJ_{1}J'_{1}}^{\otimes} , e^{\otimes}_{lm, l'm'} \rangle_{\b{H} \otimes \b{H}}^{2} \right]
    \le (B^{4} \vertiii{\m{S}^{\otimes}}_{1} + 2 \sigma_{r_{i}}^{2} B^{2} \vertiii{\m{\Gamma}}_{1} + \sigma_{r_{i}}^{4}) \b{E} \left[ \langle \varphi_{iJ_{1}J'_{1}}^{\otimes} , e^{\otimes}_{lm, l'm'} \rangle_{\b{H} \otimes \b{H}}^{2} \right] \\
    &= \omega_{d_{r}, d}^{-2} (B^{4} \vertiii{\m{S}^{\otimes}}_{1} + 2 \sigma_{r_{i}}^{2} B^{2} \vertiii{\m{\Gamma}}_{1} + \sigma_{r_{i}}^{4}) \langle \s{K}^{\otimes} e^{\otimes}_{lm, l'm'} , e^{\otimes}_{lm, l'm'} \rangle_{\b{H} \otimes \b{H}}^{2} \\
    &= \omega_{d_{r}, d}^{-2} (B^{4} \vertiii{\m{S}^{\otimes}}_{1} + 2 \sigma_{r_{i}}^{2} B^{2} \vertiii{\m{\Gamma}}_{1} + \sigma_{r_{i}}^{4}) \tau_{ll'}^{2}.
\end{align*}
\end{enumerate}
\end{proof}

\subsection{Proofs in Section \ref{sec:asym:mean}}\label{sec:proof:asym:mean}
For the $p$-Schatten norm with $p \in [1, \infty)$, it holds that $\vertiii{\s{K}}_{p} = ( \sum_{l=1}^{\infty} |\lambda_{l}|^{p} )^{1/p}$, where $\lambda_{l}$ are singular values of $\s{K}$, from which the following properties follow, which we use in the proofs of asymptotic theory:
\begin{enumerate}
    \item (Monotonicity) For $1 \le p \le q \le \infty$, we have $\VERT \s{K} \VERT_{1} \ge \VERT \s{K} \VERT_{p} \ge \VERT \s{K} \VERT_{q} \ge \VERT \s{K} \VERT_{\infty}$.
    \item (H\"older) For $p, q \in [1, \infty]$ with $p^{-1} + q^{-1} =1$ ($q=\infty$ if $p=1$), we have $\VERT \s{K}_{1} \s{K}_{2} \VERT_{1} \le \VERT \s{K}_{1} \VERT_{p} \VERT \s{K}_{2} \VERT_{q}$.
\end{enumerate}

It is natural to expect that the convergence rate gets better as $\beta \ge 0$ increases as $\s{K}$ is a compact operator. However, when $\beta = 0$, it can be readily seen that the source regularity condition with $\kappa_{1} > 0$ is necessary to obtain a meaningful convergence rate. On the other hand, once $\kappa_{1}$ is greater than or equal to $1/2$, \cref{thm:gen:asym:mean} reveals that a greater $\kappa_{1}$ no longer affects the convergence rate. In this regard, we only consider $0 \le \kappa_{1} \le 1/2$ in \cref{def:srce:reg}.

\begin{theorem}\label{thm:gen:asym:mean}
Consider a sequence of the penalty parameter $\nu_{n} > 0$ converging to $0$. Assume (A) and that there is some $\zeta \in (0, 1/2]$ with $\VERT \s{K}^{1-2\zeta} \VERT_{1} < \infty$ and $\VERT (\s{K} + \nu_{n} I)^{-2} \s{K}^{2\zeta+1} \VERT_{1}=o(n \cdot \min(\overline{r},
\overline{r}^{\sigma^{2}}))$. Then, for any $\beta \in [0, 1/2]$ satisfying $\VERT (\s{K} + \nu_{n} I)^{-2} \s{K}^{2\beta+1} \VERT_{1}= o(n \cdot \min(\overline{r},
\overline{r}^{\sigma^{2}}))$, it holds that
\begin{align*}
    \lim_{D \rightarrow \infty} \limsup_{n \rightarrow \infty} \sup_{\b{P}_{\c{Y}} \in \c{S}_{1}(\kappa_{1}, M_{1}, M_{2})} \b{P}_{\c{Y}} \left[ \| \s{K}^{\beta} (\hat{\m{\mu}}_{\nu_{n}}-\m{\mu}) \|_{\b{H}}^{2} > D \max(c_{n}, a_{n}, n^{-1}) \right] = 0,
\end{align*}
where $c_{n} := \nu_{n}^{2[(\beta + \kappa_{1}) \wedge 1]}$ and $a_{n} := [n \cdot \min(\overline{r},
\overline{r}^{\sigma^{2}})]^{-1} \VERT (\s{K} + \nu_{n} I)^{-2} \s{K}^{2\beta+1} \VERT_{1}$.
\end{theorem}

\begin{proof}
Let
\begin{align*}
    &\bar{\m{\mu}}_{\nu} = (\s{K} + \nu I)^{-1} \s{K} \m{\mu}, \\
    &\tilde{\m{\mu}}_{\nu} = \bar{\m{\mu}}_{\nu} - (\s{K} + \nu I)^{-1} (-\c{Z}_{n} + (\hat{\s{K}} + \nu I) \bar{\m{\mu}}_{\nu}), \\
    &\hat{\m{\mu}}_{\nu} = (\hat{\s{K}} + \nu I)^{-1} \c{Z}_{n}.    
\end{align*}
To prove the statement, we decompose as follows:
\begin{equation*}
    (\hat{\m{\mu}}_{\nu}-\m{\mu}) = (\bar{\m{\mu}}_{\nu}-\m{\mu}) + (\tilde{\m{\mu}}_{\nu}-\bar{\m{\mu}}_{\nu}) + (\hat{\m{\mu}}_{\nu}-\tilde{\m{\mu}}_{\nu}).
\end{equation*}

The idea behind this decomposition is as follows. The first term is deterministic and easy to control. Also, the second part does not involve the inverse of the random operator $(\hat{\s{K}} + \nu I)$, which we do not want to work with. Then, we control the second moment of the second term using the spectral decomposition of $\s{K}$. Finally, we impose certain regularity conditions so that the final term, which involves $(\hat{\s{K}} + \nu I)^{-1}$ could be majorized by the second term. It will turn out that the first term is majorizing when the penalty parameter is large, whereas the second term is majorizing when the penalty parameter is small.

Additionally, We list some equations that we will use several times in the proof below. Here, $f \in \b{H}$ and $\c{Y}$ is a second-order $\b{H}$-valued random element. In the proof, $\b{E}[\cdot \vert \varphi]$ should be understood as the conditional expectation given all $\varphi_{ijk}$'s.
\begin{align}
    &\b{E} \|\c{Y}\|_{\b{H}}^{2} = \sum_{l=1}^{\infty} \sum_{m=1}^{N(l)} \b{E} [ \langle \c{Y}, e_{lm} \rangle_{\b{H}}^{2} ] \label{eq:parseval} \\
    &\b{E} [\left. \c{Z}_{n} \right\vert \varphi] = \b{E} \left[\left. \frac{\omega_{d_{r}, d}}{n} \sum_{i=1}^{n} \frac{1}{r_{i}} \sum_{j=1}^{r_{i}} \frac{1}{s_{ij}} \sum_{k=1}^{s_{ij}} Z_{ijk} \varphi_{ijk} \right\vert \varphi \right] = \hat{\s{K}} \m{\mu} \label{eq:Wn:cond:exp}, \\
    &\Var \left[ \langle (\varphi_{ijk} \otimes \varphi_{ijk}) f, e_{lm} \rangle_{\b{H}} \right] 
    \le \b{E} \left[ \langle \varphi_{ijk}, f \rangle_{\b{H}}^{2} \langle \varphi_{ijk}, e_{lm} \rangle_{\b{H}}^{2}  \right] \nonumber \\
    &\hspace*{6em} \quad \le B^{2} \|f\|_{\b{H}}^{2} \b{E} \left[ \langle \varphi_{ijk}, e_{lm} \rangle_{\b{H}}^{2}  \right] = \frac{B^{2}}{\omega_{d_{r}, d}} \|f\|_{\b{H}}^{2} \langle \s{K} e_{lm}, e_{lm} \rangle_{\b{H}} = \frac{\tau_{l}^{2} B^{2}}{\omega_{d_{r}, d}} \|f\|_{\b{H}}^{2} \label{eq:var:lk}, \\
    &\vertiii{(\s{K} + \nu I)^{-2} \s{K}^{\alpha}}_{1} = \sum_{l=1}^{\infty} \sum_{m=1}^{N(l)} \langle (\s{K} + \nu I)^{-2} \s{K}^{\alpha} e_{lm}, e_{lm} \rangle_{\b{H}} = \sum_{l=1}^{\infty} \left( \frac{\tau_{l}^{\alpha}}{\tau_{l}^{2} + \nu} \right)^{2} N(l).
    \label{eq:tr:norm}
\end{align}
Also, for any Hilbert-Schmidt operator $G \in \b{H} \otimes \b{H}$, we have
\begin{align}\label{eq:HS:inn:pr}
    \omega_{d_{r}, d} \b{E} \left[ \langle G \varphi_{ij}, \varphi_{ij} \rangle_{\b{H}} \right] = \langle G, \s{K} \rangle_{\b{H} \otimes \b{H}} = \sum_{l=1}^{\infty} \sum_{m=1}^{N(l)} \langle G e_{lm}, \s{K} e_{lm} \rangle_{\b{H}}
    = \sum_{l=1}^{\infty} \sum_{m=1}^{N(l)} \tau_{l}^{2} \langle G e_{lm}, e_{lm} \rangle_{\b{H}}.
\end{align}

\begin{enumerate}
\item [(Step 1)] For any $\beta \ge 0, \nu > 0$, we claim that
\begin{equation*}
    \sup_{\b{P}_{\c{Y}} \in \c{S}_{1}(\kappa_{1}, M_{1}, M_{2})} \| \s{K}^{\beta} (\m{\mu} - \bar{\m{\mu}}_{\nu}) \|_{\b{H}}^{2} = O(\nu^{2[(\beta + \kappa_{1}) \wedge 1]}).
\end{equation*} 

Then \eqref{eq:mean:asym:1st} and \cref{lem:srce:oper:norm} yield
\begin{align}\label{eq:step1}
    \| \s{K}^{\beta} (\m{\mu} - \bar{\m{\mu}}_{\nu}) \|_{\b{H}}^{2} 
    &= \| \nu (\s{K} + \nu I)^{-1} \s{K}^{\beta} \m{\mu} \|_{\b{H}}^{2}
    = \nu^{2} \| (\s{K} + \nu I)^{-1} \s{K}^{\beta + \kappa_{1}} \delta \|_{\b{H}}^{2} \nonumber \\
    &\le (\nu^{(\beta + \kappa_{1}) \wedge 1} \vertiii{\s{K}}_{\infty}^{(\beta + \kappa_{1}-1) \vee 0} \|\delta \|_{\b{H}})^{2}
    \le \tilde{M}_{1} \nu^{2[(\beta + \kappa_{1}) \wedge 1]},
\end{align}
where $\tilde{M}_{1} := M_{1} \vertiii{\s{K}}_{\infty}^{2(\beta + \kappa_{1}-1) \vee 0}$.

\item [(Step 2)] For any $\beta \ge 0, \nu > 0$, we claim that
\begin{equation*}
    \sup_{\b{P}_{\c{Y}} \in \c{S}_{1}(\kappa_{1}, M_{1}, M_{2})} \b{E} \left[ \| \s{K}^{\beta} (\tilde{\m{\mu}}_{\nu}-\bar{\m{\mu}}_{\nu}) \|_{\b{H}}^{2} \right] = O ((n\overline{r})^{-1} \vertiii{ (\s{K} + \nu I)^{-2} \s{K}^{2\beta+1} }_{1} + n^{-1}).
\end{equation*} 

First, we remark that whenever $\beta \ge 0$ and $\nu > 0$, $(\s{K} + \nu I)^{-2} \s{K}^{2\beta+1}$ is a trace-class operator since
\begin{align*}
    \vertiii{ (\s{K} + \nu I)^{-2} \s{K}^{2\beta+1} }_{1} \le \vertiii{ (\s{K} + \nu I)^{-1} \s{K}^{\beta} }_{\infty}^{2} \vertiii{ \s{K}}_{1} < \infty.
\end{align*}
by \cref{lem:srce:oper:norm}.
Observe $\tilde{\m{\mu}}_{\nu}-\bar{\m{\mu}}_{\nu} = (\s{K} + \nu I)^{-1} (-\c{Z}_{n} + (\hat{\s{K}} + \nu I) \bar{\m{\mu}}_{\nu})$.
For any $\beta \ge 0$, we obtain
\begin{align}\label{eq:step2:decomp}
    \b{E} \left[ \| \s{K}^{\beta} (\tilde{\m{\mu}}_{\nu}-\bar{\m{\mu}}_{\nu}) \|_{\b{H}}^{2} \right]
    &\stackrel{\eqref{eq:parseval}}{=} \sum_{l=1}^{\infty} \sum_{m=1}^{N(l)} \b{E} [ \langle \s{K}^{2 \beta} (\s{K} + \nu I)^{-1} (-\c{Z}_{n} + (\hat{\s{K}} + \nu I) \bar{\m{\mu}}_{\nu}), e_{lm} \rangle_{\b{H}}^{2} ] \nonumber \\
    &= \sum_{l=1}^{\infty} \sum_{m=1}^{N(l)} \b{E} [ \langle  -\c{Z}_{n} + (\hat{\s{K}} + \nu I) \bar{\m{\mu}}_{\nu}, \s{K}^{2 \beta} (\s{K} + \nu I)^{-1} e_{lm} \rangle_{\b{H}}^{2} ] \nonumber \\
    &= \sum_{l=1}^{\infty} \left( \frac{\tau_{l}^{2 \beta}}{\tau_{l}^{2} + \nu} \right)^{2} \sum_{m=1}^{N(l)} \b{E} [ \langle -\c{Z}_{n} + (\hat{\s{K}} + \nu I) \bar{\m{\mu}}_{\nu}, e_{lm} \rangle_{\b{H}}^{2} ] \nonumber \\
    &= \sum_{l=1}^{\infty} \left( \frac{\tau_{l}^{2 \beta}}{\tau_{l}^{2} + \nu} \right)^{2} \sum_{m=1}^{N(l)} \Var [ \langle \c{Z}_{n} - \hat{\s{K}} \bar{\m{\mu}}_{\nu}, e_{lm} \rangle_{\b{H}} ],
\end{align}
where the last equality follows from the fact that $\bar{\m{\mu}}_{\nu}$ is deterministic and
\begin{align*}
    \b{E} \left[ \c{Z}_{n} - (\hat{\s{K}} + \nu I) \bar{\m{\mu}}_{\nu} \right] 
    &= \b{E} \left[ \b{E} [ \left.  \c{Z}_{n} - (\hat{\s{K}} + \nu I) \bar{\m{\mu}}_{\nu} \right\vert \varphi ] \right] \\
    &\stackrel{\eqref{eq:Wn:cond:exp}}{=} \b{E} \left[ \hat{\s{K}} \m{\mu} - (\hat{\s{K}} + \nu I) \bar{\m{\mu}}_{\nu} \right]
    = \s{K} \m{\mu} - (\s{K} + \nu I) \bar{\m{\mu}}_{\nu} = 0.
\end{align*}
To calculate \eqref{eq:step2:decomp}, we first calculate the summand for fixed $l$ and $m$, using the law of total variance:
\begin{align*}
    \Var [ \langle \c{Z}_{n} - \hat{\s{K}} \bar{\m{\mu}}_{\nu}, e_{lm} \rangle_{\b{H}} ]
    = \underbrace{\Var \left( \b{E} \left[ \left. \langle \c{Z}_{n} - \hat{\s{K}} \bar{\m{\mu}}_{\nu}, e_{lm} \rangle_{\b{H}} \right\vert \varphi \right] \right)}_{=: (A)_{lm}} + \underbrace{\b{E} \left( \Var \left[ \left. \langle \c{Z}_{n} - \hat{\s{K}} \bar{\m{\mu}}_{\nu}, e_{lm} \rangle_{\b{H}} \right\vert \varphi \right] \right)}_{=: (B)_{lm}}.
\end{align*}
Because $\varphi_{ijk}$ and $\varphi_{ij'k'}$ are i.i.d. whenever $j \neq j'$, note that
\begin{align*}
    (A)_{lm} &\stackrel{\eqref{eq:Wn:cond:exp}}{=} \Var \left[ \langle \hat{\s{K}} (\m{\mu} - \bar{\m{\mu}}_{\nu}) , e_{lm} \rangle_{\b{H}} \right] \\
    &= \left( \frac{\omega_{d_{r}, d}}{n} \right)^{2} \sum_{i=1}^{n} \frac{1}{r_{i}^{2}} \sum_{j=1}^{r_{i}} \frac{1}{s_{ij}^{2}} \Var \left[ \sum_{k=1}^{s_{ij}}  \langle (\varphi_{ijk} \otimes \varphi_{ijk}) (\m{\mu} - \bar{\m{\mu}}_{\nu}), e_{lm} \rangle_{\b{H}} \right] \\
    &\le \left( \frac{\omega_{d_{r}, d}}{n} \right)^{2} \sum_{i=1}^{n} \frac{1}{r_{i}^{2}} \sum_{j=1}^{r_{i}} \frac{1}{s_{ij}} \sum_{k=1}^{s_{ij}} \b{E} \left[ \langle (\varphi_{ijk} \otimes \varphi_{ijk}) (\m{\mu} - \bar{\m{\mu}}_{\nu}), e_{lm} \rangle_{\b{H}}^{2} \right] \\
    &\stackrel{\eqref{eq:var:lk}}{\le} \frac{\tau_{l}^{2} B^{2} \omega_{d_{r}, d}}{n^{2}} \sum_{i=1}^{n} \frac{1}{r_{i}^{2}} \sum_{j=1}^{r_{i}} \|\m{\mu} - \bar{\m{\mu}}_{\nu}\|_{\b{H}}^{2} \\
    &\stackrel{\eqref{eq:step1}}{\le} \frac{\tau_{l}^{2} B^{2} \omega_{d_{r}, d}}{n^{2}} \sum_{i=1}^{n} \frac{\nu^{2(\kappa_{1} \wedge 1)} \tilde{M}_{1}}{r_{i}}
    = \tilde{M}_{1} B^{2} \omega_{d_{r}, d} \nu^{2(\kappa_{1} \wedge 1)} \frac{\tau_{l}^{2}}{n \overline{r}}.
\end{align*}

To calculate $(B)_{lm}$, we first derive the conditional variance:
\begin{align}\label{eq:step2:cond:var}
    &\Var \left[ \left. \langle \c{Z}_{n} - \hat{\s{K}} \bar{\m{\mu}}_{\nu}, e_{lm} \rangle_{\b{H}} \right\vert \varphi \right] = \Var \left[ \left. \langle \c{Z}_{n}, e_{lm} \rangle_{\b{H}} \right\vert \varphi \right] \nonumber \\
    &= \left( \frac{\omega_{d_{r}, d}}{n} \right)^{2} \sum_{i=1}^{n} \frac{1}{r_{i}^{2}} \Var \left[ \left. \sum_{j=1}^{r_{i}} \frac{1}{s_{ij}} \sum_{k=1}^{s_{ij}} \langle Z_{ijk} \varphi_{ijk}, e_{lm} \rangle_{\b{H}} \right\vert \varphi \right] \nonumber \\
    &\le \left( \frac{\omega_{d_{r}, d}}{n} \right)^{2} \sum_{i=1}^{n} \frac{1}{r_{i}^{2}} \sum_{(j,k,j',k') \in \c{I}^{i}} \frac{1}{s_{ij} s_{ij'}} \b{E} \left[ \left. Z_{ijk} Z_{ij'k'} \right\vert \varphi \right] \langle \varphi_{ijk}, e_{lm} \rangle_{\b{H}} \langle  \varphi_{ij'k'}, e_{lm} \rangle_{\b{H}},
\end{align}
where $\c{I}^{i}$ is the index set for $(j,k,j',k')$ given by
\begin{equation*}
    \c{I}^{i} := \{ (j,k,j',k') \in \b{N}^{4} : 1 \le j, j' \le r_{i}, 1 \le k \le s_{ij}, 1 \le k' \le s_{ij'} \}.
\end{equation*}
Now, consider the partitions of $\c{I}^{i}$ as follows:
\begin{align*}
    \c{I}^{i}_{\neq} := \{ (j,k,j',k') \in \c{I}^{i} : j \neq j' \}, \quad \c{I}^{i}_{=} := \{ (j,k,j',k') \in \c{I}^{i} : j = j' \}
\end{align*}

We remark that when $(j,k,j',k') \in \c{I}^{i}_{\neq}$, by the independence of $\varphi_{ijk}$ and $\varphi_{ij'k'}$, 
\begin{align*}
    &\b{E} \left[ \b{E} [ \left. Z_{ijk} Z_{ij'k'} \right\vert \varphi ] \langle \varphi_{ijk}, e_{lm} \rangle_{\b{H}} \langle  \varphi_{ij'k'}, e_{lm} \rangle_{\b{H}} \right]
    = \b{E} \left[ \langle \m{\Gamma} \varphi_{ijk}, \varphi_{ij'k'} \rangle \langle \varphi_{ijk}, e_{lm} \rangle_{\b{H}} \langle  \varphi_{ij'k'}, e_{lm} \rangle_{\b{H}} \right] \\
    &= \omega_{d_{r}, d}^{-2} \langle \m{\Gamma} \b{E} \left[ \omega_{d_{r}, d} \langle \varphi_{ijk}, e_{lm} \rangle_{\b{H}} \varphi_{ijk} \right], \b{E} \left[ \omega_{d_{r}, d} \langle \varphi_{ij'k'}, e_{lm} \rangle_{\b{H}} \varphi_{ij'k'} \right] \rangle_{\b{H}} \\
    &= \omega_{d_{r}, d}^{-2} \langle \m{\Gamma} \s{K} e_{lm}, \s{K} e_{lm} \rangle_{\b{H}} = \omega_{d_{r}, d}^{-2} \tau_{l}^{4} \langle \m{\Gamma} e_{lm}, e_{lm} \rangle_{\b{H}}.
\end{align*}
In the case where $(j,k,j',k') \in \c{I}^{i}_{=}$, by the Cauchy-Schwarz inequality,
\begin{align*}
    &\b{E} \left[ \b{E} [ \left. Z_{ijk} Z_{ij'k'} \right\vert \varphi ] \langle \varphi_{ijk}, e_{lm} \rangle_{\b{H}} \langle  \varphi_{ij'k'}, e_{lm} \rangle_{\b{H}} \right]
    \le \b{E} \left[ \b{E} [ \left. Z_{ijk}^{2} \right\vert \varphi ] \langle \varphi_{ijk}, e_{lm} \rangle_{\b{H}}^{2} \right] \\
    &= \b{E} \left[ (\langle \m{\Gamma} \varphi_{ijk}, \varphi_{ijk} \rangle + \sigma_{r_{i}}^{2}) \langle \varphi_{ijk}, e_{lm} \rangle_{\b{H}}^{2} \right]
    \le (B^{2} \vertiii{\m{\Gamma}}_{\infty} + \sigma_{r_{i}}^{2}) \b{E} \left[ \langle \varphi_{ijk}, e_{lm} \rangle_{\b{H}}^{2} \right] \\
    &= \omega_{d_{r}, d}^{-1} (B^{2} \vertiii{\m{\Gamma}}_{\infty} + \sigma_{r_{i}}^{2}) \langle \s{K} e_{lm}, e_{lm} \rangle_{\b{H}} = \omega_{d_{r}, d}^{-1} \tau_{l}^{2} (B^{2} \vertiii{\m{\Gamma}}_{\infty} + \sigma_{r_{i}}^{2}).
\end{align*}

From \eqref{eq:step2:cond:var} and the above inequalities, we get
\begin{small}
\begin{align*}
    (B)_{lm} 
    &= \b{E} \left( \Var \left[ \left. \langle \c{Z}_{n} - \hat{\s{K}} \bar{\m{\mu}}_{\nu}, e_{lm} \rangle_{\b{H}} \right\vert \varphi \right] \right) \\
    &\le \left( \frac{\omega_{d_{r}, d}}{n} \right)^{2} \sum_{i=1}^{n} \frac{1}{r_{i}^{2}} \sum_{(j,k,j',k') \in \c{I}^{i}} \frac{1}{s_{ij} s_{ij'}} \b{E} \left[ \b{E} [ \left. Z_{ijk} Z_{ij'k'} \right\vert \varphi ] \langle \varphi_{ijk}, e_{lm} \rangle_{\b{H}} \langle  \varphi_{ij'k'}, e_{lm} \rangle_{\b{H}} \right] \\
    &\le \left( \frac{\omega_{d_{r}, d}}{n} \right)^{2} \sum_{i=1}^{n} \frac{1}{r_{i}^{2}} \left( \sum_{(j,k,j',k') \in \c{I}^{i}_{\neq}} \frac{\omega_{d_{r}, d}^{-2} \tau_{l}^{4} \langle \m{\Gamma} e_{lm}, e_{lm} \rangle_{\b{H}}}{s_{ij} s_{ij'}}  + \sum_{(j,k,j',k') \in \c{I}^{i}_{=}} \frac{\omega_{d_{r}, d}^{-1} \tau_{l}^{2} (B^{2} \vertiii{\m{\Gamma}}_{\infty} + \sigma_{r_{i}}^{2})}{s_{ij} s_{ij'}}  \right) \\
    &\le \frac{1}{n} \langle \m{\Gamma} e_{lm}, e_{lm} \rangle_{\b{H}} \tau_{l}^{4} + \frac{\omega_{d_{r}, d} B^{2} \vertiii{\m{\Gamma}}_{\infty}}{n \overline{r}} \tau_{l}^{2} + \frac{\omega_{d_{r}, d} \sigma^{2}}{n \overline{r}^{\sigma^{2}}} \tau_{l}^{2},
\end{align*}    
\end{small}
where we have used in the last inequality that 
\begin{equation*}
    \sum_{(j,k,j',k') \in \c{I}^{i}} \frac{1}{s_{ij} s_{ij'}} = \left( \sum_{j=1}^{r_{i}} \sum_{k=1}^{s_{ij}} \frac{1}{s_{ij}} \right)^{2} = r_{i}^{2}, \quad  \sum_{(j,k,j',k') \in \c{I}^{i}_{=}} \frac{1}{s_{ij} s_{ij'}} = \sum_{j=1}^{r_{i}} \sum_{k=1}^{s_{ij}} \sum_{k'=1}^{s_{ij}} \frac{1}{s_{ij}^{2}} = r_{i}.
\end{equation*}
Therefore,
\begin{align*}
    &\Var [ \langle \c{Z}_{n} - \hat{\s{K}} \bar{\m{\mu}}_{\nu}, e_{lm} \rangle_{\b{H}} ]
    = (A)_{lm} + (B)_{lm} \\
    &\le \omega_{d_{r}, d} (\tilde{M}_{1} B^{2} \nu^{2(\kappa_{1} \wedge 1)} + B^{2} \vertiii{\m{\Gamma}}_{\infty}) \frac{\tau_{l}^{2}}{n\overline{r}} + \omega_{d_{r}, d} \sigma^{2} \frac{\tau_{l}^{2}}{n \overline{r}^{\sigma^{2}}} + \langle \m{\Gamma} e_{lm}, e_{lm} \rangle_{\b{H}} \frac{\tau_{l}^{4}}{n}.
\end{align*}
Finally, \eqref{eq:step2:decomp} becomes
\begin{small}
\begin{align}\label{eq:step2:conclud}
    \b{E} &\left[ \| \s{K}^{\beta} (\tilde{\m{\mu}}_{\nu}-\bar{\m{\mu}}_{\nu}) \|_{\b{H}}^{2} \right] \\
    &= \sum_{l=1}^{\infty} \left( \frac{\tau_{l}^{2 
    \beta}}{\tau_{l}^{2} + \nu} \right)^{2} \sum_{m=1}^{N(l)} \Var [ \langle \c{Z}_{n} - \hat{\s{K}} \bar{\m{\mu}}_{\nu}, e_{lm} \rangle_{\b{H}} ] \nonumber \\
    &\le \sum_{l=1}^{\infty} \left( \frac{\tau_{l}^{2 \beta}}{\tau_{l}^{2} + \nu} \right)^{2} \sum_{m=1}^{N(l)} \left( \omega_{d_{r}, d}  \left( \tilde{M}_{1} B^{2} \nu^{2(\kappa_{1} \wedge 1)} + B^{2} \vertiii{\m{\Gamma}}_{\infty} + \sigma^{2} \right) \frac{\tau_{l}^{2}}{n \cdot \min(\overline{r},
    \overline{r}^{\sigma^{2}})} + \langle \m{\Gamma} e_{lm}, e_{lm} \rangle_{\b{H}} \frac{\tau_{l}^{4}}{n} \right) \nonumber \\
    &\stackrel{\eqref{eq:tr:norm}}{\le}  \frac{\omega_{d_{r}, d}  \left( \tilde{M}_{1} B^{2} \nu^{2(\kappa_{1} \wedge 1)} + B^{2} \vertiii{\m{\Gamma}}_{\infty} + \sigma^{2} \right)}{n \cdot \min(\overline{r},
    \overline{r}^{\sigma^{2}})} \| (\s{K} + \nu I)^{-2} \s{K}^{2\beta+1}\|_{1} 
    + \frac{1}{n} \sum_{l=1}^{\infty} \left( \frac{\tau_{l}^{2 \beta+2}}{\tau_{l}^{2} + \nu} \right)^{2} \sum_{m=1}^{N(l)} \langle \m{\Gamma} e_{lm}, e_{lm} \rangle_{\b{H}} \nonumber \\
    &= \frac{\omega_{d_{r}, d}  \left( \tilde{M}_{1} B^{2} \nu^{2(\kappa_{1} \wedge 1)} + B^{2} \vertiii{\m{\Gamma}}_{\infty} + \sigma^{2} \right)}{n \cdot \min(\overline{r},
    \overline{r}^{\sigma^{2}})} \vertiii{(\s{K} + \nu I)^{-2} \s{K}^{2\beta+1}}_{1} + \frac{1}{n} \vertiii{(\s{K} + \nu I)^{-2} \s{K}^{2\beta+2} \m{\Gamma} }_{1} \nonumber \\
    &\le \frac{\omega_{d_{r}, d}  \left( \tilde{M}_{1} B^{2} \nu^{2(\kappa_{1} \wedge 1)} + B^{2} M_{2} + \sigma^{2} \right)}{n \cdot \min(\overline{r},
    \overline{r}^{\sigma^{2}})} \vertiii{ (\s{K} + \nu I)^{-2} \s{K}^{2\beta+1} }_{1} + \frac{M_{2}}{n} \vertiii{(\s{K} + \nu I)^{-2} \s{K}^{2\beta+2} }_{\infty} \nonumber \\
    &= O ((n \cdot \min(\overline{r},
    \overline{r}^{\sigma^{2}}))^{-1} \vertiii{ (\s{K} + \nu I)^{-2} \s{K}^{2\beta+1} }_{1} + n^{-1}), \nonumber 
\end{align}
\end{small}
where the first term in the last inequality is justified by H\"older's inequality for the Schatten norm, and the second term is due to $\vertiii{(\s{K} + \nu I)^{-1} \s{K}^{\beta+1}}_{\infty} \le \vertiii{\s{K}}_{\infty}^{\beta}$ using \cref{lem:srce:oper:norm}.

\item [(Step 3)] Denote 
\begin{align*}
    &a_{n} := [n \cdot \min(\overline{r},
    \overline{r}^{\sigma^{2}})]^{-1} \vertiii{ (\s{K} + \nu I)^{-2} \s{K}^{2\beta+1} }_{1}, \\
    &b_{n} := [n \cdot \min(\overline{r},
    \overline{r}^{\sigma^{2}})]^{-1} \vertiii{ (\s{K} + \nu I)^{-2} \s{K}^{2\zeta+1} }_{1} + n^{-1}, \\
    &\gamma_{n} := [n \cdot \min(\overline{r},
    \overline{r}^{\sigma^{2}})]^{-1} \vertiii{ (\s{K} + \nu I)^{-2} \s{K}^{2\beta+1} }_{1} + n^{-1}.
\end{align*}
We remark that if $\VERT (\s{K} + \nu I)^{-2} \s{K}^{2\zeta+1} \VERT_{1}=o(n \cdot \min(\overline{r}, \overline{r}^{\sigma^{2}}))$ and $\VERT (\s{K} + \nu I)^{-2} \s{K}^{2\beta+1} \VERT_{1}=o(n \cdot \min(\overline{r}, \overline{r}^{\sigma^{2}}))$, then $a_{n} = o(1)$ and $a_{n}b_{n}=o(\gamma_{n})$ since
\begin{align*}
    \frac{a_{n}b_{n}}{\gamma_{n}} &= \frac{a_{n} (n \cdot \min(\overline{r}, \overline{r}^{\sigma^{2}}))^{-1} \vertiii{ (\s{K} + \nu I)^{-2} \s{K}^{2\zeta+1} }_{1} + n^{-1} a_{n}}{a_{n} + n^{-1}} \\
    &\le \max \left([n \cdot \min(\overline{r}, \overline{r}^{\sigma^{2}})]^{-1} \VERT (\s{K} + \nu I)^{-2} \s{K}^{2\zeta+1} \VERT_{1}, a_{n} \right).
\end{align*}

For any fixed $\varepsilon > 0$ and any $\beta \in [0, 1/2]$ satisfying $\VERT (\s{K} + \nu_{n} I)^{-2} \s{K}^{2\beta+1} \VERT_{1}= o(n\overline{r})$, we claim that
\begin{equation*}
    \lim_{n \rightarrow \infty} \sup_{\b{P}_{\c{Y}} \in \c{S}_{1}(\kappa_{1}, M_{1}, M_{2})} \b{P} \left[ \| \s{K}^{\beta} (\hat{\m{\mu}}_{\nu_{n}}-\tilde{\m{\mu}}_{\nu_{n}}) \|_{\b{H}}^{2} > \varepsilon \gamma_{n} \right] = 0.
\end{equation*} 
Note that
\begin{align*}
    \hat{\m{\mu}}_{\nu} - \tilde{\m{\mu}}_{\nu} 
    &= (\hat{\s{K}} + \nu I)^{-1} \c{Z}_{n} - \bar{\m{\mu}}_{\nu} + (\s{K} + \nu I)^{-1} (-\c{Z}_{n} + (\hat{\s{K}} + \nu I) \bar{\m{\mu}}_{\nu}) \\
    &= [(\hat{\s{K}} + \nu I)^{-1} - (\s{K} + \nu I)^{-1}] \c{Z}_{n} - (\s{K} + \nu I)^{-1} (\s{K} - \hat{\s{K}}) \bar{\m{\mu}}_{\nu} \\
    &= (\s{K} + \nu I)^{-1} (\s{K} - \hat{\s{K}}) [(\hat{\s{K}} + \nu I)^{-1} \c{Z}_{n} - \bar{\m{\mu}}_{\nu}] \\
    &= (\s{K} + \nu I)^{-1} (\s{K} - \hat{\s{K}})  (\hat{\m{\mu}}_{\nu} - \bar{\m{\mu}}_{\nu})
\end{align*}
If we decompose $\hat{\m{\mu}}_{\nu} - \bar{\m{\mu}}_{\nu} = \sum_{l, m} f_{lm} e_{lm}$, where $f_{lm} := \langle \hat{\m{\mu}}_{\nu} - \bar{\m{\mu}}_{\nu}, e_{lm} \rangle_{\b{H}}$ are real-valued random variables, then
\begin{align*}
    \langle \hat{\m{\mu}}_{\nu} - \tilde{\m{\mu}}_{\nu} , e_{lm} \rangle_{\b{H}} 
    &= \langle (\s{K} + \nu I)^{-1} (\s{K} - \hat{\s{K}})  (\hat{\m{\mu}}_{\nu} - \bar{\m{\mu}}_{\nu}), e_{lm} \rangle_{\b{H}} \\
    &= \frac{1}{\tau_{l}^{2} + \nu} \langle (\s{K} - \hat{\s{K}})  \sum_{l', m'} f_{l'm'} e_{l'm'}, e_{lm} \rangle_{\b{H}} \\
    &= \frac{1}{\tau_{l}^{2} + \nu} \sum_{l', m'} f_{l'm'} V_{lm, l'm'},
\end{align*}
where $V_{lm, l'm'} := \langle (\s{K} - \hat{\s{K}}) e_{l'm'}, e_{lm} \rangle_{\b{H}} = \tau_{l}^{2} \delta_{ll'} \delta_{mm'} - \langle \hat{\s{K}} e_{l'm'}, e_{lm} \rangle_{\b{H}}$ are real-valued random variables with zero expectation as $\b{E}[ \hat{\s{K}} ] =\s{K}$. Hence, for any $\beta \ge 0$, by the Cauchy-Schwarz inequality, 
\begin{align*}
    \| \s{K}^{\beta} (\hat{\m{\mu}}_{\nu} - \tilde{\m{\mu}}_{\nu} ) \|_{\b{H}}^{2} 
    &= \sum_{l, m} \langle \hat{\m{\mu}}_{\nu} - \tilde{\m{\mu}}_{\nu} , \s{K}^{\beta} e_{lm} \rangle_{\b{H}}^{2} = \sum_{l, m} \tau_{l}^{4 \beta} \langle \hat{\m{\mu}}_{\nu} - \tilde{\m{\mu}}_{\nu} , e_{lm} \rangle_{\b{H}}^{2} \\
    &= \sum_{l, m} \left( \frac{\tau_{l}^{2 \beta}}{\tau_{l}^{2} + \nu} \right)^{2} \left( \sum_{l', m'} f_{l'm'} V_{lm, l'm'} \right)^{2} \\
    &\le \left( \sum_{l', m'} \tau_{l'}^{4 \zeta} f_{l'm'}^{2} \right) \sum_{l, m} \left( \frac{\tau_{l}^{2 \beta}}{\tau_{l}^{2} + \nu} \right)^{2}  \left( \sum_{l', m'} \tau_{l'}^{-4 \zeta} V_{lm, l'm'}^{2} \right) \\
    &= \| \s{K}^{\zeta} (\hat{\m{\mu}}_{\nu} - \bar{\m{\mu}}_{\nu}) \|_{\b{H}}^{2} \underbrace{\sum_{l, m} \left( \frac{\tau_{l}^{2 \beta}}{\tau_{l}^{2} + \nu} \right)^{2} \sum_{l', m'} \tau_{l'}^{-4 \zeta} V_{lm, l'm'}^{2}}_{=: \mathcal{X}_{\zeta, \beta}}.
\end{align*}

We shall prove $\b{E}[\mathcal{X}_{\zeta, \beta}] = O(a_{n})$. Assuming (A), for fixed $l, l' \in \b{N}$, note that
\begin{align}\label{eq:A4:sum}
    \sum_{m=1}^{N(l)} \sum_{m'=1}^{N(l')} &\b{E}[ \langle \varphi_{ijk} ,  e_{lm} \rangle_{\b{H}}^{2} \langle \varphi_{ijk} ,  e_{l'm'} \rangle_{\b{H}}^{2} ] 
    = \sum_{m=1}^{N(l)} \sum_{m'=1}^{N(l')} \b{E}[ \s{P}  e_{lm} (R, \m{x})^{2} \s{P}  e_{l'm'} (R, \m{x})^{2} ] \nonumber \\
    &\le \sum_{m=1}^{N(l)} \sup_{R, \m{x}} \s{P}  e_{lm} (R, \m{x})^{2} \sum_{m'=1}^{N(l')} \b{E}[ \s{P}  e_{l'm'} (R, \m{x})^{2} ] \nonumber \\
    &\le \tilde{B} \tau_{l}^{2} N(l) \sum_{m'=1}^{N(l')} \b{E}[ \s{P}  e_{l'm'} (R, \m{x})^{2} ]
    = \frac{\tilde{B}}{\omega_{d_{r}, d}}  \tau_{l}^{2} N(l) \tau_{l'}^{2} N(l') ,
\end{align}
and
\begin{align*}
    \b{E} [ V_{lm, l'm'}^{2} ] &= \Var V_{lm, l'm'} = \Var [ \langle \hat{\s{K}} e_{l'm'}, e_{lm} \rangle_{\b{H}} ] \\
    &= \left( \frac{\omega_{d_{r}, d}}{n} \right)^{2} \sum_{i=1}^{n} \frac{1}{r_{i}^{2}} \sum_{j=1}^{r_{i}} \frac{1}{s_{ij}^{2}} \Var [ \sum_{k=1}^{s_{ij}} \langle \varphi_{ijk} ,  e_{lm} \rangle_{\b{H}} \langle \varphi_{ijk} ,  e_{l'm'} \rangle_{\b{H}} ] \\
    &\le \left( \frac{\omega_{d_{r}, d}}{n} \right)^{2} \sum_{i=1}^{n} \frac{1}{r_{i}^{2}} \sum_{j=1}^{r_{i}} \frac{1}{s_{ij}} \sum_{k=1}^{s_{ij}} \b{E}[ \langle \varphi_{ijk} ,  e_{lm} \rangle_{\b{H}}^{2} \langle \varphi_{ijk} ,  e_{l'm'} \rangle_{\b{H}}^{2} ].
\end{align*}
Hence, we obtain
\begin{align*}
    \sup_{\b{P}_{\c{Y}} \in \c{S}_{1}(\kappa_{1}, M_{1}, M_{2})} \b{E} \mathcal{X}_{\zeta, \beta} 
    &\le \frac{\tilde{B}}{\omega_{d_{r}, d}} \left( \frac{\omega_{d_{r}, d}}{n} \right)^{2} \sum_{i=1}^{n} \frac{1}{r_{i}^{2}} \sum_{j=1}^{r_{i}} \left( \sum_{l} \left( \frac{\tau_{l}^{2 \beta+1}}{\tau_{l}^{2} + \nu} \right)^{2} N(l) \sum_{l'} \tau_{l'}^{2-4 \zeta} N(l') \right) \\
    &\le \frac{\tilde{B} \omega_{d_{r}, d}}{n\overline{r}} \sum_{l} \left( \frac{\tau_{l}^{2 \beta + 1}}{\tau_{l}^{2} + \nu} \right)^{2} N(l) \sum_{l'} \tau_{l'}^{2-4\zeta} N(l') \\
    &= \frac{\tilde{B} \omega_{d_{r}, d}}{n\overline{r}} \vertiii{(\s{K} + \nu I)^{-2} \s{K}^{2 \beta+1}}_{1} \vertiii{\s{K}^{1-2\zeta}}_{1}
    = O(a_{n}),
\end{align*}
which leads to
\begin{align}\label{eq:step3:decom}
    &\sup_{\b{P}_{\c{Y}} \in \c{S}_{1}(\kappa_{1}, M_{1}, M_{2})} \b{P} \left[ \| \s{K}^{\beta} (\hat{\m{\mu}}_{\nu}-\tilde{\m{\mu}}_{\nu}) \|_{\b{H}}^{2} > \varepsilon \gamma_{n} \right] 
    \le \sup_{\b{P}_{\c{Y}} \in \c{S}_{1}(\kappa_{1}, M_{1}, M_{2})} \b{P} \left[ \mathcal{X}_{\zeta, \beta} \| \s{K}^{\zeta} (\hat{\m{\mu}}_{\nu} - \bar{\m{\mu}}_{\nu}) \|_{\b{H}}^{2} > \varepsilon \gamma_{n} \right] \nonumber \\
    &\le \sup_{\b{P}_{\c{Y}} \in \c{S}_{1}(\kappa_{1}, M_{1}, M_{2})} \b{P} \left[ \mathcal{X}_{\zeta, \beta} \| \s{K}^{\zeta} (\hat{\m{\mu}}_{\nu}-\bar{\m{\mu}}_{\nu}) \|_{\b{H}}^{2} > \varepsilon \gamma_{n}, \mathcal{X}_{\zeta, \beta} < 1/4 \right] + \b{P} \left[ \mathcal{X}_{\zeta, \beta} \ge 1/4 \right] \nonumber \\
    &\le \sup_{\b{P}_{\c{Y}} \in \c{S}_{1}(\kappa_{1}, M_{1}, M_{2})} \b{P} \left[ \mathcal{X}_{\zeta, \beta} \| \s{K}^{\zeta} (\hat{\m{\mu}}_{\nu}-\bar{\m{\mu}}_{\nu}) \|_{\b{H}}^{2} > \varepsilon \gamma_{n}, \mathcal{X}_{\zeta, \beta} < 1/4 \right] + 4 a_{n},
\end{align}
where we use the Markov inequality for the last line. Meanwhile, on the event $(\mathcal{X}_{\zeta, \beta} < 1/4)$, we have
\begin{align*}
    \| \s{K}^{\zeta} (\tilde{\m{\mu}}_{\nu}-\bar{\m{\mu}}_{\nu}) \|_{\b{H}}
    &\ge \| \s{K}^{\zeta} (\hat{\m{\mu}}_{\nu}-\bar{\m{\mu}}_{\nu}) \|_{\b{H}} - \| \s{K}^{\zeta} (\hat{\m{\mu}}_{\nu}-\tilde{\m{\mu}}_{\nu}) \|_{\b{H}} \\
    &\ge (1 - \sqrt{\mathcal{X}_{\zeta, \beta}}) \| \s{K}^{-\theta} (\hat{\m{\mu}}_{\nu}-\bar{\m{\mu}}_{\nu}) \|_{\b{H}} = \| \s{K}^{\zeta} (\hat{\m{\mu}}_{\nu}-\bar{\m{\mu}}_{\nu}) \|_{\b{H}}/2,
\end{align*}
and we know that $\b{E} \| \s{K}^{\zeta} (\tilde{\m{\mu}}_{\nu}-\bar{\m{\mu}}_{\nu}) \|_{\b{H}}^{2} = O(b_{n})$. Therefore, \eqref{eq:step3:decom} becomes
\begin{align*}
    &\sup_{\b{P}_{\c{Y}} \in \c{S}_{1}(\kappa_{1}, M_{1}, M_{2})} \b{P} \left[ \| \s{K}^{\beta} (\hat{\m{\mu}}_{\nu}-\tilde{\m{\mu}}_{\nu}) \|_{\b{H}}^{2} > \varepsilon \gamma_{n} \right] \\
    &\le \sup_{\b{P}_{\c{Y}} \in \c{S}_{1}(\kappa_{1}, M_{1}, M_{2})} \b{P} \left[ \frac{\mathcal{X}_{\zeta, \beta}}{a_{n}} \frac{\| \s{K}^{\zeta} (\tilde{\m{\mu}}_{\nu}-\bar{\m{\mu}}_{\nu}) \|_{\b{H}}^{2}}{b_{n}} > \frac{\varepsilon \gamma_{n}}{4 a_{n} b_{n}} \right] + 4 a_{n} \\
    &\le \sup_{\b{P}_{\c{Y}} \in \c{S}_{1}(\kappa_{1}, M_{1}, M_{2})} \b{P} \left[ \frac{\mathcal{X}_{\zeta, \beta}}{a_{n}}  > \sqrt{\frac{\varepsilon \gamma_{n}}{4 a_{n} b_{n}}} \right] + \b{P} \left[ \frac{\| \s{K}^{\zeta} (\tilde{\m{\mu}}_{\nu}-\bar{\m{\mu}}_{\nu}) \|_{\b{H}}^{2}}{b_{n}}  > \sqrt{\frac{\varepsilon \gamma_{n}}{4 a_{n} b_{n}}} \right] + 4 a_{n} \\
    &\le \sqrt{\frac{4 a_{n} b_{n}}{\varepsilon \gamma_{n}}} \left( \sup_{\b{P}_{\c{Y}} \in \c{S}_{1}(\kappa_{1}, M_{1}, M_{2})} \b{E} \left[ \frac{\mathcal{X}_{\zeta, \beta}}{a_{n}} \right] + \b{E} \left[ \frac{\| \s{K}^{\zeta} (\tilde{\m{\mu}}_{\nu}-\bar{\m{\mu}}_{\nu}) \|_{\b{H}}^{2}}{b_{n}} \right] \right) + 4 a_{n} \rightarrow 0,     
\end{align*}
as $n \rightarrow \infty$ since $a_{n} = o(1)$ and $a_{n}b_{n}=o(\gamma_{n})$.

\item [(Step 4)] Let $c_{n} := \nu_{n}^{2[(\beta + \kappa_{1}) \wedge 1]}$. Choose any $D > 3 \tilde{M}_{1}$, where the constant $\tilde{M}_{1} >0$ is given in (Step 1).
\begin{small}
\begin{align*}
    &\sup_{\b{P}_{\c{Y}} \in \c{S}_{1}(\kappa_{1}, M_{1}, M_{2})} \b{P} \left[ \| \s{K}^{\beta} (\hat{\m{\mu}}_{\nu}-\m{\mu}) \|_{\b{H}}^{2} > D ( c_{n} + \gamma_{n}) \right] \\
    &\le \sup_{\b{P}_{\c{Y}} \in \c{S}_{1}(\kappa_{1}, M_{1}, M_{2})} \b{P} \left[ \| \s{K}^{\beta} (\bar{\m{\mu}}_{\nu}-\m{\mu}) \|_{\b{H}}^{2}  + \| \s{K}^{\beta} (\tilde{\m{\mu}}_{\nu}-\bar{\m{\mu}}_{\nu}) \|_{\b{H}}^{2} + \| \s{K}^{\beta} (\hat{\m{\mu}}_{\nu}-\tilde{\m{\mu}}_{\nu}) \|_{\b{H}}^{2} > \frac{D}{3} ( c_{n} + \gamma_{n}) \right] \\
    &\le \sup_{\b{P}_{\c{Y}} \in \c{S}_{1}(\kappa_{1}, M_{1}, M_{2})} \b{P} \left[ \| \s{K}^{\beta} (\tilde{\m{\mu}}_{\nu}-\bar{\m{\mu}}_{\nu}) \|_{\b{H}}^{2} + \| \s{K}^{\beta} (\hat{\m{\mu}}_{\nu}-\tilde{\m{\mu}}_{\nu}) \|_{\b{H}}^{2} > \frac{D}{3} \gamma_{n} \right] \\
    &\le \sup_{\b{P}_{\c{Y}} \in \c{S}_{1}(\kappa_{1}, M_{1}, M_{2})} \b{P} \left[ \| \s{K}^{\beta} (\tilde{\m{\mu}}_{\nu}-\bar{\m{\mu}}_{\nu}) \|_{\b{H}}^{2} > \frac{D}{6} \gamma_{n} \right] + \b{P} \left[ \| \s{K}^{\beta} (\hat{\m{\mu}}_{\nu}-\tilde{\m{\mu}}_{\nu}) \|_{\b{H}}^{2} > \frac{D}{6} \gamma_{n} \right],
\end{align*}   
\end{small}
hence by (Step 3), for any $D>0$, it holds that
\begin{align*}
    &\limsup_{n \rightarrow \infty} \sup_{\b{P}_{\c{Y}} \in \c{S}_{1}(\kappa_{1}, M_{1}, M_{2})} \b{P} \left[ \| \s{K}^{\beta} (\hat{\m{\mu}}_{\nu}-\m{\mu}) \|_{\b{H}}^{2} > D ( c_{n} + \gamma_{n}) \right] \\ 
    &\le \frac{6}{D} \sup_{n} \sup_{\b{P}_{\c{Y}} \in \c{S}_{1}(\kappa_{1}, M_{1}, M_{2})} \frac{\b{E}[ \| \s{K}^{\beta} (\tilde{\m{\mu}}_{\nu}-\bar{\m{\mu}}_{\nu}) \|_{\b{H}}^{2} ]}{\gamma_{n}}.
\end{align*}
Finally, from (Step 2), we conclude using Markov's inequality that
\begin{align*}
    \lim_{D \rightarrow \infty} \limsup_{n \rightarrow \infty} \sup_{\b{P}_{\c{Y}} \in \c{S}_{1}(\kappa_{1}, M_{1}, M_{2})} \b{P} \left[ \| \s{K}^{\beta} (\hat{\m{\mu}}_{\nu}-\m{\mu}) \|_{\b{H}}^{2} > D ( c_{n} + \gamma_{n}) \right] = 0.
\end{align*}
\end{enumerate}
\end{proof}

In \cref{thm:gen:asym:mean}, there is a trade-off of the convergence rate concerning the decaying speed of the penalty parameter $\nu_{n} > 0$. When $\nu_{n}$ decreases faster, $c_{n}$ decreases faster but $a_{n}$ decreases more slowly. Therefore, in order to achieve the optimal convergence rate of our regularized mean estimator, one may choose $\nu_{n} > 0$ satisfying $c_{n}=\nu_{n}^{2(\beta + \kappa_{1})} \asymp [n \cdot \min(\overline{r}, \overline{r}^{\sigma^{2}})]^{-1} \VERT (\s{K} + \nu_{n} I)^{-2} \s{K}^{2\beta+1} \VERT_{1}=a_{n}$. Specifically, when $\tau_{l} \asymp l^{-p}$ and $N(l) \asymp l^{q}$ for some $p, q >0$ with $2p > q+1$, it holds for any $\beta \ge 0$ that
\begin{align*}
    \vertiii{(\s{K} + \nu I)^{-2} \s{K}^{2 \beta + 1}}_{1} &= \sum_{l=1}^{\infty} N(l) \frac{\tau_{l}^{4 \beta + 2}}{(\tau_{l}^{2} + \nu)^{2}} = \sum_{l=1}^{\infty}  \frac{N(l) \tau_{l}^{-2(1 - 2\beta)}}{(1 + \nu \tau_{l}^{-2})^{2}} \\
    &\asymp \sum_{l=1}^{\infty}  \frac{l^{2p(1 - 2\beta)+q}}{(1 + \nu l^{2p})^{2}} \asymp \int_{1}^{\infty} \frac{x^{2p(1 - 2\beta)+q}}{(1+ \nu x^{2p})^{2}} \rd x.
\end{align*}
Applying the change of variables with $s=(\nu x^{2p})^{(2p(1 - 2\beta)+q+1)/2p}$, we get
\begin{align}\label{eq:tr:norm:rate}
    \vertiii{(\s{K} + \nu I)^{-2} \s{K}^{2 \beta + 1}}_{1} 
    &\asymp \nu^{-(q+2p(1 - 2\beta)+1)/2p} \int_{\nu^{(q+2p(1 - 2\beta)+1)/2p}}^{\infty} (1+ s^{2p/(q+2p(1 - 2\beta)+1)})^{-2} \rd s \nonumber \\ 
    &\asymp \nu^{-(2p(1 - 2\beta)+q+1)/2p} = \nu^{-(\frac{2p+q+1}{2p}-2\beta)}.
\end{align}
Therefore, 
\begin{align}\label{eq:opt:pen:rate}
    &\nu_{n}^{2(\beta + \kappa_{1})} \asymp [n \cdot \min(\overline{r},
    \overline{r}^{\sigma^{2}})]^{-1} \VERT (\s{K} + \nu_{n} I)^{-2} \s{K}^{2\beta+1} \VERT_{1} \\ &\Longleftrightarrow \quad \nu_{n} \asymp [n \cdot \min(\overline{r},
    \overline{r}^{\sigma^{2}})]^{-(\frac{2p+q+1}{2p} + 2\kappa_{1})^{-1}}, \nonumber   
\end{align}
regardless of the value of $\beta \ge 0$.

\begin{proof}[Proof of \cref{thm:mean:asym:L2}]
Take $\beta=1/2$ and $\kappa_{1}=0$. Also, choose $\nu_{n} \asymp [n \cdot \min(\overline{r}, \overline{r}^{\sigma^{2}})]^{-(\frac{2p+q+1}{2p})^{-1}}$ in view of \eqref{eq:opt:pen:rate}. It remains to check whether the assumptions other than (A) in \cref{thm:gen:asym:mean} hold. First, we verify
\begin{equation*}
    \VERT (\s{K} + \nu_{n} I)^{-2} \s{K}^{2\beta+1} \VERT_{1} \stackrel{\eqref{eq:tr:norm:rate}}{\asymp} \nu_{n}^{-(\frac{2p+q+1}{2p}-1)} \asymp [n \cdot \min(\overline{r}, \overline{r}^{\sigma^{2}})]^{(\frac{2p+q+1}{2p})^{-1}(\frac{2p+q+1}{2p}-1)} = o(n \cdot \min(\overline{r},
    \overline{r}^{\sigma^{2}})).
\end{equation*}
Second, choose $\theta \in (\frac{q+1}{2p},1) \neq \emptyset$ and set $\zeta := (1-\theta)/2 \in (0, 1/2]$. Then, $-2\theta p+q < -1$ so
\begin{equation*}
    \VERT \s{K}^{1-2\zeta} \VERT_{1} \asymp \sum_{l=1}^{\infty} l^{-2\theta p+q} < \infty,
\end{equation*}
and additionally
\begin{equation*}
    \VERT (\s{K} + \nu_{n} I)^{-2} \s{K}^{2\zeta+1} \VERT_{1} \stackrel{\eqref{eq:tr:norm:rate}}{\asymp} \nu_{n}^{-(\frac{2p+q+1}{2p}-2\zeta)} \asymp [n \cdot \min(\overline{r}, \overline{r}^{\sigma^{2}})]^{(\frac{q+1}{2p}+1)^{-1}(\frac{q+1}{2p}+\theta)} = o(n \cdot \min(\overline{r}, \overline{r}^{\sigma^{2}})).
\end{equation*}
\end{proof}

\begin{proof}[Proof of \cref{thm:mean:asym:RKHS}]
Take $\beta=0$, $\kappa_{1} > 0$, and choose $\nu_{n} \asymp [n \cdot \min(\overline{r}, \overline{r}^{\sigma^{2}})]^{-(\frac{2p+q+1}{2p} + 2\kappa_{1})^{-1}}$ in view of \eqref{eq:opt:pen:rate}. It remains to check whether the assumptions other than (A) in \cref{thm:gen:asym:mean} hold. First, we verify
\begin{equation*}
    \VERT (\s{K} + \nu_{n} I)^{-2} \s{K}^{2\beta+1} \VERT_{1} \stackrel{\eqref{eq:tr:norm:rate}}{\asymp} \nu_{n}^{-\frac{2p+q+1}{2p}} \asymp [n \cdot \min(\overline{r}, \overline{r}^{\sigma^{2}})]^{(\frac{2p+q+1}{2p} + 2\kappa_{1})^{-1}\frac{2p+q+1}{2p}} = o(n \cdot \min(\overline{r}, \overline{r}^{\sigma^{2}})).
\end{equation*}
Second, choose $\theta \in (\frac{q+1}{2p},1) \neq \emptyset$ and set $\zeta := (1-\theta)/2 \in (0, 1/2]$. Then, $-2\theta p+q < -1$ so $\VERT \s{K}^{1-2\zeta} \VERT_{1} < \infty$, and
\begin{equation*}
    \VERT (\s{K} + \nu_{n} I)^{-2} \s{K}^{2\zeta+1} \VERT_{1} \stackrel{\eqref{eq:tr:norm:rate}}{\asymp} \nu_{n}^{-(\frac{2p+q+1}{2p}-2\zeta)} \asymp [n \cdot \min(\overline{r}, \overline{r}^{\sigma^{2}})]^{(\frac{q+1}{2p}+ 1+ 2\kappa_{1})^{-1}(\frac{q+1}{2p}+\theta)} = o(n \cdot \min(\overline{r}, \overline{r}^{\sigma^{2}})).
\end{equation*}
\end{proof}

\subsection{Proofs in Section \ref{sec:asym:cov}}\label{sec:proof:asym:cov} 

\begin{proof}[Proof of \cref{prop:srce:cond}]
First, by Cauchy-Schwarz inequality, we have $\b{E} [\|\c{Y}\|_{\b{H}}^{2}] \le (\b{E} [\|\c{Y}\|_{\b{H}}^{4}])^{1/2} \le \sqrt{M_{4}}$. Let $\m{\Gamma} = (\s{K}^{\otimes})^{\kappa_{2}} \delta^{\otimes}$ with $\|\delta^{\otimes}\|_{\b{H} \otimes \b{H}}^{2} \le M_{3}$. Then by H\"older's inequality,
\begin{equation*}
    \VERT \m{\Gamma} \VERT_{1} \le \VERT (\s{K}^{\otimes})^{\kappa_{2}} \VERT_{2} \VERT \delta^{\otimes} \VERT_{2} = \VERT \s{K}^{2\kappa_{2}}\VERT_{1}^{2} \|\delta^{\otimes}\|_{\b{H} \otimes \b{H}}^{2}
\end{equation*}
Then $\sqrt{M_{3}} \s{K}^{2 \kappa_{2}} \succeq \m{\Gamma}$ since for any $f \in \b{H}$,
\begin{align*}
    \langle f, \m{\Gamma} f \rangle_{\b{H}} = \langle f, (\s{K}^{\otimes})^{\kappa_{2}} \delta^{\otimes} f \rangle_{\b{H}} &= \langle f \otimes f, (\s{K}^{\otimes})^{\kappa_{2}} \delta^{\otimes} \rangle_{\b{H} \otimes \b{H}} = \langle (\s{K}^{\otimes})^{\kappa_{2}} (f \otimes f),  \delta^{\otimes} \rangle_{\b{H} \otimes \b{H}} \\
    &= \langle \s{K}^{\kappa_{2}} f \otimes \s{K}^{\kappa_{2}} f,  \delta^{\otimes} \rangle_{\b{H} \otimes \b{H}} \le \sqrt{M_{3}} \| \s{K}^{\kappa_{2}} f \otimes \s{K}^{\kappa_{2}} f \|_{\b{H} \otimes \b{H}} \\
    &= \sqrt{M_{3}} \| \s{K}^{\kappa_{2}} f \|_{\b{H}}^{2}
    = \sqrt{M_{3}} \langle f, \s{K}^{2\kappa_{2}} f \rangle_{\b{H}}.
\end{align*}
Due to the Douglas factorization theorem \cite{douglas1966majorization}, we obtain $\sqrt{M_{3}} \s{K}^{2 \kappa_{2}} \succeq \m{\Gamma} \succeq \m{\mu} \otimes \m{\mu}$,
and \eqref{eq:Doug:equiv} yields the result.
\end{proof}

\begin{theorem}\label{thm:gen:asym:cov}
Consider a sequence of the penalty parameter $\eta_{n} > 0$ converging to $0$. Assume (A) and that there is some $\zeta \in (0, 1/2]$ with $\VERT \s{K}^{1-2\zeta} \VERT_{1} < \infty$ and $\VERT (\s{K}^{\otimes} + \eta_{n} I)^{-2} (\s{K}^{\otimes})^{2\zeta+1} \VERT_{1}=o(n \cdot \min(\overline{r}, \overline{r}^{\sigma^{4}}))$. Then, for any $\beta \in [0, 1/2]$ satisfying $\VERT (\s{K}^{\otimes} + \eta_{n} I)^{-2} (\s{K}^{\otimes})^{2\beta+1} \VERT_{1}=o(n \cdot \min(\overline{r}, \overline{r}^{\sigma^{4}}))$, it holds that
\begin{align*}
    \lim_{D \rightarrow \infty} \limsup_{n \rightarrow \infty} \sup_{\b{P}_{\c{Y}} \in \c{S}_{2}(\kappa_{2}, M_{3}, M_{4})} \b{P} \left[ \| (\s{K}^{\otimes})^{\beta} (\hat{\m{\Gamma}}_{\eta_{n}}-\m{\Gamma}) \|_{\b{H} \otimes \b{H}}^{2} > D \max(c_{n}^{\otimes}, a_{n}^{\otimes}, n^{-1}) \right] = 0,
\end{align*}
where
\begin{align*}
    c^{\otimes}_{n} := \eta_{n}^{2[(\beta + \kappa_{2}) \wedge 1]}, \quad
    a^{\otimes}_{n} := [n \cdot \min(\overline{r},
    \overline{r}^{\sigma^{4}})]^{-1} \vertiii{ (\s{K}^{\otimes} + \eta_{n} I)^{-2} (\s{K}^{\otimes})^{2\beta+1} }_{1}.
\end{align*}
\end{theorem}
\begin{proof}
Similar to \cref{thm:gen:asym:mean}, consider
\begin{align*}
    &\bar{\m{\Gamma}}_{\eta} = (\s{K}^{\otimes} + \eta I)^{-1} \s{K}^{\otimes} \m{\Gamma}, \\
    &\tilde{\m{\Gamma}}_{\eta} = \bar{\m{\Gamma}}_{\eta} - (\s{K}^{\otimes} + \eta I)^{-1} (- \c{Z}^{\otimes}_{n} + (\hat{\s{K}}^{\otimes} + \eta I) \bar{\m{\Gamma}}_{\eta}), \\
    &\hat{\m{\Gamma}}_{\eta} = (\hat{\s{K}}^{\otimes} + \eta I)^{-1} \c{Z}^{\otimes}_{n},    
\end{align*}
and the decomposition $(\hat{\m{\Gamma}}_{\eta}-\m{\Gamma}) = (\bar{\m{\Gamma}}_{\eta}-\m{\Gamma}) + (\tilde{\m{\Gamma}}_{\eta}-\bar{\m{\Gamma}}_{\eta}) + (\hat{\m{\Gamma}}_{\eta}-\tilde{\m{\Gamma}}_{\eta})$. We list some equations that we will use several times in the proof below. Here, $G \in \b{H} \otimes \b{H}$ is any Hilbert-Schmidt operator:
\begin{align}
    &\b{E} [\left. \c{Z}^{\otimes}_{n} \right\vert \varphi] = \b{E} \left[\left. \frac{\omega_{d_{r}, d}^{2}}{n} \sum_{i=1}^{n} \frac{1}{r_{i}(r_{i}-1)} \sum_{1 \le j \neq j' \le r_{i}} \frac{1}{s_{ij}s_{ij'}} \sum_{\substack{1 \le k \le s_{ij} \\ 1 \le k' \le s_{ij'}}} Z_{iJJ'} \varphi_{iJJ'}^{\otimes} \right\vert \varphi \right] = \hat{\s{K}}^{\otimes} \m{\Gamma} \label{eq:Wn:cond:exp:cov}, \\
    &\Var \left[ \langle (\varphi_{iJJ'}^{\otimes} \otimes_{2} \varphi_{iJJ'}^{\otimes}) G, e^{\otimes}_{lm, l'm'} \rangle_{\b{H} \otimes \b{H}} \right] 
    \le \b{E} \left[ \langle \varphi_{iJJ'}^{\otimes}, G \rangle_{\b{H} \otimes \b{H}}^{2} \langle \varphi_{iJJ'}^{\otimes}, e^{\otimes}_{lm, l'm'} \rangle_{\b{H} \otimes \b{H}}^{2}  \right] \nonumber \\
    &\hspace*{6em} \quad \le B^{4} \|G\|_{\b{H} \otimes \b{H}}^{2} \b{E} \left[ \langle \varphi_{iJJ'}^{\otimes}, e^{\otimes}_{lm, l'm'} \rangle_{\b{H} \otimes \b{H}}^{2}  \right] \nonumber \\
    &\hspace*{6em} \quad = \frac{B^{4}}{\omega_{d_{r}, d}^{2}} \|G\|_{\b{H} \otimes \b{H}}^{2} \langle \s{K}^{\otimes} e^{\otimes}_{lm, l'm'}, e^{\otimes}_{lm, l'm'} \rangle_{\b{H} \otimes \b{H}} = \frac{\tau_{ll'}^{2} B^{4}}{\omega_{d_{r}, d}^{2}} \|G\|_{\b{H} \otimes \b{H}}^{2} , \quad j \neq j', \label{eq:var:lk:cov}\\
    &\vertiii{(\s{K}^{\otimes} + \eta I)^{-2} (\s{K}^{\otimes})^{\alpha}}_{1} = \sum_{l, l'=1}^{\infty} \sum_{m, m'=1}^{N(l)} \langle (\s{K}^{\otimes} + \eta I)^{-2} (\s{K}^{\otimes})^{\alpha} e^{\otimes}_{lm, l'm'}, e^{\otimes}_{lm, l'm'} \rangle_{\b{H} \otimes \b{H}} \nonumber \\ 
    &\hspace*{6em} \quad = \sum_{l, l'=1}^{\infty} \left( \frac{\tau_{ll'}^{\alpha}}{\tau_{ll'}^{2} + \eta} \right)^{2} N(l) N(l')
    \label{eq:tr:norm:cov},
\end{align}
where the constant $B > 0$ is defined as in the assumption (A).

\begin{enumerate}
\item [(Step 1)] For any $\beta \ge 0, \eta > 0$, we claim that
\begin{equation*}
    \sup_{\b{P}_{\c{Y}} \in \c{S}_{2}(\kappa_{2}, M_{3}, M_{4})} \| (\s{K}^{\otimes})^{\beta} (\m{\Gamma} - \bar{\m{\Gamma}}_{\eta}) \|_{\b{H} \otimes \b{H}}^{2} = O(\eta^{2[(\beta + \kappa_{2}) \wedge 1]}).
\end{equation*} 
Similar to the (Step 1) in the proof of \cref{thm:gen:asym:mean}, we get
\begin{align}\label{eq:step1:cov}
    \| (\s{K}^{\otimes})^{\beta} (\m{\Gamma} - \bar{\m{\Gamma}}_{\eta}) \|_{\b{H} \otimes \b{H}}^{2} \le \eta^{2[(\beta + \kappa_{2}) \wedge 1]} M_{5} = O(\eta^{2[(\beta + \kappa_{2}) \wedge 1]}),
\end{align}
where $M_{5} = M_{3} \VERT \s{K}^{\otimes} \VERT_{\infty}^{2(\beta+\kappa_{2}-1) \vee 0} = M_{3} \VERT \s{K} \VERT_{\infty}^{4(\beta+\kappa_{2}-1) \vee 0}$.

\item [(Step 2)] For any $\beta \ge 0, \eta > 0$, we claim that
\begin{equation*}
    \sup_{\b{P}_{\c{Y}} \in \c{S}_{2}(\kappa_{2}, M_{3}, M_{4})} \b{E} \left[ \| (\s{K}^{\otimes})^{\beta} (\tilde{\m{\Gamma}}_{\eta} - \bar{\m{\Gamma}}_{\eta}) \|_{\b{H} \otimes \b{H}}^{2} \right] = O ((n\overline{r})^{-1} \vertiii{ (\s{K}^{\otimes} + \eta I)^{-2} (\s{K}^{\otimes})^{2\beta+1} }_{1} + n^{-1}).
\end{equation*} 
First, we remark that whenever $\beta \ge 0$ and $\eta > 0$, $(\s{K}^{\otimes} + \eta I)^{-2} (\s{K}^{\otimes})^{2\beta+1}$ is a trace-class operator since
\begin{align*}
    \vertiii{ (\s{K}^{\otimes} + \eta I)^{-2} (\s{K}^{\otimes})^{2\beta+1} }_{1} 
    &\le \vertiii{ (\s{K}^{\otimes} + \eta I)^{-1} (\s{K}^{\otimes})^{\beta+1} }_{\infty}^{2} \vertiii{\s{K}^{\otimes}}_{1} \\
    &\le \vertiii{\s{K}^{\otimes}}_{\infty}^{\beta} \vertiii{\s{K}^{\otimes}}_{1} = \vertiii{\s{K}}_{\infty}^{2\beta} \vertiii{\s{K}}_{2}< \infty,
\end{align*}
by \cref{lem:srce:oper:norm}. Observe that
\begin{align*}
    \tilde{\m{\Gamma}}_{\eta}-\bar{\m{\Gamma}}_{\eta} &= (\s{K}^{\otimes} + \eta I)^{-1} (-\c{Z}^{\otimes}_{n} + (\hat{\s{K}}^{\otimes} + \eta I) \bar{\m{\Gamma}}_{\eta}).
\end{align*}
Thus, for any $\beta \ge 0$, similar to \eqref{eq:step2:decomp} in the proof of \cref{thm:gen:asym:mean}, we obtain
\begin{align}\label{eq:step2:decomp:cov}
    \b{E} \left[ \| (\s{K}^{\otimes})^{\beta} (\tilde{\m{\Gamma}}_{\eta} - \bar{\m{\Gamma}}_{\eta}) \|_{\b{H} \otimes \b{H}}^{2} \right]
    = \sum_{l, l' =1}^{\infty} \left( \frac{\tau_{ll'}^{2 \beta}}{\tau_{ll'}^{2} + \eta} \right)^{2} \sum_{m, m'=1}^{N(l)} \Var [ \langle \c{Z}^{\otimes}_{n} - \hat{\s{K}}^{\otimes} \bar{\m{\Gamma}}_{\eta}, e^{\otimes}_{lm, l'm'} \rangle_{\b{H} \otimes \b{H}} ].
\end{align}
To calculate \eqref{eq:step2:decomp:cov}, we first calculate the summand for fixed $l, l', m$, and $m'$, using the law of total variance:
\begin{align}\label{eq:step2:cov:total:var}
    &\Var [ \langle \c{Z}^{\otimes}_{n} - \hat{\s{K}}^{\otimes} \bar{\m{\Gamma}}_{\eta}, e^{\otimes}_{lm, l'm'} \rangle_{\b{H} \otimes \b{H}} ] \\
    &= \underbrace{\Var \left( \b{E} \left[ \left. \langle \c{Z}^{\otimes}_{n} - \hat{\s{K}}^{\otimes} \bar{\m{\Gamma}}_{\eta}, e^{\otimes}_{lm, l'm'} \rangle_{\b{H} \otimes \b{H}} \right\vert \varphi \right] \right)}_{=: (A)_{lm, l'm'}} + \underbrace{\b{E} \left( \Var \left[ \left. \langle \c{Z}^{\otimes}_{n} - \hat{\s{K}}^{\otimes} \bar{\m{\Gamma}}_{\eta}, e^{\otimes}_{lm, l'm'} \rangle_{\b{H} \otimes \b{H}} \right\vert \varphi \right] \right)}_{=: (B)_{lm, l'm'}}. \nonumber
\end{align}

First, we calculate $(A)_{lm, l'm'}$. Because $\varphi_{i_{1}J_{1}J'_{1}}^{\otimes}$ and $\varphi_{i_{2}J_{2}J'_{2}}^{\otimes}$ are i.i.d. whenever $i_{1} \neq i_{2}$, we get
\begin{align*}
    &(A)_{lm, l'm'} \stackrel{\eqref{eq:Wn:cond:exp:cov}}{=} \Var \left[ \langle \hat{\s{K}}^{\otimes} (\m{\Gamma} - \bar{\m{\Gamma}}_{\eta}) , e^{\otimes}_{lm, l'm'} \rangle_{\b{H} \otimes \b{H}} \right] \\
    &= \left( \frac{\omega_{d_{r}, d}^{2}}{n} \right)^{2} \sum_{i=1}^{n} \frac{1}{(r_{i}(r_{i}-1))^{2}} \Var \underbrace{\left[ \sum_{1 \le j \neq j' \le r_{i}}  T_{ijj'}^{lm, l'm'} \right]}_{=:T_{i}^{lm, l'm'}}.
\end{align*}
where
\begin{equation*}
    T_{ijj'}^{lm, l'm'} := \frac{1}{s_{ij}s_{ij'}} \sum_{\substack{1 \le k \le s_{ij} \\ 1 \le k' \le s_{ij'}}} \langle (\m{\Gamma} - \bar{\m{\Gamma}}_{\eta}), (\varphi_{iJJ'}^{\otimes} \otimes_{2} \varphi_{iJJ'}^{\otimes}) e^{\otimes}_{lm, l'm'} \rangle_{\b{H} \otimes \b{H}}, \quad j \neq j'.
\end{equation*}

Note that
\begin{align}
    &\b{E}[T_{ijj'}^{lm, l'm'}] = \langle (\m{\Gamma} - \bar{\m{\Gamma}}_{\eta}), \s{K}^{\otimes} e^{\otimes}_{lm, l'm'} \rangle_{\b{H} \otimes \b{H}} = \tau_{ll'}^{2} \langle (\m{\Gamma} - \bar{\m{\Gamma}}_{\eta}), e^{\otimes}_{lm, l'm'} \rangle_{\b{H} \otimes \b{H}}, \label{eq:cov:step2:tijj:exp} \\
    &\b{E}[(T_{ijj'}^{lm, l'm'})^{2}] \le \frac{1}{s_{ij}s_{ij'}} \sum_{\substack{1 \le k \le s_{ij} \\ 1 \le k' \le s_{ij'}}} \b{E} \left[ \langle  (\m{\Gamma} - \bar{\m{\Gamma}}_{\eta}), \varphi_{iJJ'}^{\otimes} \rangle_{\b{H} \otimes \b{H}}^{2} \langle  e^{\otimes}_{lm, l'm'}, \varphi_{iJJ'}^{\otimes} \rangle_{\b{H} \otimes \b{H}}^{2} \right] \nonumber \\
    &\hspace*{6em} \quad \stackrel{\eqref{eq:var:lk:cov}}{\le} \frac{\tau_{ll'}^{2} B^{4}}{\omega_{d_{r}, d}^{2}} \|\m{\Gamma} - \bar{\m{\Gamma}}_{\eta}\|_{\b{H} \otimes \b{H}}^{2}, \label{eq:cov:step2:tijj:var} 
\end{align}
hence the first moment of $T_{i}^{lm, l'm'}$ becomes
\begin{align*}
    \b{E}[T_{i}^{lm, l'm'}] = r_{i}(r_{i}-1) \tau_{ll'}^{2} \langle (\m{\Gamma} - \bar{\m{\Gamma}}_{\eta}), e^{\otimes}_{lm, l'm'} \rangle_{\b{H} \otimes \b{H}}.
\end{align*}
On the other hand, the second moment of $T_{i}^{lm, l'm'}$ is given by
\begin{align}\label{eq:cond:exp:cov:decomp}
    \b{E}[(T_{i}^{lm, l'm'})^{2}] = \sum_{\m{J} \in \c{J}^{i}} \b{E} \left[ T_{ij_{1} j'_{1}}^{lm, l'm'} \cdot T_{ij_{2} j'_{2}}^{lm, l'm'} \right],
\end{align}
where $\c{J}^{i}$ is the set of four indices $\m{J} := (j_{1}, j'_{1}, j_{2}, j'_{2})$ given by
\begin{equation*}
    \c{J}^{i} := \{ \m{J} \in \b{N}^{4} : 1 \le j_{1}, j'_{1}, j_{2}, j'_{2} \le r_{i}, j_{1} \neq j'_{1}, j_{2} \neq j'_{2} \}, \quad |\c{J}^{i}| = (r_{i}(r_{i}-1))^{2}.
\end{equation*}
Now, consider a partition of $\c{J}^{i}$ as follows:
\begin{alignat*}{2}
    &\c{J}^{i}_{2} := \{ \m{J} \in \c{J}^{i} : |\{j_{1}, j'_{1}\} \cap \{j_{2}, j'_{2}\}| = 2 \}, \quad &&|\c{J}^{i}_{2}| = 2 r_{i}(r_{i}-1), \\
    &\c{J}^{i}_{1} := \{ \m{J} \in \c{J}^{i} : |\{j_{1}, j'_{1}\} \cap \{j_{2}, j'_{2}\}| = 1 \}, \quad &&|\c{J}^{i}_{1}| = 4 r_{i}(r_{i}-1)(r_{i}-2), \\
    &\c{J}^{i}_{0} := \{ \m{J} \in \c{J}^{i} : |\{j_{1}, j'_{1}\} \cap \{j_{2}, j'_{2}\}| = 0 \}, \quad &&|\c{J}^{i}_{0}| = r_{i}(r_{i}-1)(r_{i}-2)(r_{i}-3).
\end{alignat*}

For $\m{J} \in \c{J}^{i}_{0}$, by independence, 
\begin{equation*}
    \b{E} \left[ T_{ij_{1} j'_{1}}^{lm, l'm'} \cdot T_{ij_{2} j'_{2}}^{lm, l'm'} \right] = \b{E} \left[ T_{ij_{1} j'_{1}}^{lm, l'm'} \right] \b{E} \left[ T_{ij_{2} j'_{2}}^{lm, l'm'} \right] \stackrel{\eqref{eq:cov:step2:tijj:exp}}{=} \tau_{ll'}^{4} \langle (\m{\Gamma} - \bar{\m{\Gamma}}_{\eta}), e^{\otimes}_{lm, l'm'} \rangle_{\b{H} \otimes \b{H}}^{2}.
\end{equation*}
For $\m{J} \in \c{J}^{i}_{1} \cup \c{J}^{i}_{2}$, by the Cauchy-Schwarz inequality, 
\begin{equation*}
    \b{E} \left[ T_{ij_{1} j'_{1}}^{lm, l'm'} \cdot T_{ij_{2} j'_{2}}^{lm, l'm'} \right] = \b{E}[(T_{ijj'}^{lm, l'm'})^{2}] \stackrel{\eqref{eq:cov:step2:tijj:var}}{\le} \frac{\tau_{ll'}^{2} B^{4}}{\omega_{d_{r}, d}^{2}} \|\m{\Gamma} - \bar{\m{\Gamma}}_{\eta}\|_{\b{H} \otimes \b{H}}^{2}.
\end{equation*}
Therefore,
\begin{align*}
    &\Var [T_{i}^{lm, l'm'}] = \b{E} [(T_{i}^{lm, l'm'})^{2}] - ( \b{E} [T_{i}^{lm, l'm'}] )^{2} \\
    &\le r_{i}(r_{i}-1)(4 r_{i}-6) \left[ \frac{\tau_{ll'}^{2} B^{4}}{\omega_{d_{r}, d}^{2}} \|\m{\Gamma} - \bar{\m{\Gamma}}_{\eta}\|_{\b{H} \otimes \b{H}}^{2} - \tau_{ll'}^{4} \langle (\m{\Gamma} - \bar{\m{\Gamma}}_{\eta}), e^{\otimes}_{lm, l'm'} \rangle_{\b{H} \otimes \b{H}}^{2} \right] \\
    &\le 4r_{i}(r_{i}-1)^{2} \frac{\tau_{ll'}^{2} B^{4}}{\omega_{d_{r}, d}^{2}} \|\m{\Gamma} - \bar{\m{\Gamma}}_{\eta}\|_{\b{H} \otimes \b{H}}^{2},
\end{align*}
which results in
\begin{align}\label{eq:cond:exp:cov:var:A}
    (A)_{lm, l'm'} &\le \left( \frac{\omega_{d_{r}, d}^{2}}{n} \right)^{2} \sum_{i=1}^{n} \frac{4r_{i}(r_{i}-1)^{2}}{(r_{i}(r_{i}-1))^{2}} \frac{\tau_{ll'}^{2} B^{4}}{\omega_{d_{r}, d}^{2}} \|\m{\Gamma} - \bar{\m{\Gamma}}_{\eta}\|_{\b{H} \otimes \b{H}}^{2} = \frac{4 \omega_{d_{r}, d}^{2} \tau_{ll'}^{2} B^{4}}{n\overline{r}} \|\m{\Gamma} - \bar{\m{\Gamma}}_{\eta}\|_{\b{H} \otimes \b{H}}^{2} \nonumber \\
    &\le \frac{4 M_{5} B^{4} \omega_{d_{r}, d}^{2}}{n\overline{r}} \eta^{2(\kappa_{2} \wedge 1)}  \tau_{ll'}^{2}. 
\end{align}

Second, to compute $(B)_{lm, l'm'}$ in \eqref{eq:step2:cov:total:var}, we first derive the conditional variance:
\begin{align}\label{eq:step2:cov:cond:var}
    &\Var \left[ \left. \langle \c{Z}^{\otimes}_{n} - \hat{\s{K}}^{\otimes} \bar{\m{\Gamma}}_{\eta}, e^{\otimes}_{lm, l'm'} \rangle_{\b{H} \otimes \b{H}} \right\vert \varphi \right]
    = \Var \left[ \left. \langle \c{Z}^{\otimes}_{n} , e^{\otimes}_{lm, l'm'} \rangle_{\b{H} \otimes \b{H}} \right\vert \varphi \right] \\
    &= \left( \frac{\omega_{d_{r}, d}^{2}}{n} \right)^{2} \sum_{i=1}^{n} \frac{1}{(r_{i}(r_{i}-1))^{2}} \Var \left[ \left. \sum_{1 \le j \neq j' \le r_{i}} \frac{1}{s_{ij}s_{ij'}} \sum_{\substack{1 \le k \le s_{ij} \\ 1 \le k' \le s_{ij'}}} \langle  Z_{iJJ'} \varphi_{iJJ'}^{\otimes} , e^{\otimes}_{lm, l'm'} \rangle_{\b{H} \otimes \b{H}} \right\vert \varphi \right] \nonumber \\
    &\le \left( \frac{\omega_{d_{r}, d}^{2}}{n} \right)^{2} \sum_{i=1}^{n} \frac{1}{(r_{i}(r_{i}-1))^{2}}  \sum_{\m{J} \in \c{J}^{i}} \frac{1}{s_{ij_{1}}s_{ij'_{1}}s_{ij_{2}}s_{ij'_{2}}} \times \\
    &\hspace*{4em} \quad \sum_{\substack{1 \le k_{1} \le s_{ij_{1}}, 1 \le k_{2} \le s_{ij_{2}} \\ 1 \le k'_{1} \le s_{ij'_{1}}, 1 \le k'_{2} \le s_{ij'_{2}}}} \b{E} \left[ \left. Z_{iJ_{1}J'_{1}} Z_{iJ_{2}J'_{2}} \right\vert \varphi  \right] \langle \varphi_{iJ_{1}J'_{1}}^{\otimes} , e^{\otimes}_{lm, l'm'} \rangle_{\b{H} \otimes \b{H}} \langle \varphi_{ij_{2}j'_{2}}^{\otimes} , e^{\otimes}_{lm, l'm'} \rangle_{\b{H} \otimes \b{H}}, \nonumber
\end{align}
where $\m{J} = (j_{1}, j'_{1}, j_{2}, j'_{2})$ denotes a tuple of four indices. Therefore, due to \cref{lem:quad:cond:exp} and the Cauchy-Schwarz inequality,
\begin{align*}
    (B)_{lm, l'm'} 
    &\le \left( \frac{\omega_{d_{r}, d}^{2}}{n} \right)^{2} \sum_{i=1}^{n} \frac{1}{(r_{i}(r_{i}-1))^{2}} [ |\c{J}^{i}_{0}| \omega_{d_{r}, d}^{-4} \langle \m{S}^{\otimes} e^{\otimes}_{lm, l'm'}, e^{\otimes}_{lm, l'm'} \rangle_{\b{H} \otimes \b{H}} \tau_{ll'}^{4} +\\
    &\hspace*{6em} \quad (|\c{J}^{i}_{1}| + |\c{J}^{i}_{2}|) \omega_{d_{r}, d}^{-2} (B^{4} \vertiii{\m{S}^{\otimes}}_{1} + 2 \sigma_{r_{i}}^{2} B^{2} \vertiii{\m{\Gamma}}_{1} + \sigma_{r_{i}}^{4}) \tau_{ll'}^{2} ] \\
    &\le \left( \frac{\omega_{d_{r}, d}^{2}}{n} \right)^{2} \sum_{i=1}^{n} \frac{1}{r_{i}(r_{i}-1)} [ (r_{i}-2)(r_{i}-3) \omega_{d_{r}, d}^{-4} \langle \m{S}^{\otimes} e^{\otimes}_{lm, l'm'}, e^{\otimes}_{lm, l'm'} \rangle_{\b{H} \otimes \b{H}} \tau_{ll'}^{4} +\\
    &\hspace*{6em} \quad (4 r_{i}-6) \omega_{d_{r}, d}^{-2} (B^{4} (\vertiii{\m{S}^{\otimes}}_{1} + \vertiii{\m{\Gamma}}_{1}^{2}) + 2 \sigma_{r_{i}}^{4}) \tau_{ll'}^{2} ] \\
    &\le \frac{1}{n} \langle \m{S}^{\otimes} e^{\otimes}_{lm, l'm'}, e^{\otimes}_{lm, l'm'} \rangle_{\b{H} \otimes \b{H}} \tau_{ll'}^{4}  + \frac{4 \omega_{d_{r}, d}^{2} B^{4} (\vertiii{\m{S}^{\otimes}}_{1} + \vertiii{\m{\Gamma}}_{1}^{2})}{n\overline{r}} \tau_{ll'}^{2} + \frac{8 \omega_{d_{r}, d}^{2} \sigma^{4}}{n \overline{r}^{\sigma^{4}}} \tau_{ll'}^{2}.
\end{align*}

Therefore, \eqref{eq:step2:cov:total:var} becomes
\begin{align*}
    &\Var [ \langle \c{Z}^{\otimes}_{n} - \hat{\s{K}}^{\otimes} \bar{\m{\Gamma}}_{\eta}, e^{\otimes}_{lm, l'm'} \rangle_{\b{H} \otimes \b{H}} ] = (A)_{lm, l'm'} + (B)_{lm, l'm'} \\
    &\le 4\omega_{d_{r}, d}^{2} \left( M_{5} B^{4} \eta^{2(\kappa_{2} \wedge 1)} + B^{4} (\vertiii{\m{S}^{\otimes}}_{1} + \vertiii{\m{\Gamma}}_{1}^{2}) + \sigma^{4} \right) \frac{\tau_{ll'}^{2}}{n \cdot \min(\overline{r},
    \overline{r}^{\sigma^{4}})} + \\
    &\hspace{6em} \langle \m{S}^{\otimes} e^{\otimes}_{lm, l'm'}, e^{\otimes}_{lm, l'm'} \rangle_{\b{H} \otimes \b{H}} \frac{\tau_{ll'}^{4}}{n}.
\end{align*}

Finally, following the same rationale to \eqref{eq:step2:conclud} in the proof of \cref{thm:gen:asym:mean}, the upper bound of \eqref{eq:step2:decomp:cov} is given by
\begin{align*}
    &\b{E} \left[ \| (\s{K}^{\otimes})^{\beta} (\tilde{\m{\Gamma}}_{\eta} - \bar{\m{\Gamma}}_{\eta}) \|_{\b{H} \otimes \b{H}}^{2} \right]
    \le \sum_{l, l' =1}^{\infty} \left( \frac{\tau_{ll'}^{2 \beta}}{\tau_{ll'}^{2} + \eta} \right)^{2} \sum_{m=1}^{N(l)} \sum_{m'=1}^{N(l')} \Var [ \langle \c{Z}^{\otimes}_{n} - \hat{\s{K}}^{\otimes} \bar{\m{\Gamma}}_{\eta}, e^{\otimes}_{lm, l'm'} \rangle_{\b{H} \otimes \b{H}} ] \\
    &\le \frac{4\omega_{d_{r}, d}^{2} \left( M_{5} B^{4} \eta^{2(\kappa_{2} \wedge 1)} + B^{4} (\vertiii{\m{S}^{\otimes}}_{1} + \vertiii{\m{\Gamma}}_{1}^{2}) + \sigma^{4} \right)}{n \cdot \min(\overline{r},
    \overline{r}^{\sigma^{4}})} \vertiii{ (\s{K}^{\otimes} + \eta I)^{-2} (\s{K}^{\otimes})^{2\beta+1} }_{1} \\
    &\hspace*{6em} \quad + \frac{M_{5}}{n} \vertiii{ (\s{K}^{\otimes} + \eta I)^{-2} (\s{K}^{\otimes})^{2\beta+2} }_{\infty} \nonumber \\
    &= O ([n \cdot \min(\overline{r},
    \overline{r}^{\sigma^{4}})]^{-1} \vertiii{ (\s{K}^{\otimes} + \eta I)^{-2} (\s{K}^{\otimes})^{2\beta+1} }_{1} + n^{-1}).
\end{align*}

\item [(Step 3)] For any fixed $\varepsilon > 0$ and any $\beta \in [0, 1/2]$ satisfying $\VERT (\s{K}^{\otimes} + \eta I)^{-2} (\s{K}^{\otimes})^{2\beta+1} \VERT_{1}= o(n \cdot \min(\overline{r}, \overline{r}^{\sigma^{4}}))$, we claim that
\begin{equation*}
    \lim_{n \rightarrow \infty} \sup_{\b{P}_{\c{Y}} \in \c{S}_{2}(\kappa_{2}, M_{3}, M_{4})} \b{P} \left[ \| (\s{K}^{\otimes})^{\beta} (\hat{\m{\Gamma}}_{\eta} - \tilde{\m{\Gamma}}_{\eta}) \|_{\b{H} \otimes \b{H}}^{2} > \varepsilon \gamma^{\otimes}_{n} \right] = 0,
\end{equation*} 
where $\gamma^{\otimes}_{n} := [n \cdot \min(\overline{r}, \overline{r}^{\sigma^{4}})]^{-1} \vertiii{ (\s{K}^{\otimes} + \eta I)^{-2} (\s{K}^{\otimes})^{2\beta+1} }_{1} + n^{-1}$.

As in the proof of \cref{thm:gen:asym:mean}, we have
\begin{align*}
    \hat{\m{\Gamma}}_{\eta} - \tilde{\m{\Gamma}}_{\eta} 
    = (\s{K}^{\otimes} + \eta I)^{-1} (\s{K}^{\otimes} - \hat{\s{K}}^{\otimes})  (\hat{\m{\Gamma}}_{\eta} - \bar{\m{\Gamma}}_{\eta}).
\end{align*}
Consider the following Parseval identity 
\begin{equation*}
    \hat{\m{\Gamma}}_{\eta} - \bar{\m{\Gamma}}_{\eta} = \sum_{l_{2}, m_{2}, l'_{2}, m'_{2}} f_{l_{2}m_{2}, l'_{2}m'_{2}} e^{\otimes}_{l_{2}m_{2}, l'_{2}m'_{2}}, \quad f_{l_{2}m_{2}, l'_{2}m'_{2}} := \langle \hat{\m{\Gamma}}_{\eta} - \bar{\m{\Gamma}}_{\eta}, e^{\otimes}_{l_{2}m_{2}, l'_{2}m'_{2}} \rangle_{\b{H} \otimes \b{H}},
\end{equation*}
to obtain
\begin{align*}
    \langle \hat{\m{\Gamma}}_{\eta} - \tilde{\m{\Gamma}}_{\eta} , e^{\otimes}_{l_{1}m_{1}, l'_{1}m'_{1}} \rangle_{\b{H} \otimes \b{H}} 
    = \frac{1}{\tau_{l_{1}l'_{1}}^{2} + \eta} \sum_{l_{2}, m_{2}, l'_{2}, m'_{2}} f_{l_{2}m_{2}, l'_{2}m'_{2}} V_{(l_{1}, m_{1}, l'_{1}, m'_{1}), (l_{2}, m_{2}, l'_{2}, m'_{2})},
\end{align*}
where $V_{(l_{1}, m_{1}, l'_{1}, m'_{1}), (l_{2}, m_{2}, l'_{2}, m'_{2})} := \tau_{l_{1}l'_{1}}^{2} \delta_{(l_{1}, m_{1}, l'_{1}, m'_{1}), (l_{2}, m_{2}, l'_{2}, m'_{2})} - \langle \hat{\s{K}}^{\otimes} e^{\otimes}_{l_{1}m_{1}, l'_{1}m'_{1}}, e^{\otimes}_{l_{2}m_{2}, l'_{2}m'_{2}} \rangle_{\b{H} \otimes \b{H}}$  are real-valued random variables with zero expectation as $\b{E}[ \hat{\s{K}}^{\otimes} ] =\s{K}^{\otimes}$. This yields, for any $\beta \ge 0$, analogous to the proof of \cref{thm:gen:asym:mean}, 
\begin{align*}
    &\| (\s{K}^{\otimes})^{\beta} (\hat{\m{\Gamma}}_{\eta} - \tilde{\m{\Gamma}}_{\eta}) \|_{\b{H} \otimes \b{H}}^{2} \\
    &\le \| (\s{K}^{\otimes})^{\zeta} (\hat{\m{\Gamma}}_{\eta} - \bar{\m{\Gamma}}_{\eta}) \|_{\b{H} \otimes \b{H}}^{2} \underbrace{\sum_{l_{1}, m_{1}, l'_{1}, m'_{1}} \left( \frac{\tau_{l_{1}l'_{1}}^{2 \beta}}{\tau_{l_{1}l'_{1}}^{2} + \eta} \right)^{2} \sum_{l_{2}, m_{2}, l'_{2}, m'_{2}} \tau_{l_{2}l'_{2}}^{-4 \zeta} V_{(l_{1}, m_{1}, l'_{1}, m'_{1}), (l_{2}, m_{2}, l'_{2}, m'_{2})}^{2}}_{=: \mathcal{X}^{\otimes}_{\zeta, \beta}}.
\end{align*}

Then, following the same rationale of (Step 3) in the proof of \cref{thm:gen:asym:mean}, it only remains to show that $\b{E}[\mathcal{X}^{\otimes}_{\zeta, \beta}] = O(a_{n}^{\otimes})$, where $a_{n}^{\otimes} := (n\overline{r})^{-1} \VERT (\s{K}^{\otimes} + \eta I)^{-2} (\s{K}^{\otimes})^{2\beta+1} \VERT_{1}$. 
To this end, note that
\begin{align*}
    &\b{E} [ V_{(l_{1}, m_{1}, l'_{1}, m'_{1}), (l_{2}, m_{2}, l'_{2}, m'_{2})}^{2} ] = \Var [V_{(l_{1}, m_{1}, l'_{1}, m'_{1}), (l_{2}, m_{2}, l'_{2}, m'_{2})} ] \\
    &= \Var [ \langle \hat{\s{K}}^{\otimes} e^{\otimes}_{l_{1}m_{1}, l'_{1}m'_{1}}, e^{\otimes}_{l_{2}m_{2}, l'_{2}m'_{2}} \rangle_{\b{H} \otimes \b{H}} ] \\
    &= \left( \frac{\omega_{d_{r}, d}^{2}}{n} \right)^{2} \sum_{i=1}^{n} \frac{1}{(r_{i}(r_{i}-1))^{2}} \times \\
    &\hspace{10mm} \Var [ \sum_{1 \le j \neq j' \le r_{i}} \underbrace{\frac{1}{s_{ij}s_{ij'}} \sum_{\substack{1 \le k \le s_{ij} \\ 1 \le k' \le s_{ij'}}} \langle \varphi_{iJJ'}^{\otimes}, e^{\otimes}_{l_{1}m_{1}, l'_{1}m'_{1}} \rangle_{\b{H} \otimes \b{H}} \langle \varphi_{iJJ'}^{\otimes}, e^{\otimes}_{l_{2}m_{2}, l'_{2}m'_{2}} \rangle_{\b{H} \otimes \b{H}}}_{=:S_{ij j'}^{(l_{1}, m_{1}, l'_{1}, m'_{1}), (l_{2}, m_{2}, l'_{2}, m'_{2})}}] \\
    &\le \left( \frac{\omega_{d_{r}, d}^{2}}{n} \right)^{2} \sum_{i=1}^{n} \frac{1}{(r_{i}(r_{i}-1))^{2}} \sum_{\m{J} \in \c{J}^{i}} \Cov \left[S_{ij_{1} j'_{1}}^{(l_{1}, m_{1}, l'_{1}, m'_{1}), (l_{2}, m_{2}, l'_{2}, m'_{2})}, S_{ij_{2} j'_{2}}^{(l_{1}, m_{1}, l'_{1}, m'_{1}), (l_{2}, m_{2}, l'_{2}, m'_{2})} \right] \\
    &= \left( \frac{\omega_{d_{r}, d}^{2}}{n} \right)^{2} \sum_{i=1}^{n} \frac{1}{(r_{i}(r_{i}-1))^{2}} \sum_{\m{J} \in \c{J}^{i}_{1} \cup \c{J}^{i}_{2}} \Cov \left[S_{ij_{1} j'_{1}}^{(l_{1}, m_{1}, l'_{1}, m'_{1}), (l_{2}, m_{2}, l'_{2}, m'_{2})}, S_{ij_{2} j'_{2}}^{(l_{1}, m_{1}, l'_{1}, m'_{1}), (l_{2}, m_{2}, l'_{2}, m'_{2})} \right] \\
    &\le \left( \frac{\omega_{d_{r}, d}^{2}}{n} \right)^{2} \sum_{i=1}^{n} \frac{r_{i}(r_{i}-1)(4 r_{i}-6)}{(r_{i}(r_{i}-1))^{2}} \b{E} \left[S_{ij_{1} j'_{1}}^{(l_{1}, m_{1}, l'_{1}, m'_{1}), (l_{2}, m_{2}, l'_{2}, m'_{2})} \right] ^{2} \\
    &\le \frac{4 \omega_{d_{r}, d}^{4}}{n\overline{r}} \b{E} \left[S_{ij_{1} j'_{1}}^{(l_{1}, m_{1}, l'_{1}, m'_{1}), (l_{2}, m_{2}, l'_{2}, m'_{2})} \right] ^{2}
\end{align*}

Using the Cauchy-Schwarz inequality with
\begin{align*}
    S_{ij j'}^{(l_{1}, m_{1}, l'_{1}, m'_{1}), (l_{2}, m_{2}, l'_{2}, m'_{2})} =S_{ij}^{(l_{1}, m_{1}), (l_{2}, m_{2})} S_{ij'}^{(l'_{1}, m'_{1}), (l'_{2}, m'_{2})},
\end{align*}
where
\begin{equation*}
    S_{ij}^{(l_{1}, m_{1}), (l_{2}, m_{2})} := \frac{1}{s_{ij}} \sum_{k=1}^{s_{ij}} \langle \varphi_{iJ}, e_{l_{1}m_{1}} \rangle_{\b{H}} \langle \varphi_{iJ}, e_{l_{2}m_{2}} \rangle_{\b{H}},
\end{equation*}
we then get
\begin{align*}
    \b{E} [ V_{(l_{1}, m_{1}, l'_{1}, m'_{1}), (l_{2}, m_{2}, l'_{2}, m'_{2})}^{2} ]
    &\le \frac{4 \omega_{d_{r}, d}^{4}}{n\overline{r}} \b{E} \left[S_{ij}^{(l_{1}, m_{1}), (l_{2}, m_{2})} \right] ^{2} \b{E} \left[S_{ij'}^{(l'_{1}, m'_{1}), (l'_{2}, m'_{2})} \right] ^{2} \\
    &\le \frac{4 \omega_{d_{r}, d}^{4}}{n\overline{r}} \b{E}[ \langle \varphi_{iJ}, e_{l_{1}m_{1}} \rangle_{\b{H}}^{2} \langle \varphi_{iJ}, e_{l_{2}m_{2}} \rangle_{\b{H}}^{2}] \b{E}[ \langle \varphi_{iJ'}, e_{l'_{1}m'_{1}} \rangle_{\b{H}}^{2} \langle \varphi_{iJ'}, e_{l'_{2}m'_{2}} \rangle_{\b{H}}^{2}].
\end{align*}

As a result, assuming (A), we obtain
\begin{align*}
    &\sup_{\b{P}_{\c{Y}} \in \c{S}_{2}(\kappa_{2}, M_{3}, M_{4})} \b{E} \mathcal{X}_{\zeta, \beta} 
    = \sum_{l_{1}, m_{1}, l'_{1}, m'_{1}} \left( \frac{\tau_{l_{1}l'_{1}}^{2 \beta}}{\tau_{l_{1}l'_{1}}^{2} + \eta} \right)^{2} \sum_{l_{2}, m_{2}, l'_{2}, m'_{2}} \tau_{l_{2}l'_{2}}^{-4 \zeta} \b{E}[ V_{(l_{1}, m_{1}, l'_{1}, m'_{1}), (l_{2}, m_{2}, l'_{2}, m'_{2})}^{2}] \\
    &\le \frac{4 \omega_{d_{r}, d}^{4}}{n\overline{r}} \sum_{l_{1}, l'_{1}, l_{2}, l'_{2}=1}^{\infty} \left( \frac{\tau_{l_{1}l'_{1}}^{2 \beta}}{\tau_{l_{1}l'_{1}}^{2} + \eta} \right)^{2} \tau_{l_{2}l'_{2}}^{-4 \zeta} \times \\
    &\hspace{5mm} \sum_{m_{1}, m_{2}} \b{E}[ \langle \varphi_{iJ}, e_{l_{1}m_{1}} \rangle_{\b{H}}^{2} \langle \varphi_{iJ}, e_{l_{2}m_{2}} \rangle_{\b{H}}^{2}] \sum_{m'_{1}, m'_{2}} \b{E}[ \langle \varphi_{iJ'}, e_{l'_{1}m'_{1}} \rangle_{\b{H}}^{2} \langle \varphi_{iJ'}, e_{l'_{2}m'_{2}} \rangle_{\b{H}}^{2}] \\
    &\stackrel{\eqref{eq:A4:sum}}{\le} \left( \frac{\tilde{B}}{\omega_{d_{r}, d}} \right)^{2} \frac{4 \omega_{d_{r}, d}^{4}}{n\overline{r}} \sum_{l_{1}, l'_{1}=1}^{\infty} \left( \frac{\tau_{l_{1}l'_{1}}^{2 \beta}}{\tau_{l_{1}l'_{1}}^{2} + \eta} \right)^{2} \tau_{l_{1}}^{2} N(l_{1}) \tau_{l'_{1}}^{2} N(l'_{1}) \sum_{l_{2}, l'_{2}=1}^{\infty} \tau_{l_{2}l'_{2}}^{-4 \zeta} \tau_{l_{2}}^{2} N(l_{2}) \tau_{l'_{2}}^{2} N(l'_{2}) \\
    &= \frac{4 \tilde{B}^{2} \omega_{d_{r}, d}^{2}}{n\overline{r}} \VERT (\s{K}^{\otimes} + \eta I)^{-2} (\s{K}^{\otimes})^{2\beta+1} \VERT_{1} \VERT (\s{K}^{\otimes})^{1-2\zeta} \VERT_{1} = O(a_{n}^{\otimes}),
\end{align*}
since $\VERT (\s{K}^{\otimes})^{1-2\zeta} \VERT_{1} = \VERT \s{K}^{1-2\zeta} \VERT_{1}^{2} < \infty$. This concludes (Step 3).
\item [(Step 4)] The final step is identical to that of \cref{thm:gen:asym:mean}.
\end{enumerate}
\end{proof}

The following corollary dictates that if the decaying speed of the penalty parameter $\eta_{n}$ is suitable for covariance estimation, then so it is for mean estimation.

\begin{corollary}\label{cor:gen:asym:cov}
Consider a sequence of the penalty parameter $\eta_{n} > 0$ converging to $0$. Assume (A) and that there is some $\zeta \in (0, 1/2]$ with $\VERT \s{K}^{1-2\zeta} \VERT_{1} < \infty$ and $\VERT (\s{K}^{\otimes} + \eta_{n} I)^{-2} (\s{K}^{\otimes})^{2\zeta+1} \VERT_{1}=o(n \cdot \min(\overline{r}, \overline{r}^{\sigma^{4}}))$. Define $\hat{\m{\Sigma}}_{\eta_{n}} := \hat{\m{\Gamma}}_{\eta_{n}} - \hat{\m{\mu}}_{\eta_{n}} \otimes \hat{\m{\mu}}_{\eta_{n}}$. Then, for any $\beta \in [0, 1/2]$ satisfying $\VERT (\s{K}^{\otimes} + \eta_{n} I)^{-2} (\s{K}^{\otimes})^{2\beta+1} \VERT_{1}=o(n \cdot \min(\overline{r}, \overline{r}^{\sigma^{4}}))$, it holds that
\begin{align*}
    \lim_{D \rightarrow \infty} \limsup_{n \rightarrow \infty} \sup_{\b{P}_{\c{Y}} \in \c{S}_{2}(\kappa_{2}, M_{3}, M_{4})} \b{P} \left[ \| (\s{K}^{\otimes})^{\beta} (\hat{\m{\Sigma}}_{\eta_{n}}-\m{\Sigma}) \|_{\b{H} \otimes \b{H}}^{2} > D \max(c_{n}^{\otimes}, a_{n}^{\otimes}, n^{-1}) \right] = 0,
\end{align*}
where
\begin{align*}
    c^{\otimes}_{n} := \eta_{n}^{2[(\beta + \kappa_{2}) \wedge 1]}, \quad
    a^{\otimes}_{n} := [n \cdot \min(\overline{r}, \overline{r}^{\sigma^{4}})]^{-1} \vertiii{ (\s{K}^{\otimes} + \eta_{n} I)^{-2} (\s{K}^{\otimes})^{2\beta+1} }_{1}.
\end{align*}    
\end{corollary}
\begin{proof}
In the proof, we may assume that $0 < \eta_{n} < 1$ for any $n \in \b{N}$ as $\eta_{n} > 0$ converges to $0$. We remark that $\c{S}_{2}(\kappa_{2}, M_{3}, M_{4}) \subset \c{S}_{1}(\kappa_{2}, \sqrt{M_{3}}, \sqrt{M_{4}})$ due to \cref{prop:srce:cond}. We claim that
\begin{align*}
    \lim_{D \rightarrow \infty} \limsup_{n \rightarrow \infty} \sup_{\b{P}_{\c{Y}} \in \c{S}_{1}(\kappa_{2}, \sqrt{M_{3}}, \sqrt{M_{4}})} \b{P}_{\c{Y}} \left[ \| \s{K}^{\beta} (\hat{\m{\mu}}_{\eta_{n}}-\m{\mu}) \|_{\b{H}}^{2} > D \max(c_{n}, a_{n}, n^{-1}) \right] = 0,
\end{align*}
where $c_{n} := \eta_{n}^{2[(\beta + \kappa_{2}) \wedge 1]}$ and $a_{n} := [n \cdot \min(\overline{r}, \overline{r}^{\sigma^{2}})]^{-1} \VERT (\s{K} + \eta_{n} I)^{-2} \s{K}^{2\beta+1} \VERT_{1}$.

For any $\beta \ge 0$ and $\eta > 0$, note that
\begin{align*}
    \vertiii{ (\s{K}^{\otimes} + \eta I)^{-2} (\s{K}^{\otimes})^{2\beta+1} }_{1} 
    &\le \vertiii{ (\s{K}^{\otimes} + \eta I)^{-1} (\s{K}^{\otimes})^{\beta+1} }_{\infty}^{2} \vertiii{\s{K}^{\otimes}}_{1} \\
    &\le \vertiii{\s{K}^{\otimes}}_{\infty}^{\beta} \vertiii{\s{K}^{\otimes}}_{1} = \vertiii{\s{K}}_{\infty}^{2\beta} \vertiii{\s{K}}_{2}< \infty,
\end{align*}
is well-defined due to \cref{lem:srce:oper:norm}. On the other hand, for any $0 < \eta < 1$, we have $(\tau_{l}^{2} \tau_{l'}^{2} + \eta) \le (\tau_{l}^{2} + \eta) (\tau_{l'}^{2} + \eta)$, hence
\begin{align}\label{eq:tens:trace:sq}
    \VERT (\s{K}^{\otimes} + \eta I)^{-2} (\s{K}^{\otimes})^{2\beta+1} \VERT_{1} &= \sum_{l, l'=1}^{\infty} N(l) N(l') \frac{\tau_{l}^{4 \beta + 2} \tau_{l'}^{4 \beta + 2}}{(\tau_{l}^{2} \tau_{l'}^{2} + \eta)^{2}} \nonumber \\
    &\ge \left( \sum_{l=1}^{\infty} N(l) \frac{\tau_{l}^{4 \beta + 2}}{(\tau_{l}^{2} + \eta)^{2}} \right)^{2} = \VERT (\s{K} + \eta I)^{-2} \s{K}^{2\beta+1} \VERT_{1}^{2}.
\end{align}
Then applying $\beta_{n}, \eta_{n}$, instead of $\beta$, to \eqref{eq:tens:trace:sq} demonstrates that $\eta_{n}$ satisfies the conditions for not only \cref{thm:gen:asym:cov} but also \cref{thm:gen:asym:mean}, thus the claim holds. Moreover, we also get $c_{n}^{2} = O(c^{\otimes}_{n})$ and $a_{n}^{2} = O(a^{\otimes}_{n})$.

Observe using the triangle inequality that
\begin{align*}
    \| (\s{K}^{\otimes})^{\beta} (\hat{\m{\Sigma}}_{\eta_{n}}-\m{\Sigma}) \|_{\b{H} \otimes \b{H}} 
    &\le \| (\s{K}^{\otimes})^{\beta} (\hat{\m{\Gamma}}_{\eta_{n}}-\m{\Gamma}) \|_{\b{H} \otimes \b{H}} + \| (\s{K}^{\otimes})^{\beta} (\hat{\m{\mu}}_{\eta_{n}} \otimes \hat{\m{\mu}}_{\eta_{n}} -\m{\mu} \otimes \m{\mu}) \|_{\b{H} \otimes \b{H}},
\end{align*}
which leads to
\begin{align*}
    &\b{P} \left[ \| (\s{K}^{\otimes})^{\beta} (\hat{\m{\Sigma}}_{\eta_{n}}-\m{\Sigma}) \|_{\b{H} \otimes \b{H}}^{2} > D \max(c_{n}^{\otimes}, a_{n}^{\otimes}, n^{-1}) \right]  \\
    &\le \b{P} \left[ \| (\s{K}^{\otimes})^{\beta} (\hat{\m{\Gamma}}_{\eta_{n}}-\m{\Gamma}) \|_{\b{H} \otimes \b{H}}^{2} > D \max(c_{n}^{\otimes}, a_{n}^{\otimes}, n^{-1}) \right] \\
    &\hspace{5mm} + \b{P} \left[ \| (\s{K}^{\otimes})^{\beta} (\hat{\m{\mu}}_{\eta_{n}} \otimes \hat{\m{\mu}}_{\eta_{n}} -\m{\mu} \otimes \m{\mu}) \|_{\b{H} \otimes \b{H}}^{2} > D \max(c_{n}^{\otimes}, a_{n}^{\otimes}, n^{-1}) \right] \\
    &\le \b{P} \left[ \| (\s{K}^{\otimes})^{\beta} (\hat{\m{\Gamma}}_{\eta_{n}}-\m{\Gamma}) \|_{\b{H} \otimes \b{H}}^{2} > D \max(c_{n}^{\otimes}, a_{n}^{\otimes}, n^{-1}) \right] \\
    &\hspace{5mm} + \b{P} \left[ \| (\s{K}^{\otimes})^{\beta} (\hat{\m{\mu}}_{\eta_{n}} \otimes \hat{\m{\mu}}_{\eta_{n}} -\m{\mu} \otimes \m{\mu}) \|_{\b{H} \otimes \b{H}} > D \max(c_{n}, a_{n}, n^{-1}) \right].
\end{align*}
Finally,
\begin{align*}
    \| (\s{K}^{\otimes})^{\beta} (\hat{\m{\mu}}_{\eta_{n}} \otimes \hat{\m{\mu}}_{\eta_{n}} -\m{\mu} \otimes \m{\mu}) \|_{\b{H} \otimes \b{H}} &\le \| \s{K}^{\beta} (\hat{\m{\mu}}_{\eta_{n}} -\m{\mu}) \|_{\b{H}}^{2} + 2 \| \s{K}^{\beta} (\hat{\m{\mu}}_{\eta_{n}} -\m{\mu}) \|_{\b{H}} \| \s{K}^{\beta} \m{\mu} \|_{\b{H}} \\
    &= O_{\b{P}} (\| \s{K}^{\beta} (\hat{\m{\mu}}_{\eta_{n}} -\m{\mu}) \|_{\b{H}})
\end{align*}   
and the claim in the beginning of the proof yields the result.
\end{proof}

Again, \cref{thm:gen:asym:cov} and \cref{cor:gen:asym:cov} reveal the trade-off of the convergence rate concerning the decaying speed of the penalty parameter $\eta_{n} > 0$. When $\eta_{n}$ decreases faster, $c_{n}^{\otimes}$ decreases faster but $a_{n}^{\otimes}$ decreases more slowly. Specifically, when $\tau_{l} \asymp l^{-p}$ and $N(l) \asymp l^{q}$ for some $p, q >0$ with $2p > q+1$, it holds for any $\beta \ge 0$ that
\begin{align}\label{eq:tr:norm:rate:cov}
    \VERT (\s{K}^{\otimes} + \eta I)^{-2} (\s{K}^{\otimes})^{2\beta+1} \VERT_{1} &= \sum_{l, l'=1}^{\infty} N(l) N(l') \frac{\tau_{ll'}^{4 \beta + 2}}{(\tau_{ll'}^{2} + \eta)^{2}} = \sum_{l, l'=1}^{\infty}  \frac{N(l) N(l') \tau_{ll'}^{-2(1 - 2\beta)}}{(1 + \eta \tau_{ll'}^{-2})^{2}} \nonumber \\
    &\asymp \sum_{l, l'=1}^{\infty} (l^{2p} l'^{2p})^{\frac{q+2p(1 - 2\beta)}{2p}} (1+ \eta l^{2p} l'^{2p})^{-2}
    \asymp \eta^{-(\frac{2p+q+1}{2p}-2\beta)} \log \left(\frac{1}{\eta} \right) ,
\end{align}
where we used Lemma 2.3 in \cite{lin2000tensor}. Using \eqref{eq:tr:norm:rate:cov}, one may show as in \cite{caponera2022functional} that
\begin{equation}\label{eq:opt:pen:rate:cov}
    a_{n}^{\otimes} = O(\max(c_{n}^{\otimes}, n^{-1})) \quad \Longleftrightarrow \quad  \eta_{n}^{-(\frac{2p+q+1}{2p}-2\beta)} \log \left(\frac{1}{\eta_{n}} \right) = O\left(\max \left[ \overline{r}, (n \cdot \min(\overline{r}, \overline{r}^{\sigma^{4}})) \eta_{n}^{2(\beta + \kappa_{2})} \right] \right),
\end{equation}
is satisfied regardless of the value of $\beta \ge 0$ if we choose
\begin{equation*}
    \eta_{n} \asymp \left( \frac{n \cdot \min(\overline{r}, \overline{r}^{\sigma^{4}})}{\log n} \right)^{-(\frac{2p+q+1}{2p} + 2\kappa_{2})^{-1}}.
\end{equation*}

\begin{proof}[Proof of \cref{thm:cov:asym:L2}]
Take $\beta=1/2$, which leads to 
\begin{equation*}
    \| (\s{K}^{\otimes})^{\beta} (\hat{\m{\Gamma}}_{\eta_{n}}-\m{\Gamma}) \|_{\b{H} \otimes \b{H}}^{2} = \| (\s{P} \otimes \s{P}) (\hat{\m{\Gamma}}_{\eta_{n}}-\m{\Gamma}) \|_{\c{L}_{2} (\b{G} \times \Omega^{d-d_{r}}, W_{d_{r}, d})}^{2},
\end{equation*}
by \eqref{eq:isomet}. Also, choose $\kappa_{2}=0$ and $\eta_{n} \asymp (n \overline{r}/ \log n)^{-(\frac{2p+q+1}{2p})^{-1}}$ so that
\begin{equation*}
    \max(c_{n}^{\otimes}, a_{n}^{\otimes}, n^{-1}) \asymp \max(c_{n}^{\otimes}, n^{-1}) = \max(\eta_{n}, n^{-1}),
\end{equation*}
in view of \eqref{eq:opt:pen:rate:cov}. It remains to check whether the assumptions other than (A) in \cref{thm:gen:asym:cov}, or equivalently \cref{cor:gen:asym:cov}, hold. First, we verify
\begin{align*}
    \VERT (\s{K}^{\otimes} + \eta_{n} I)^{-2} (\s{K}^{\otimes})^{2\beta+1} \VERT_{1} &\stackrel{\eqref{eq:tr:norm:rate:cov}}{\asymp} \eta_{n}^{-(\frac{2p+q+1}{2p}-1)} \log \left(\frac{1}{\eta_{n}} \right) \\
    &\asymp \left( \frac{n \cdot \min(\overline{r}, \overline{r}^{\sigma^{4}})}{\log n} \right)^{(\frac{2p+q+1}{2p}-1)(\frac{2p+q+1}{2p})^{-1}} \log \left( \frac{n \cdot \min(\overline{r}, \overline{r}^{\sigma^{4}})}{\log n} \right) \\
    &= O \left( \left( n \cdot \min(\overline{r}, \overline{r}^{\sigma^{4}}) \right)^{(\frac{2p+q+1}{2p}-1)(\frac{2p+q+1}{2p})^{-1}} \log \left( n \overline{r} \right) \right) \\
    &= o(n \cdot \min(\overline{r}, \overline{r}^{\sigma^{4}})).
\end{align*}
Second, choose $\theta \in (\frac{q+1}{2p},1) \neq \emptyset$ and set $\zeta := (1-\theta)/2 \in (0, 1/2]$. Then, $-2\theta p+q < -1$ so $\VERT \s{K}^{1-2\zeta} \VERT_{1} < \infty$, and
\begin{align*}
    \VERT (\s{K}^{\otimes} + \eta I)^{-2} (\s{K}^{\otimes})^{2\zeta+1} \VERT_{1} &\stackrel{\eqref{eq:tr:norm:rate:cov}}{\asymp} \eta_{n}^{-(\frac{2p+q+1}{2p}-2\zeta)} \log \left(\frac{1}{\eta_{n}} \right) \\
    &\asymp \left( \frac{n \cdot \min(\overline{r}, \overline{r}^{\sigma^{4}})}{\log n} \right)^{(\frac{q+1}{2p}+\theta)(\frac{q+1}{2p}+1)^{-1}} \log \left( \frac{n \cdot \min(\overline{r}, \overline{r}^{\sigma^{4}})}{\log n} \right) \\
    &= o(n \cdot \min(\overline{r}, \overline{r}^{\sigma^{4}})).
\end{align*}
\end{proof}

\begin{proof}[Proof of \cref{thm:cov:asym:RKHS}]
Take $\beta=0$, $\kappa_{2} > 0$, and choose $\eta_{n} \asymp (n \cdot \min(\overline{r}, \overline{r}^{\sigma^{4}})/ \log n)^{-(\frac{2p+q+1}{2p} + 2\kappa_{2})^{-1}}$ so that
\begin{equation*}
    \max(c_{n}^{\otimes}, a_{n}^{\otimes}, n^{-1}) \asymp \max(c_{n}^{\otimes}, n^{-1}) = \max(\eta_{n}^{2\kappa_{2}}, n^{-1}),
\end{equation*}
in view of \eqref{eq:opt:pen:rate:cov}. It remains to check whether the assumptions other than (A) in \cref{thm:gen:asym:cov}, or equivalently \cref{cor:gen:asym:cov}, hold. First, we verify
\begin{align*}
    \VERT (\s{K}^{\otimes} + \eta_{n} I)^{-2} (\s{K}^{\otimes})^{2\beta+1} \VERT_{1} &\stackrel{\eqref{eq:tr:norm:rate:cov}}{\asymp} \eta_{n}^{-\frac{2p+q+1}{2p}} \log \left(\frac{1}{\eta_{n}} \right) \\
    &\asymp \left( \frac{n \cdot \min(\overline{r}, \overline{r}^{\sigma^{4}})}{\log n} \right)^{\frac{2p+q+1}{2p}(\frac{2p+q+1}{2p} + 2\kappa_{2})^{-1}} \log \left( \frac{n \cdot \min(\overline{r}, \overline{r}^{\sigma^{4}})}{\log n} \right) \\
    &= O \left( \left( n \cdot \min(\overline{r}, \overline{r}^{\sigma^{4}}) \right)^{\frac{2p+q+1}{2p}(\frac{2p+q+1}{2p} + 2\kappa_{2})^{-1}} \log \left( n \cdot \min(\overline{r}, \overline{r}^{\sigma^{4}}) \right) \right) \\
    &= o(n \cdot \min(\overline{r}, \overline{r}^{\sigma^{4}})).
\end{align*}
Second, choose $\theta \in (\frac{q+1}{2p},1) \neq \emptyset$ and set $\zeta := (1-\theta)/2 \in (0, 1/2]$. Then, $-2\theta p+q < -1$ so $\VERT \s{K}^{1-2\zeta} \VERT_{1} < \infty$, and
\begin{align*}
    \VERT (\s{K}^{\otimes} + \eta_{n} I)^{-2} (\s{K}^{\otimes})^{2\zeta+1} \VERT_{1} &\stackrel{\eqref{eq:tr:norm:rate:cov}}{\asymp} \eta_{n}^{-(\frac{2p+q+1}{2p}-2\zeta)} \log \left(\frac{1}{\eta_{n}} \right) \\
    &\asymp \left( \frac{n \cdot \min(\overline{r}, \overline{r}^{\sigma^{4}})}{\log n} \right)^{(\frac{q+1}{2p}+ 1+ 2\kappa_{1})^{-1}(\frac{q+1}{2p}+\theta)} \log \left( \frac{n \cdot \min(\overline{r}, \overline{r}^{\sigma^{4}})}{\log n} \right) \\
    &= O \left( \left( n \cdot \min(\overline{r}, \overline{r}^{\sigma^{4}}) \right)^{(\frac{q+1}{2p}+ 1+ 2\kappa_{1})^{-1}(\frac{q+1}{2p}+\theta)} \log \left( n \cdot \min(\overline{r}, \overline{r}^{\sigma^{4}}) \right) \right) \\
    &= o(n \cdot \min(\overline{r}, \overline{r}^{\sigma^{4}})).
\end{align*}
\end{proof}

\subsection{Proofs in Section \ref{sec:cryo}}\label{sec:proof:cryo}
\begin{proof}[Proof of \cref{prop:cryo:cond:exp}]
Using the tower property, the conditional expectation is given by
\begin{align*}
    \b{E}[Z_{ijk} \vert (\tilde{\m{R}}_{ij}, \m{X}_{ijk}) ] &= \b{E}_{\m{R}_{i}}[Z_{ijk} \vert (\m{R}_{i}, \tilde{\m{R}}_{ij}, \m{X}_{ijk}) ] = \b{E}_{\m{R}_{i}} [\langle \m{\mu}, \varphi(\tilde{\m{R}}_{ij} \m{R}_{i}, \m{X}_{ijk}) \rangle_{\b{H}}] \\
    &= \langle \m{\mu}, \overline{\varphi}(\tilde{\m{R}}_{ij}, \m{X}_{ijk}) \rangle_{\b{H}} = \langle \m{\mu}, \overline{\varphi}(I, \m{X}_{ijk}) \rangle_{\b{H}} = \langle \m{\mu}, \overline{\varphi}_{ijk} \rangle_{\b{H}}.
\end{align*}
For $i \neq i'$, the conditional second moment is given by
\begin{align*}
    \b{E}[Z_{ijk} Z_{i'j'k'} \vert (\tilde{\m{R}}_{ij}, \m{X}_{ijk}) , (\tilde{\m{R}}_{i'j'}, \m{X}_{i'j'k'}) ] &= \b{E}[Z_{ijk} \vert (\tilde{\m{R}}_{ij}, \m{X}_{ijk})]  \b{E}[Z_{i'j'k'} \vert (\tilde{\m{R}}_{i'j'}, \m{X}_{i'j'k'}) ] \\
    &= \langle \m{\mu}, \overline{\varphi}_{ijk} \rangle_{\b{H}} \langle \m{\mu}, \overline{\varphi}_{i'j'k'} \rangle_{\b{H}}.
\end{align*}
On the other hand, when $i=i'$, using \eqref{eq:WW:cond:exp}, we get
\begin{align*}
    &\b{E}[Z_{ijk} Z_{ij'k'} \vert (\tilde{\m{R}}_{ij}, \m{X}_{ijk}) , (\tilde{\m{R}}_{ij'}, \m{X}_{ij'k'}) ] \\
    &= \b{E}_{\m{R}_{i}} [Z_{ijk} Z_{ij'k'} \vert (\m{R}_{i}, \tilde{\m{R}}_{ij}, \m{X}_{ijk}) , (\m{R}_{i}, \tilde{\m{R}}_{ij'}, \m{X}_{ij'k'}) ] \\
    &= \b{E}_{\m{R}_{i}} [\langle \m{\Gamma}, \varphi(\tilde{\m{R}}_{ij} \m{R}_{i}, \m{X}_{ijk}) \otimes \varphi(\tilde{\m{R}}_{ij'} \m{R}_{i}, \m{X}_{ijk}) \rangle_{\b{H} \otimes \b{H}}] + \sigma_{r_{i}}^{2} \delta_{jj'} \delta_{kk'}.
\end{align*}
Note that
\begin{align*}
    \b{E}_{\m{R}_{i}} [\varphi(\tilde{\m{R}}_{ij} \m{R}_{i}, \m{X}_{ijk}) \otimes \varphi(\tilde{\m{R}}_{ij'} \m{R}_{i}, \m{X}_{ijk})] = \overline{\varphi \otimes \varphi}(\tilde{\m{R}}_{ij}, \m{X}_{ijk}, \tilde{\m{R}}_{ij'}, \m{X}_{ij'k'}) = \overline{\varphi}_{iJJ'}^{\otimes},
\end{align*}
which yields \eqref{eq:cryo:cond:2nd}. Finally, due to the isotropy of $\overline{\varphi}_{ijk}$, it holds that $\overline{\varphi}_{ijk} = \rho(R) \overline{\varphi}_{ijk}$ for any $R \in \b{G}$, hence
\begin{align*}
    \langle \overline{\m{\mu}}, \overline{\varphi}_{ijk} \rangle_{\b{H}} = \b{E}_{\m{R}} [ \langle \rho(\m{R}) \m{\mu}, \overline{\varphi}_{ijk} \rangle_{\b{H}} ] = \b{E}_{\m{R}} [ \langle \m{\mu}, \rho(\m{R}^{-1}) \overline{\varphi}_{ijk} \rangle_{\b{H}} ] = \langle \m{\mu}, \overline{\varphi}_{ijk} \rangle_{\b{H}},
\end{align*}
using the invariance of the Haar measure. Similarly,
\begin{align*}
    \langle \overline{\m{\Gamma}}, \overline{\varphi}_{iJJ'}^{\otimes} \rangle_{\b{H} \otimes \b{H}} &= \b{E}_{\m{R}} [\langle \rho(\m{R}) \m{\Gamma} \rho(\m{R}^{-1}), \overline{\varphi}_{iJJ'}^{\otimes} \rangle_{\b{H} \otimes \b{H}}] \\
    &= \b{E}_{\m{R}} [\langle  \m{\Gamma} , \rho(\m{R}^{-1}) \overline{\varphi}_{iJJ'}^{\otimes} \rho(\m{R}) \rangle_{\b{H} \otimes \b{H}}]
    = \langle \m{\Gamma}, \overline{\varphi}_{iJJ'}^{\otimes} \rangle_{\b{H} \otimes \b{H}}.
\end{align*}
\end{proof}

\section{Index}\label{sec:list:note}
\begin{itemize}
    \item $\Omega^{d}, \Omega^{d-d_{r}}$: Compact domains of input and output functions.
    \item $W_{0, d}, W_{d_{r}, d}$: Weight functions for input and output spaces.
    \item $w_{0, d}, w_{d_{r}, d}$: Total weights (normalization constants).
    \item $\s{P}$: Sinogram or forward operator.
    \item $\b{G}$: Compact Lie group representing orientations.
    \item $\rho$: Unitary group representation acting on $\b{G}$.
    \item $K$: Mercer kernel on the input space.
    \item $\tilde{K}$: Induced kernel on the output (observation) space.
    \item $\varphi$: Kernel feature map.
    \item $\c{Y}_{i}$: Random function in the RKHS $\b{H}(K)$.
    \item $\boldsymbol{\mu}, \boldsymbol{\Gamma}, \boldsymbol{\Sigma}$: Mean, second-moment, and covariance functions.
    \item $\m{R}_{ij}$: Orientation (viewing direction) for the $j$-th observation of $i$-th function.
    \item $\m{X}_{ijk}$: Spatial location (e.g., detector position) for sample $(i,j,k)$.
    \item $\varepsilon_{ijk}$: Additive noise at observation $(i,j,k)$ with heteroscadestic variance $\sigma_{r_{i}}^{2}$.
    \item $Z_{ijk}$: Observed scalar value at sample $(i,j,k)$.
    \item $\m{\Phi}, \m{\Phi}^{\odot}$: Mean and covariance Gram matrices.
    \item $\nu, \eta$: Tikhonov regularization parameters for mean and second-moment estimation.
    \item $\hat{\m{\alpha}}_{\nu}, \hat{\m{\alpha}}_{\eta}^{\odot}$: Coefficient vectors for the representer theorem solutions.
    \item $\s{T}_{K}$: Integral operator associated with kernel $K$.
    \item $\imath$: Continuous embedding map of RKHS into $\c{L}_{2}$.
    \item $g_{lm}, e_{lm}, \tau_{l}^{2}$: Eigenfunctions and eigenvalues of $\s{T}_{\tilde{K}}$ and $\s{K}$.
    \item $p, q$: Exponents governing eigenvalue decay and multiplicity growth.
    \item $\kappa, \delta, \beta, \zeta$: Source condition parameters for convergence analysis.
    \item $\c{S}$: Class of probability measures satisfying source conditions.
    \item $\gamma$: Bandwidth parameter of the Gaussian kernel.
    \item $\c{J}, \s{J}$: Off-diagonal index set and associated elimination tensor.
    \item $i, j, k, l, m$: Index variables for function, orientation, location, eigenvalue, and multiplicity.
    \item $n, r, s, \b{N}, N(l)$: Cardinalities for indices $i, j, k, l, m$.
    \item $G$: Green's function.
    \item $C_{l}^{\lambda}$: Gegenbauer polynomials.
    \item $\Gamma$: Gamma function.
    \item $\delta_{jj'}$: Kronecker delta.
\end{itemize}

\end{document}